\documentclass[11pt]{amsart}
\usepackage{amssymb,amsmath, amsthm, amsfonts}

\usepackage{hyperref}
\usepackage{cite}
\hypersetup{linktocpage,colorlinks, linkcolor = {blue}, citecolor={magenta}}

\usepackage[normalem]{ulem}

\usepackage[usenames,dvipsnames]{xcolor}
\usepackage{graphicx}
\usepackage{listings}

\usepackage[margin=1in]{geometry}
\usepackage{lstautogobble}
\usepackage{enumerate}
\usepackage[shortlabels]{enumitem}
\usepackage{thmtools}
\usepackage{thm-restate}
\usepackage{amsthm}
\usepackage{verbatim}
\usepackage{accents}
\usepackage{mathtools}
\usepackage{physics}

\usepackage{hyperref}
\usepackage[capitalise]{cleveref}
\usepackage[bottom]{footmisc}
\crefformat{equation}{(#2#1#3)}

\usepackage{float}
\restylefloat{table}

\usepackage{mathrsfs}
\setlist{  
  listparindent=\parindent,
  parsep=0pt,
}

\theoremstyle{plain}
\newtheorem{thm}{Theorem}[section]
\newtheorem{prop}[thm]{Proposition}
\newtheorem{lemma}[thm]{Lemma}
\newtheorem{cor}[thm]{Corollary}

\theoremstyle{definition}

\newtheorem{remark}[thm]{Remark}

\Crefname{thm}{Theorem}{Theorems}
\Crefname{prop}{Proposition}{Propositions}

\numberwithin{equation}{section} 

\DeclarePairedDelimiter{\pa}{\lparen}{\rparen}

\DeclarePairedDelimiter{\jp}{\langle}{\rangle}

\DeclareMathOperator{\supp}{supp}

\DeclareMathOperator{\dist}{dist}

\renewcommand{\det}{\mathrm{det}}

\newcommand{\M}{{\mathcal{M}}}

\newcommand{\p}{{\partial}}

\newcommand{\cre}{\color{red}}

\renewcommand{\d}{\mathsf{d}}

\newcommand{\R}{{\mathbb{R}}}

\newcommand{\N}{{\mathbb{N}}}

\newcommand{\Ss}{{\mathbb{S}}}

\newcommand{\T}{{\mathbb{T}}}
\newcommand{\g}{{\mathsf{g}}}
\newcommand{\G}{{\mathsf{g}}}

\newcommand{\Sc}{{\mathcal{S}}}

\renewcommand{\M}{{\mathbb{M}}}
\newcommand{\I}{\mathbb{I}}

\renewcommand{\k}{\mathsf{k}}
\newcommand{\ga}{\gamma}
\newcommand{\nab}{\nabla}

\newcommand{\tl}{\tilde}

\newcommand{\ph}{\phantom{=}}
\newcommand{\nn}{\nonumber}

\newcommand{\ux}{X}

\newcommand{\ep}{\epsilon}
\newcommand{\vep}{\varepsilon}
\newcommand{\al}{\alpha}
\newcommand{\be}{\beta}
\newcommand{\ka}{\kappa}
\newcommand{\la}{\lambda}
\newcommand{\Om}{\Omega}

\newcommand{\indic}{\mathbf{1}}
\newcommand{\f}{\mathsf{f}}

\newcommand{\Fc}{\mathcal{F}}

\newcommand{\E}{{\mathbb{E}}}

\newcommand{\rs}{\mathsf{r}}
\newcommand{\cd}{\mathsf{c}_{\mathsf{d},\mathsf{s}}}

\newcommand{\s}{\mathsf{s}}
\renewcommand{\k}{\mathsf{k}}
\newcommand{\tv}{\tl{v}}

\newcommand{\zg}{|z|^\ga}

\newcommand{\Fr}{\mathsf{F}}
\renewcommand{\P}{\mathcal{P}}
\newcommand{\Te}{\mathrm{Term}}

\let\div\relax
\DeclareMathOperator{\div}{\mathrm{div}}

\def\XXint#1#2#3{{\setbox0=\hbox{$#1{#2#3}{\int}$ }
\vcenter{\hbox{$#2#3$ }}\kern-.6\wd0}}

\let\oldtocsection=\tocsection
 
\let\oldtocsubsection=\tocsubsection
 
\let\oldtocsubsubsection=\tocsubsubsection
 
\renewcommand{\tocsection}[2]{\hspace{0em}\oldtocsection{#1}{#2}}
\renewcommand{\tocsubsection}[2]{\hspace{1em}\oldtocsubsection{#1}{#2}}
\renewcommand{\tocsubsubsection}[2]{\hspace{2em}\oldtocsubsubsection{#1}{#2}}

\title[Sharp commutator estimates for Coulomb and Riesz modulated energies]{Sharp commutator estimates of all order for  Coulomb and Riesz modulated energies}

\author[M. Rosenzweig]{Matthew Rosenzweig}
\address{Matthew Rosenzweig, Carnegie Mellon University, Department of Mathematical Sciences, Pittsburgh, PA} 
\email{mrosenz2@andrew.cmu.edu}
\thanks{M.R. was supported by the Simons Foundation through the Simons Collaboration on Wave Turbulence and by NSF grants DMS-2052651, DMS-2206085, DMS-2345533.}
\author[S. Serfaty]{Sylvia Serfaty}
\address{Sylvia Serfaty, Courant Institute of Mathematical Sciences, New York University, New York City, NY}
\email{serfaty@cims.nyu.edu}
\thanks{S.S. was supported by NSF grant DMS-2247846 and by the Simons Foundation through the Simons Investigator program.}

\begin{document}
\begin{abstract}
We prove functional inequalities in any dimension controlling the iterated derivatives along a transport of the Coulomb or super-Coulomb Riesz modulated energy in terms of the modulated energy itself. This modulated energy was introduced by the second author and collaborators in the study of mean-field limits and statistical mechanics of Coulomb/Riesz gases, where control of such derivatives by the energy itself is an essential ingredient. In this paper, we extend and improve such  functional inequalities, proving estimates which are now sharp in their additive error term, in their density dependence,  valid at arbitrary order of differentiation, and localizable to the support of the transport.  Our method relies on the observation that these iterated derivatives are the quadratic form of a commutator. Taking advantage of the Riesz nature of the interaction, we identify these commutators as solutions to a degenerate elliptic equation with a right-hand side exhibiting a recursive structure in terms of lower-order commutators and develop a local regularity theory for the commutators, which may be of independent interest. 

These estimates have applications to obtaining sharp rates of convergence for mean-field limits, quasi-neutral limits, and in proving central limit theorems for the fluctuations of Coulomb/Riesz gases. In particular, we show here the expected  $N^{\frac{\s}{\d}-1}$-rate in the modulated energy distance for the  mean-field convergence of first-order Hamiltonian and gradient flows.


\end{abstract}
\maketitle

\section{Introduction}\label{sec:intro}
\subsection{Motivation}\label{ssec:introMot}
In any dimension $\d$,  consider the class of  \emph{Riesz} interactions 
\begin{equation}\label{eq:gmod}
\g(x)= \begin{cases}  \frac{1}{\s} |x|^{-\s}, \quad  & \s \neq 0\\
-\log |x|, \quad & \s=0,
\end{cases}
\end{equation}
with the assumption that  $\d-2\le \s<\d$. Up to a normalizing constant $\cd$, these interactions are characterized as fundamental solutions of the fractional Laplacian: $(-\Delta)^{\frac{\d-\s}{2}}\g = \cd \delta_0$. The particular case $\s=\d-2$  corresponds to the classical \emph{Coulomb} interaction from physics, and thus $\s \ge \d-2$ means that we are considering the \emph{super-Coulomb} case. 

When studying systems of $N$ distinct points $\ux_N = (x_1, \dots, x_N)\in (\R^\d)^N$  with interaction energy  
\begin{equation}
\sum_{1\le i\neq j\le N} \g(x_i-x_j),
\end{equation}
one is led to comparing the sequence of empirical measures $\mu_N \coloneqq \frac1N \sum_{i=1}^N \delta_{x_i}$ to an  average, or \emph{mean-field}, density $\mu$. This comparison is conveniently performed by considering a {\it modulated energy}, or Coulomb/Riesz (squared) ``distance" between $\mu_N$ and $\mu$, defined by 
\begin{equation}\label{eq:modenergy}
\Fr_N(\ux_N, \mu) \coloneqq \frac12\int_{(\R^\d)^2 \setminus \triangle} \g(x-y) d\Big(\frac{1}{N} \sum_{i=1}^N \delta_{x_i} - \mu\Big)(x)d\Big(\frac{1}{N} \sum_{i=1}^N \delta_{x_i} - \mu\Big)(y),
\end{equation}
where we excise the diagonal $\triangle$ in order to remove the infinite self-interaction of each particle.  This object first appeared in the study of the statistical mechanics of  Coulomb and Riesz gases in the works \cite{SS2015log, SS2015,RS2016, PS2017} and in the context of the derivation of mean-field dynamics in \cite{Duerinckx2016,Serfaty2020}, which has been extended in \cite{NRS2021}. We refer to the forthcoming \cite[Chapter 4]{SerfatyLN} for a description of what is known on the modulated energy.
 
In both of the aforementioned contexts, an essential point is to control quantities that correspond to differentiating $\Fr_N$ along  a transport $v$.  More precisely, given a vector field $v:\R^\d\rightarrow\R^\d$, these are  the quantities
\begin{equation}\label{13}
\frac{d}{dt}\Big|_{t=0} \Fr_N\pa*{ (\I + tv)^{{ \oplus N}} (\ux_N), (\I + tv)\# \mu},
\end{equation}
where $\I:\R^\d\rightarrow\R^\d$ is the identity and by $(\I+t v)^{\oplus N} (\ux_N)$, we mean the configuration of points $(x_1 + tv(x_1), \ldots, x_N+ tv(x_N))$. It is straightforward to compute that the derivative in \eqref{13} is equal to 
\begin{equation} \label{14}
\frac12\int_{(\R^\d)^2\backslash \triangle} (v(x)-v(y)) \cdot \nabla \g(x-y) d ( \mu_N- \mu)^{\otimes 2} (x,y),
\end{equation}
and in the same way 
\begin{equation}\label{15}
 \frac{d^n}{dt^n}\Big|_{t=0} \Fr_N( (\I + tv)^{\oplus N} (\ux_N), (\I + tv)\# \mu)= {\frac12}\int_{(\R^\d)^2\setminus \triangle} 
\nabla^{\otimes n} \g(x-y):  (v(x)-v(y))^{\otimes n}  d ( \mu_N- \mu)^{\otimes 2} (x,y),
\end{equation}
where $:$ denotes the inner product between the tensors.\footnote{A generalization of the Frobenius inner product of matrices.} 

In questions of dynamics, $v$ is the velocity field of the limiting evolution as $N \to \infty$. In statistical mechanics questions, $v$ is the transport field   for the ``transport method" of \cite{LS2018} (see also its further implementation in \cite{BLS2018,Serfaty2023}). The quantity in \eqref{14} is also the same that  appears in the loop or Dyson-Schwinger equations, which have been used in the case of logarithmic interaction in
\cite{BG2013, BG2024, BBNY2019}, among others. Loop equations are common in mathematical physics and amount to transcribing the conservation of energy under translations, in effect equivalent to a transport method.

The question we address is to bound from above the right-hand side of \eqref{15} in terms of $v$ and the energy $\Fr_N$, which is akin to showing that this expression defines a quadratic form on $\dot{H}^{\frac{\s-\d}{2}}$. Such questions are also of interest in the topic of singular integrals, e.g. \cite{Calderon1980, coifman1978commutateurs, CJ1987, SSS2019}.

The control takes the form of a functional inequality asserting that
the quantity in \eqref{14} is always bounded by $C(\Fr_N(\ux_N, \mu) + N^{-\alpha})$ with $\alpha >0$,  where $C>0$ is a constant depending on $\d,\s$ and the size of $\mu$. For $n=1$, this was first proved  in \cite{LS2018} in the two-dimensional (logarithmic) Coulomb case, then generalized to the super-Coulomb case, with $\max\{\d-2,0\}\le \s <\d$, in \cite{Serfaty2020}. 
 The proof relied in reformulating \eqref{15} in terms of  a {\it stress-energy tensor} (see \cref{ssec:intropf} below)  and using integration by parts.
As observed by the first author \cite{Rosenzweig2020spv}, such an estimate may also be viewed as an estimate for the quadratic form of a \textit{commutator}, reminiscent of the famous Calder\'{o}n commutator \cite{Calderon1980}. This point of view led us in work with Q.H. Nguyen \cite{NRS2021} to generalize such functional inequalities, covering the sub-Coulomb case $\s\le \d-2$ and applying to a broader class of $\g$'s that may be regarded as perturbations of Riesz interactions, including potentials of Lennard-Jones type.
 As previously alluded to, these functional inequalities have been crucial for proving central limit theorems  (CLTs) for the fluctuations of Riesz gases \cite{LS2018, Serfaty2023}, and even more so for deriving mean-field limits \cite{Serfaty2020,Rosenzweig2022,Rosenzweig2022a,NRS2021,CdCRS2023, bPCJ2025} and supercritical mean-field  limits \cite{HkI2021, Rosenzweig2021ne, Menard2022} of classical systems of particles. The inequalities have further found further applications to joint classical and mean-field limits \cite{GP2021} and supercritical mean-field limits \cite{Rosenzweig2021qe, Porat2022} of  quantum systems.

Additionally, second-order versions of these functional inequalities, which state that for $n=2$, the quantity in \eqref{15} is again controlled by 
$ C(\Fr_N(\ux_N, \mu)+ N^{-\alpha})$, were shown in \cite{LS2018,Rosenzweig2020spv}  in the $\d=2$ Coulomb case and \cite{Serfaty2023} in the general  Coulomb case (although with an estimate which is not even optimal in its $\Fr_N$ dependence), and in \cite{NRS2021} for the full Riesz-type case $0\leq \s<\d$, and they were important for the same problem of fluctuations of Coulomb gases, as well as for deriving mean-field limits with multiplicative noise.   

In the above mentioned inequalities, the exponent $\alpha$ in the error term was explicit (in $\d$ and $\s$), but not optimal. The typical nearest-neighbor distance being of order $N^{-1/\d}$, one may correctly intuit  by counting only nearest-neighbor interactions that  $\Fr_N$ is at least of order $N^{\frac{\s}{\d}-1}$; and, in fact, one can show that $\Fr_N$, even though it is not necessarily nonnegative, is nevertheless bounded from below by $- C N^{\frac{\s}{\d}-1}$ (where $C$ depends on the size of $\mu$), and also that $|\min \Fr_N|$ is of order $N^{\frac{\s}{\d}-1}$ \cite{SS2015, SS2015log, RS2016, PS2017} (see also \cite{HSSS2017} for a similar result for the flat torus $\T^\d$ for all $0\leq \s<\d$, as well as the book \cite{BHS2019} for general background on Riesz $\s$-energies). Thus, the best error term that one may hope for is of size $N^{\frac{\s}{\d}-1}$. To date, this sharp error rate has only been only proven up to second order for the $\d=2$ Coulomb case \cite{LS2018,Serfaty2023, Rosenzweig2021ne}, although the estimate obtained there is less satisfactory in the second-order case than the one we will establish here. Having the  sharp rate at first-order is crucial when studying supercritical mean-field limits (see \cite{Rosenzweig2021ne} and \cref{ssec:introapp}) and having a sharp rate at second order is crucial to obtain CLT's for fluctuations of Coulomb gases \cite{LS2018,Serfaty2023}. It is thus a necessary step in extending such results to more general Riesz interactions, and we will immediately use it in the forthcoming paper \cite{RSlake} to study supercritical mean-field, or combined quasi-neutral and mean-field, limits for Riesz interactions (see \cref{ssec:introapp} for elaboration), and in a forthcoming work of Peilen and the second author \cite{PS2024} to obtain a CLT for fluctuations of super-Coulomb Riesz gases. 
 
 The discussion so far has concerned estimates up to second order, but obtaining higher-order estimates (i.e.~$n\ge 3$ in \eqref{15}) is a natural question. For instance, such estimates allow to  obtain finer estimates  on the fluctuations of Riesz gases, allowing in \cite{PS2024} to treat a broader class of interactions. Although not explicitly written, the method of \cite{NRS2021} allows to obtain higher-order estimates, but which are in general not sharp in their additive error. The sharp estimates from \cite{Serfaty2023} are only to second order for the two-dimensional Coulomb case. More importantly, their proof is quite intricate and seems impossible to generalize to higher-order derivatives.  In contrast here, we present a proof which---even though the algebra gets lengthy as $n$ gets large---is much more transparent. It is based on identifying, for the first time, the commutators as solutions of a degenerate elliptic equation with a right-hand side exhibiting a recursive structure. This structure is, in some sense, a dual formulation of a stress-energy tensor structure in the higher-order derivatives. More importantly, this structure allows us to devise a new {\it regularity theory for commutators} leveraging on elliptic regularity theory, which may be of independent interest and allows to obtain {\it local} estimates, in contrast to existing commutator estimates in the literature.

The penultimate improvement in the present paper is the dependence of the estimates on the background density $\mu$, which will appear in the form of a $\|\mu\|_{L^\infty}$ dependence. We observed in \cite{RS2021} that, as $\|\mu\|_{L^\infty}$ gets small, the constant in the functional inequality gets smaller, and this can be exploited to prove {\it uniform-in-time} convergence of the mean-field limit, once one knows a suitable decay rate for the limit $\|\mu^t\|_{L^\infty}$ in time. We will here obtain estimates which are all explicit and sharp in $\|\mu\|_{L^\infty}$.

Our ultimate improvement is that our estimates are all {\it localizable}, in the sense that the control of the  energy in the right-hand side can be performed via a suitable restriction of $\Fr_N$ {\it to the support of the transport $v$ only}, with an error rate which also gets localized. This is essential for proving CLT's at mesoscopic scales {(i.e. $N^{-1/\d} \ll \ell \ll 1)$}, such as in \cite{PS2024}, and such localized estimates were only previously proven for the two-dimensional Coulomb case \cite{LS2018,Serfaty2023}. We emphasize that an approach to such inequalities via classical commutator estimates, as in \cite{Rosenzweig2020spv, NRS2021},  will not provide these kind of localized estimates. 

{Although the method of this paper is currently limited to the Coulomb/super-Coulomb case, forthcoming work \cite{HcRS2024} by the authors and Hess-Childs proves a sharp first-order functional inequality valid in the full Riesz case, i.e. (sub-/super-)Coulomb. The approach of this new work is somewhat orthogonal to that developed here, instead relying on a wavelet-type integral representation of the Riesz interaction and Kato-Ponce type commutator estimates. The reliance on the latter, closer in spirit to \cite{NRS2021}, however does not allow for localizable estimates.}

\subsection{New functional inequalities}
Let us state our main results.

{If one is uninterested in a localized estimate, our main estimate is expressed in terms of the modulated energy itself with the announced new sharp additive error term in $N^{\frac{\s}{\d}-1}$, together with its sharp dependence in $\|\mu\|_{L^\infty}$.
We now state this  result, proven in \cref{ssec:FIunloc}.}

{
\begin{thm}[Global estimates]\label{thm:mainunloc}
There exists a constant $C>0$ depending only $\d,\s$ such that the following holds. Let $\mu \in L^1(\R^\d)\cap L^\infty(\R^\d)$ with $\int_{\R^\d}\mu=1$. {If $\s\le 0$, suppose further that $\int_{(\R^\d)^2}|\g(x-y)|d|\mu|^{\otimes 2}(x,y)<\infty$.} Let  $v:\R^\d\rightarrow\R^\d$ be  a Lipschitz vector field, suppose that $\la<1$, where $\lambda:= (N\|\mu\|_{L^\infty(\R^\d)})^{-\frac1\d}$.  For any pairwise distinct configuration $\ux_N \in (\R^\d)^N$, it holds that
\begin{multline}\label{main2}
\Big|\int_{(\R^\d)^2\setminus\triangle}(v(x)-v(y))\cdot\nabla\g(x-y)d\Big(\frac{1}{N}\sum_{i=1}^N\delta_{x_i} - \mu\Big)^{\otimes 2}(x,y)\Big| \\
\leq C\|\nabla v\|_{L^\infty}\Bigg( \Fr_N (\ux_N,\mu) -  { \Big(\frac{\log \lambda}{2N}\Big) }\indic_{\s=0} +  C\|\mu\|_{L^\infty}\lambda^{\d-\s} \Bigg).
\end{multline}

Furthermore, for any integers $m\ge 0$ and $n\ge 1$, there exists a constant $C>0$ depending only $\d,\s,m,n$ such that the following holds. Let  $v:\R^\d\rightarrow\R^\d$ be a ${C}^{m,1}$ vector field.  For any pairwise distinct configuration $\ux_N \in (\R^\d)^N$, we have
\begin{multline}\label{eq:HOFIunloc}
\Big|\int_{(\R^\d)^2\setminus\triangle}\nabla^{\otimes n}\g(x-y) : (v(x)-v(y))^{\otimes n} d\Big(\frac{1}{N}\sum_{i=1}^N\delta_{x_i}-\mu\Big)^{\otimes 2}(x,y)\Big| \\
 \le C(\|\nab v\|_{L^\infty})^n\Big(\Fr_N (\ux_N,\mu) -  { \Big(\frac{\log \lambda}{2N}\Big) }\indic_{\s=0} +  C\|\mu\|_{L^\infty}\lambda^{\d-\s}\Big)\\
 + C\sum_{q=1}^{m+1}\sum_{r=0}^{{ q}}D_{v,q,r}\Bigg({\la^{m+1}} +\la^{m+1-q}\Big(\Fr_N (\ux_N,\mu) -  { \Big(\frac{\log \lambda}{2N}\Big) }\indic_{\s=0} +  C\|\mu\|_{L^\infty}\lambda^{\d-\s}\Big) \\
 + {\la^{m+1}\|\mu\|_{L^\infty} } \begin{cases} 1, &  \s+q\neq \d \\
\log (1/\lambda),  & \s+q=\d \end{cases} 
+{\la^{m+1}}\begin{cases}\|\mu\|_{L^\infty}\la^{\d-\s-q}, & \s+q>\d \\ \|\mu\|_{L^\infty}(1 - \la^{\d-\s-q}) +  \|\mu\|_{L^1}, & {\s+q<\d} \\ \|\mu\|_{L^\infty}\log(1/\la) + \|\mu\|_{L^1}, & {\s+q=\d} \end{cases}\Bigg),
\end{multline}
where
\begin{align}
D_{v,q,r} \coloneqq  \sum_{\substack{1\le c_1,\ldots,c_r \\ c_1+\cdots+c_r = m+1-(q-r)}}\|\nab^{\otimes c_1}v\|_{L^\infty}\cdots\|\nab^{\otimes c_r}v\|_{L^\infty}\|\nab v\|_{L^\infty}^{n-r}.
\end{align}
\end{thm}
Note that the additive error term is in $N^{-1}\lambda^{\d-\s} \propto (N^{-2+\frac{\s}{\d}})$ per point,  which is the announced optimal estimate. 

Compared to the localized estimate \eqref{mainn} below, the regularity requirement for $v$ is improved, in particular is not dependent on the order $n$ of the commutator. As the reader may check, if $\s>\d-1$, we only need $\|\nab v\|_{L^\infty}$ ($m=0$) to achieve the sharp error $\la^{\d-\s}$. If $\d-2<\s\le \d-1$, then we also need $\|\nab^{\otimes 2}v\|_{L^\infty}$ ($m=1$). If $\s=\d-2$, then we additionally need $\|\nab^{\otimes 3}v\|_{L^\infty}$ ($m=2$).

}

\smallskip

 In the local case, we are given a subset $\Omega \subset\R^\d$, meant to represent the support of the transport vector field $v$. Associated to $\Omega$, we define the microscopic length scale 
\begin{equation}\label{deflambda}
\lambda \coloneqq (N\|\mu\|_{L^\infty(\Omega)})^{-\frac{1}{\d}},
\end{equation}
 which can be thought of as the typical inter-particle distance,
and define $\hat \Omega$ to be the $\frac14 \lambda$-neighborhood of $\Omega$. 
We will use $\ell$ to denote the typical size of $\Omega$ (or of the support of the vector field $v$), which may also depend on $N$ and whose only constraint is  to remain larger than the microscopic scale $\lambda$.
Finally, we let {$I_\Omega\coloneqq \{1\le i\le N: x_i\in\Omega\}$} and use $\#$ to denote the cardinality of a finite set. 

If one does not wish to track the dependence of $\|\mu\|_{L^\infty}$ in $\Omega$, one may simply define
\begin{align}\label{deflambda'}
\lambda \coloneqq (N\|\mu\|_{L^\infty(\R^\d)})^{-\frac1\d}.
\end{align}
Moreover, at the cost of  letting all constants depend on $\|\mu\|_{L^\infty(\R^\d)}$, one can also simply set $\lambda = N^{-1/\d}$
in all the paper.


\begin{thm}[Localized estimates]\label{thm:FI}
There exists a constant $C>0$ depending only $\d,\s$ such that the following holds. Let $\mu \in L^1(\R^\d)\cap L^\infty(\R^\d)$ with $\int_{\R^\d}\mu=1$. {If $\s\le 0$, suppose further that $\int_{(\R^\d)^2}|\g(x-y)|d|\mu|^{\otimes 2}(x,y)<\infty$.} Let  $v:\R^\d\rightarrow\R^\d$ be  a Lipschitz vector field  and $\Omega$ be a closed set  containing a $2\lambda$-neighborhood of $\supp \nab v$, where $\lambda$ is  defined as in \eqref{deflambda} and $\la<1$.  For any pairwise distinct configuration $\ux_N \in (\R^\d)^N$, it holds that
\begin{multline}\label{main1}
\Big|\int_{(\R^\d)^2\setminus\triangle}(v(x)-v(y))\cdot\nabla\g(x-y)d\Big(\frac{1}{N}\sum_{i=1}^N\delta_{x_i} - \mu\Big)^{\otimes 2}(x,y)\Big| \\
\leq C\|\nabla v\|_{L^\infty}\Big({\int_{\Omega\times [-\ell, \ell]^{\k}}\zg|\nab h_{N,\vec{\rs}}|^2} + C\frac{\# I_\Omega\|\mu\|_{L^\infty(\hat \Omega)}\la^{\d-\s}}{N} \Big).
\end{multline}
{Here, $h_{N,\vec{\rs}}$ is the truncated electric potential defined in \eqref{eq:defHNeta} below.}

Suppose in addition  that  $v\in C^{2n-1,1}$ and that  $\Om'$ is a ball of radius $\ell$ containing a $\lambda$-neighboorhood of $\supp v$ and $\Omega$  contains a $5\ell$-neighborhood of $\Omega'$,\footnote{{For integer $k\ge 0$ and real $\al\in [0,1]$, $C^{k,\alpha}$ denotes the inhomogeneous H\"{o}lder space of $k$-times continuously differentiable functions whose $k$-th derivative is $\alpha$-H\"older continuous.}} where $\la,\ell$ satisfy $\min(\la,\la/\ell) < \frac{1}{2}$. For any $n \ge 2$, we have 
\begin{multline} \label{mainn}
\Big|\int_{(\R^\d)^2\setminus\triangle}\nabla^{\otimes n}\g(x-y) : (v(x)-v(y))^{\otimes n} d\Big(\frac{1}{N}\sum_{i=1}^N\delta_{x_i} - \mu\Big)^{\otimes 2}(x,y)\Big| \\
 \qquad \le  C\sum_{p=0}^n (\ell\|\nab^{\otimes 2}v\|_{L^\infty})^p
 \sum_{\substack{1\leq c_1,\ldots,c_{n-p} \\ n-p\le c_1+\cdots+c_{n-p} \le 2n}} \lambda^{-(n-p)+\sum_{k=1}^{p} c_{n-k} } \|\nabla^{\otimes c_1} v\|_{L^\infty}\cdots\|\nabla^{\otimes c_{n-p}} v\|_{L^\infty}  \\
 \qquad \qquad \qquad \times \Big({\int_{\Omega\times [-\ell,\ell]^{\k}} \zg |\nab h_{N,\vec{\rs}}|^2 }+  C\frac{\# I_\Omega\|\mu\|_{L^\infty(\hat \Omega)}\la^{\d-\s}}{N} \Big),
 \end{multline} 

where $C>0$ depends only on $n,\d,\s$ and the summation of the $c_i$ is understood as vacuous when $p=n$. 
\end{thm}

 Note that the estimate \eqref{mainn} does not allow to take $\Omega=\R^\d$ because that would mean $\ell=\infty$, in which case the right-hand side is infinite. 

In the estimate \eqref{mainn}, the dependence in $v$ retains the correct homogeneity in that if $v$ varies at small scale $\ell>\lambda$, they scale like $\ell^{-n}$. To be more precise, using that $\ell \ge 2\lambda $, we have the following  corollary.

\begin{cor}Let $n \ge {1}$. Assume in addition that there exists $M>0$ such that $\|\nabla^{\otimes k} v\|_{L^\infty} \le M \ell^{-k}$ for every $k \le 2n$. 
Then 
\begin{multline} \label{mainn2}
\Big|\int_{(\R^\d)^2\setminus\triangle}\nabla^{\otimes n}\g(x-y) : (v(x)-v(y))^{\otimes n} d\Big(\frac{1}{N}\sum_{i=1}^N\delta_{x_i} - \mu\Big)^{\otimes 2}(x,y)\Big| \\
\leq CM \ell^{-n}
 \Big({\int_{\Omega\times[-\ell,\ell]^{\k}} |z|^\ga |\nab h_{N,\vec\rs}|^2} +  C\frac{\# I_\Omega\|\mu\|_{L^\infty(\hat \Omega)}\la^{\d-\s}}{N}  \Big) .
\end{multline}
\end{cor}
\begin{remark}
While the conditions on $\Om,\Om',\la$ may appear circular upon first read, it is indeed possible for them to hold. To see this, let $x_0\in\supp v$ and consider a ball $B(x_0,\ell_0)\supset \supp v$. We choose $\ell_0$ sufficiently large so that
\begin{align}
(N\|\mu\|_{L^\infty})^{-1/\d}\le (N\|\mu\|_{L^\infty(B(x_0,\ell_0))})^{-1/\d} < 2(N\|\mu\|_{L^\infty})^{-1/\d}.
\end{align}
We then take $\ell \coloneqq \ell_0 +  2(N\|\mu\|_{L^\infty})^{-1/\d}$, $\Om' \coloneqq B(x_0,\ell)$; and $\Om\coloneqq B(x_0,6\ell)$. Evidently, $\supp v \subset \Om' \subset \Om$, and, by the triangle inequality, $\Om$ contains a $5\ell$-neighborhood of $\Om'$. Since
\begin{align}
\la = (N\|\mu\|_{L^\infty(\Om)})^{-1/\d} \le (N\|\mu\|_{L^\infty(B(x_0,\ell_0))})^{-1/\d} < 2(N\|\mu\|_{L^\infty})^{-1/\d} = \ell - \ell_0,
\end{align}
it follows from the triangle inequality that $\Om'$ contains a $\la$-neighborhood of $\supp v$.
\end{remark}

\subsection{Proof method: stress tensor and commutators}\label{ssec:intropf}
The proofs in prior papers, starting in  \cite{Duerinckx2016,LS2018,Serfaty2020}, rely on the {\it electric reformulation} of the modulated energy \eqref{eq:modenergy} as a (renormalized) version of the energy 
\begin{equation}\label{eq:gradHn}
\int_{\R^\d} |\nabla h_N|^2,
\end{equation}
where $h_N$ denotes the electric potential 
\begin{equation} \label{defHn}
h_N\coloneqq\g* ( \mu_N-\mu), 
\end{equation}and $\mu_N$ is the empirical measure. 
This reformulation is valid for the Coulomb case. In the super-Coulomb case, we may instead use the Caffarelli-Silvestre extension procedure  \cite{CS2007}. It allows to  view $\g$ as the kernel of a local operator $-\div (\zg \nab \cdot)$ and 
replace 
 \eqref{eq:gradHn}   by the  same quantity with weight $\zg$, once working in the extended space $\R^{\d+1}$ (the precise extension procedure is described in \cref{ssec:MEts}), with $z$ being the $(\d+1)$-th variable.   The fact that such a procedure is not available in the same way  if $\s<\d-2$ is what restricts us  so far to the range $\s \ge \d-2$ (cf. \cite[Sections 2.1-2.2]{NRS2021}). 

The key to the proof of inequalities such as \eqref{main1} or \eqref{main2} was the observation of a {\it stress-energy tensor structure} in the left-hand side. More precisely, introducing the stress-energy tensor (from classical mechanics or calculus of variations) associated to $h$ as 
\begin{equation}\label{defTT}
(T_h)_{ij}\coloneqq  2\partial_i h \partial_j h - |\nabla h|^2 \delta_{ij}.
\end{equation}
It is well-known that if $h$ is smooth enough, then 
\begin{equation} \label{divt}
\div T_h= 2 \nabla h \Delta h,
\end{equation} 
where the {divergence may be taken with respect either rows or columns, as the tensor is symmetric.} Ignoring the question of the diagonal excision in \eqref{eq:modenergy}, the main point is that in the Coulomb case, using \eqref{defHn} and \eqref{divt}, one may rewrite the left-hand side of \eqref{main2} after desymmetrization, and recognizing a convolution, as 
\begin{align}\label{sstruc} 
2\int_{\R^\d}\int_{\R^\d} v(x) \cdot \nabla \g(x-y) d(\mu_N- \mu) (y) d( \mu_N- \mu) (x) &=  -\frac{2}{\cd}\int_{\R^\d} v\cdot \nabla h_N \Delta h_N\nn\\
&=-\frac{1}{\cd} \int_{\R^\d} v \cdot \div T_{ h_N}.
\end{align}
An integration by parts then allows to formally conclude, since $\int_{\R^\d} |T_{ h_N}|\le \int_{\R^\d} |\nabla h_N|^2$, which is the energy. The preceding reasoning is purely formal, since in truth, $|\nabla h_N|^2$ diverges near each  point of the configurations (this is also related to the diagonal excision), and this computation needs to be  properly renormalized, which is the main technical roadblock in the proof. Things work the same in the Riesz case, after extending the space and adding the appropriate weight.

As previously mentioned, in \cite{NRS2021} we were able to bypass the use of the stress tensor, which is rigidly linked to the Coulomb or Riesz nature of $\g$, and replace it with more delicate commutator estimates that we devised. 

Considering a pure stress-tensor approach to the proof of higher-order estimates \eqref{mainn} for $n \ge 2$, it is unclear that an algebraic manipulation like \eqref{sstruc} can be found, and this has been a major obstacle so far---leading to delicate proofs in \cite{LS2018,Serfaty2023} which do not seem extendable to $n \ge 3$. For the first time, we are able to exhibit a suitable---albeit more complicated---{stress-tensor structure in the higher-order variations} \eqref{15}. This stress-tensor structure now involves  not only the electric potential $h_N$, but also new functions that we call {\it iterated commutators} of it. Let us describe this more precisely  in the Coulomb case.

Given a {measure} $f$ in $\R^\d$ of integral $0$ (think of $f=\mu_N-\mu$), $h^f = \g* f$ its Coulomb potential, and a vector-field $v$, we define the first commutator of $h^f$ as 
\begin{equation}
\kappa^{f}(x) \coloneqq \int_{\R^\d} \nabla \g(x-y) \cdot (v(x)-v(y)) df(y).
\end{equation}
It is a ``commutator'' because it can be rewritten as 
\begin{equation}\label{rewrikappa}
\kappa^f =-\g * (\div (vf)) +v \cdot \nabla \g  * f= - h^{\div (vf)}+ v \cdot \nabla h^f.
\end{equation}
In fact, the commutator  $\kappa^f$ is intimately tied to the stress-tensor itself, via the relation \eqref{sstruc}, which can be rewritten in general as
\begin{equation}\label{divtk}
2\cd\int_{\R^\d} \kappa^f dw = \cd\int_{\R^\d} \kappa^f dw  + \cd\int_{\R^\d} \kappa^w df  =  -\int_{\R^\d} v \div \comm{\nab h^f}{\nab h^w}= \int \nabla v : \comm{\nab h^f}{\nab h^w},
\end{equation}
where
\begin{align}\label{eq:introst}
\comm{\nab h^f}{\nab h^w}_{ij} \coloneqq \p_i h^f\p_j h^w + \p_j h^f \p_i h^w - \nab h^f\cdot\nab h^w\delta_{ij}	
\end{align}
is the bilinear generalization of the stress-energy tensor $T_{h}$ from \eqref{defTT}, which we note is symmetric (hence, it does not matter in which coordinate the divergence is taken).

Considered separately, each term in the definition of the function $\kappa^f$ is one derivative less regular than $h^f$; however thanks to its commutator structure, some compensation happens.  By choosing $w= -\frac{1}{\cd}\Delta \ka^{f}$ and using integration by parts on the left-hand side and Cauchy-Schwarz on the right-hand side, we are able to 
show (see \cref{prop:comm}) {the $\dot{H}^1$ control}
\begin{equation}\label{commcontrol1}
\int_{\R^\d} |\nabla \kappa^f |^2  \le C \|\nabla v\|_{L^\infty} \int_{\supp \nab v} |\nabla h^f|^2.
\end{equation}
 Compared to classical proofs of commutator estimates by means of paraproducts, the proof of \eqref{commcontrol1} is remarkably simple, thanks again to the use of the stress tensor. 

 Having discussed first-order estimates, let us now see how to obtain the second-order estimates with the help of the stress-tensor and commutator structures. Ignoring the question of renormalization, and sticking to the Coulomb case still, we have seen in \eqref{sstruc} (combining with \eqref{15}) that the first variation of the modulated energy along the transport map $\I+tv $ is
\begin{align}\label{124}
-\frac1{2\cd}\int_{\R^\d} v \cdot \div T_{h_N}dx = \frac1{2\cd} \int_{\R^\d} \nabla v : T_{h_N}dx.
\end{align}
To compute the second variation of the modulated energy, we thus need to compute the first variation of $\div T_{h_N}$ when again $\mu_N$ and $\mu$ are pushed forward by $\I+ tv$. In view of the expression for $T_{h_N}$ in \eqref{defTT}, it suffices to compute the derivative of $h_N^t$ (with obvious notation) at $t=0$, and since $
h_N= \g * (\mu_N-\mu)$, the definition of the push-forward yields that 
\begin{equation} \label{dthnt}\frac{d}{dt}\big|_{t=0} h_N^t= -\g * (\div (v (\mu_N-\mu)) =-  h^{\div (v(\mu_N-\mu))}=  \kappa^{\mu_N-\mu} -v\cdot \nab h_N,\end{equation} after using \eqref{rewrikappa}. 
{The $L^2$ norm of the gradient of this expression is one derivative more singular than the energy $\int_{\R^\d} |\nabla h_N|^2$. (equivalently, the $L^2$ norm of $\mu_N - \mu$ versus the $\dot{H}^{-1}$ norm of $\mu_N-\mu$.) }
Still with $f= \mu_N-\mu$, inserting \eqref{dthnt} into the variation of \eqref{124}, we can  decompose the second-order variation as 
\begin{equation}\label{line12}
\frac1{\cd}\int_{\R^\d} \p_i v^j : \left( \comm{\nab  h_N }{\nab  \kappa^f }-  \comm{\nab  h_N }{ \nab ( v\cdot \nabla h_N  ) }\right).
\end{equation}
 Thanks to the $\dot{H}^1$ estimate \eqref{commcontrol1} for the commutator $\ka^f$, the  first  term on the right-hand side can directly be controlled by $ C_v \int |\nabla h_N|^2$ as desired, while the second one can be transformed into  similarly controllable terms---albeit $C_v$ is now quadratic in $v$ and depends on $v,\nabla^{\otimes 2}v$---by means of integration by parts of the advection operator $v \cdot \nabla $.

This argument can then be iterated at next order, by introducing the family of $n$-th order  commutators $\kappa^{(0), f}=h^f$ and 
\begin{equation}
\kappa^{(n),f}\coloneqq \int_{\R^\d} \nabla^{\otimes n} \g(x- y) : (v(x)-v(y))^{\otimes n} df(y),
\end{equation}
together with ``transported" commutators
\begin{equation}
\kappa_t^{(n),f}\coloneqq \int_{\R^\d} \nabla^{\otimes n} \g(x-y-tv(y)) : (v(x)-v(y))^{\otimes n} df(y),
\end{equation}
and observing the nice recursion formula
\begin{equation}\label{iteratek}
\partial_t \kappa_t^{(n),f} = \kappa_t^{(n+1),f} - v \cdot \nabla \kappa^{(n),f}_t,
\end{equation}
where $t$ is the same parameter as in the map $\I+tv$ under which $f$ is originally transported. In other words, the high-order ``time-dependent'' commutators $\ka_t^{(n),f}$ satisfy a hierarchy of transport equations with a source coupling the $(n+1)$-th order commutator to the $n$-th order commutator.  An iteration of the same argument  as for \eqref{commcontrol1} using this recursion allows to prove the estimate 
\begin{equation}\label{commuordern}
\int_{\R^\d} |\nabla \kappa^{(n),f}|^2 \le C_v \int_{\supp\nab v} |\nabla h^f|^2,
\end{equation}
valid for general functions $f$ such that the right-hand side is finite, with the constant $C_v $ now being $n$-linear in $v$ and involving derivatives up to order $n$ of $v$.  One can see the relation \eqref{commuordern}, stated and proved in \cref{prop:comm}, as an {\it $L^2$-based  regularity theory for arbitrary-order commutators}. 


In the preceding argument, we have only discussed  the Coulomb case,  the complete details of which are presented in \cref{appA}. In generalizing to the Riesz case $\s \neq \d-1,\d-2$, we run into issues. The vector field $v$ may be trivially extended to $\R^{\d+\k}$, with $\k=1$ if $\s>\d-2$, by fixing the last component to zero, $\g$ may be trivially extended through radial symmetry, and all distributions $f$ on $\R^{\d}$ viewed as living on the boundary $\R^{\d}\times \{0\}^\k$. Consequently, $\ka^{(n),f}$ may be viewed as a function on $\R^{\d+\k}$.  Going through the computations above, all integrals $\int_{\R^\d}$ should be replaced $\int_{\R^{\d+\k}}|z|^{\gamma}$ (see \eqref{eq:defgamma} below for the value of $\ga$). Where we run into issues is the step of obtaining the $L^2$ gradient bound \eqref{commuordern} by duality.

Setting $L \coloneqq -\div(\zg\nab\cdot)$, we would like to take $w = \frac{1}{\cd} L\ka^{(n),f}$ in  $\int_{\R^{\d+\k}}\ka^{(n),f}w$
and then integrate by parts to conclude an estimate for $\int_{\R^{\d+\k}}\zg |\nab \ka^{(n),f}|^2$. However, this choice for $w$ is \textit{a priori} not supported on $\R^{\d}\times\{0\}^\k$. If we abandon the requirement that $\supp w \subset \R^{\d}\times \{0\}^\k$, it is no longer necessarily true that $\g\ast L\phi = \cd\phi$, for a test function $\phi$ on $\R^{\d+\k}$, unless $\ga = 0$. This forces us to come up with a new approach.

Our starting point for this new approach is the observation $L\g = \cd\delta_0$ in $\R^{\d+\k}$, which underlies the Caffarelli-Silvestre  representation of the fractional Laplacian as a degenerate elliptic operator in $\R^{\d+\k}$. Using this observation, we show (see \cref{lem:Lkapnf} and more generally, \cref{ssec:L2regeqn}) that $\ka^{(n),f}$ satisfies the equation
\begin{multline}\label{eq:introLkanf}
L\ka^{(n),f} = \cd(-1)^n\sum_{\sigma \in \Ss_n} \p_{i_{\sigma_1}}{v}^{i_1}\cdots\p_{i_{\sigma_n}}v^{i_n}f\delta_{\R^\d\times\{0\}^\k}   -n\p_i(\zg\p_i v \cdot\nu^{(n-1),f})  \\
- n\zg\p_i v \cdot \p_i\nu^{(n-1),f}+ n(n-1)\zg(\p_i v)^{\otimes 2}:\mu^{(n-2),f},
\end{multline}
where $\mathbb{S}_n$ denotes the symmetric group on $[n] \coloneqq \{1,\ldots,n\}$ and a repeated index denotes summation over that index. In the right-hand side, $\nu$ is a vector field on $\R^{\d+\k}$ that ``morally'' is like $\nab\ka^{(n),f}$, while $\mu$ is a symmetric matrix field on $\R^{\d+\k}$ that morally is like $\nab^{\otimes 2}\ka^{(n),f}$, see \eqref{eq:nudef} and \eqref{eq:defmu}, respectively, for the precise definitions.  The right-hand side of \eqref{eq:introLkanf} is good because it only depends on lower-order (i.e.~ $n-1$, $n-2$) commutators. Following the standard method for the $L^2$ regularity of elliptic equations, we want to obtain an estimate for $\int_{\R^{\d+\k}}\zg|\nab\ka^{(n),f}|^2$ by testing the equation \eqref{eq:introLkanf} against $\ka^{(n),f}$ and using an induction hypothesis that will guarantee good bounds on the lower order commutators and their derivatives. There is however a technical difficulty, discussed in Section \ref{ssec:L2regeqn},  which requires when $n \ge 3$ the use of  the Poincar\'e inequality, yielding an estimate that deteriorates when $\ell$, the size of the support of $v$, gets large. 

As a result, when $\supp v$ is not compact (this is not the case for the aforementioned CLT application however), this   approach does  not work, and the right-hand side of \eqref{mainn} is infinite. However, obtaining unlocalized $L^2$ commutator estimates is much easier and may be done following  the method of the authors' work \cite{NRS2021} with Q.H. Nguyen. Namely, we use integration by parts and the difference quotient characterization of the Sobolev seminorm $\dot{H}^{\frac{\s-\d}{2}}$ when $\s>\d-2$ and the $L^2$ boundedness of Calder\'{o}n $\d$-commutators when $\s=\d-2$. The final estimate and its proof are presented in  \cref{ssec:L2regunloc}.

The preceding results are summarized in \cref{thm:comm2}, which is our omnibus $L^2$ regularity estimate for commutators, and more generally \cref{sec:L2reg}. 

This discussion has so far left aside the question of ``renormalization,'' which is that of dealing with the singularities of the Dirac masses in $h_N$, and which is arguably the most delicate part of the analysis. For the first-order estimates, the renormalization is handled as in \cite{Serfaty2020} and subsequent works via a charge smearing and potential truncation at a lengthscale crucially {\it depending on each point} $x_i$ and equal to (a quarter of) the minimal distance from $x_i$ to all other points $x_j$. See \cref{ssec:MEts} for the charge smearing/potential truncation and \cref{ssec:FOren} for its implementation to prove the first-order estimate \eqref{main1} of \cref{thm:FI}. The renormalization in the $n\ge 2$ case is much more delicate, occupying the largest part of the paper and requiring several innovations.

First, replacing the Diracs in the left-hand side of \eqref{mainn} by their smearings $\delta_{x_i}^{(\eta_i)}$ will not allow us to apply our unrenormalized commutator estimate \cref{thm:comm2}, unless $\s=\d-2$, because $\mu\delta_{\R^\d\times\{0\}^\k} - \frac1N\sum_{i=1}^N \delta_{x_i}^{(\eta_i)}$ is not supported on $\R^\d\times \{0\}^\k$. To overcome this difficulty, we introduce a new smearing $\rho_{x_i}^{(\eta_i)}$ (see \eqref{eq:trhodef}), obtained simply by mollifying the Dirac in $\R^\d$, and estimate the electric energy difference of the two regularizations in terms of the modulated energy (see \cref{lem:barHN}). We can then directly apply our $L^2$ commutator bound to $\mu-\frac1N\sum_{i=1}^N\rho_{x_i}^{(\eta_i)}$. A novelty compared to previous work is that we will choose the mollifier to have vanishing moments up to high enough order, in order to obtain optimal error rates.

Second, to handle the error from this smearing, we establish an {\it $L^\infty$-based regularity theory for commutators}, which shows that the $L^\infty$ norm of ``horizontal'' derivatives $\nab_{x}^{\otimes m}\kappa^{(n),f}$, where $\nab_x \coloneqq (\p_1,\ldots,\p_\d)$, {\it restricted to arbitrarily small balls} can be estimated in terms of the $L^2$ norm of $\nab\ka^{(n),f}$ in a double ball (see \cref{th44}). Commutator estimates do not typically allow such localization, and therefore this result may be of independent interest. The proof relies on regularity theory for second-order divergence-form elliptic operators with $A_2$ weights \emph{\`{a} la} Fabes \emph{et al.} \cite{FKS1982} in the extended space $\R^{\d+\k}$ (see \cref{ssec:L2regLinf}), which, by the Caffarelli-Silvestre extension, may be viewed as a regularity theory for the fractional Laplacian. The proof also crucially uses the commutator structure in the form of equation \eqref{eq:introLkanf} for $L\ka^{(n),f}$, along with the recursions for $\ka^{(n),f}$, $\nu^{(n),f}$, and $\mu^{(n),f}$. When $\ell \asymp 1$, this local regularity theory may be avoided with cruder estimates, which have the benefit of requiring less regularity on the vector field $v$ (see \cref{prop:impHOFI}). However, when $\ell\ll 1$, these cruder estimates give errors that scale sub-optimally.  

Third, we improve truncation errors estimates from \cite{Serfaty2020}, obtaining the optimal dependence on $\mu$ (see \cref{ssec:MEloc}), and show new localizable {\it mesoscale interaction energy estimates}, which generalize those of  \cite{Serfaty2023} to the Riesz case (see \cref{ssec:MEmulti}).

\subsection{Applications}\label{ssec:introapp}
We now discuss applications of \cref{thm:mainunloc} related to mean-field and supercritical mean-field limits.

For the first application, we consider first-order dynamics of the form
\begin{equation}\label{eq:MFode}
\begin{cases}
\dot{x}_i^t = \displaystyle\frac{1}{N}\sum_{1\leq j\leq N : j\neq i} \M\nabla\g(x_i^t-x_j^t) - \mathsf{V}(x_i^t) \\
x_i^t\big|_{t=0} = x_i^\circ,
\end{cases}\qquad i\in[N].
\end{equation}
Here, $x_i^\circ \in\R^\d$ are the pairwise distinct initial positions, $\M$ is a $\d\times \d$ constant real matrix such that
\begin{equation}\label{eq:Mnd}
\M\xi\cdot\xi \leq 0 \qquad \forall \xi\in\R^\d,
\end{equation}
 which is a  repulsivity assumption, and $\mathsf{V}$ is an external field (e.g. $-\nabla V$ for some confining potential $V$). Choosing $\M =-\I$ yields \emph{gradient/dissipative} dynamics, while choosing $\M$ to be antisymmetric yields \emph{Hamiltonian/conservative} dynamics; mixed flows are also permitted. We assume that $\g$ is of the form \eqref{eq:gmod}. Note that we are restricting to the potential case $\s<\d$, as in the \emph{hypersingular} case $\s\geq \d$, $\g$ is no longer locally integrable and the dynamics of \eqref{eq:MFode} are of a different nature (e.g. see \cite{HSST2020}). One can check that our assumption \eqref{eq:Mnd} for $\M$ ensures that the energy for \eqref{eq:MFode} is nonincreasing, therefore if the particles are initially separated, they remain separated for all time, so that there is a unique, global strong solution to the system \eqref{eq:MFode}.

Motivations for considering systems of the form \eqref{eq:MFode} are reviewed in \cite{Serfaty2020}. We briefly mention that Riesz interactions are particularly interesting for applications to physics and approximation theory, as discussed for instance in \cite{DRAW2002, BHS2019}.

The mean-field limit refers to the convergence as $N \to \infty$ of the {\it empirical measure} 
\begin{equation}\label{eq:EMt}
\mu_N^t\coloneqq \frac1N \sum_{i=1}^N \delta_{x_i^t}
\end{equation}
associated to a solution $\ux_N^t \coloneqq (x_1^t, \dots, x_N^t)$ of the system \eqref{eq:MFode}. Assuming the points $\ux_N^\circ$, which themselves depend on $N$, are such that $\mu_N^\circ$ converges to a measure $\mu^\circ$ with sufficiently regular density, then a formal derivation leads one to expect that for $t>0$, $\mu_N^t$ converges to the solution of the Cauchy problem 
\begin{equation}\label{eq:MFlim}
\begin{cases}
\partial_t \mu= \div ((\mathsf{V}-\M \nabla \g*\mu) \mu)  & \\
\mu|_{t=0} = \mu^\circ, & \end{cases} \qquad (t,x)\in\R_+\times\R^\d.
\end{equation}
Convergence of the empirical measure is qualitatively equivalent to {\it propagation of molecular chaos} (see \cite{Golse2016ln,HM2014} and references therein).  This latter notion means that if $f_N^\circ(x_1, \dots, x_N)$ is the joint probability distribution of the initial positions $X_N^\circ$ and if $f_N^\circ$ is $\mu^\circ$-chaotic (i.e.~the marginals $f_{N;k}^\circ\rightharpoonup (\mu^\circ)^{\otimes k}$ as $N\rightarrow\infty$ for every fixed $k$), then the joint distribution $f_N^t$ of $\ux_N^t$ is $\mu^t$-chaotic.

There is a long history to mean-field limits for systems of the form \eqref{eq:MFode}, beginning with regular velocities (typically, Lipschitz) \cite{Dobrushin1979, Sznitman1991}. Singular interaction are far more challenging and only recently have breakthroughs been made to cover the full Riesz case $\s<\d$: the sub-Coulomb case $\s<\d-2$, \cite{Hauray2009,CCH2014}, the Coulomb/super-Coulomb case $\d-2\leq \s<\d$ \cite{Duerinckx2016,CFP2012,BO2019,Serfaty2020}, and the full case $0\leq \s<\d$ \cite{BJW2019edp, NRS2021}, thanks in large part to the modulated energy.  Further extensions of the modulated energy method have been obtained in the Riesz case in terms of the regularity assumptions on the limiting equation \cite{Rosenzweig2022, Rosenzweig2022a} and incorporating multiplicative \cite{Rosenzweig2020spv} and additive noise \cite{RS2021, hC2023}. We also mention the relative entropy method \cite{JW2018, GlBM2021, FW2023, RS2024ss, FW2024gc}, which allows treats $\dot{W}^{-1,\infty}$ forces, essentially corresponding to the $\s=0$ case, and its combination with the modulated energy in the form of the modulated free energy \cite{BJW2019crm, BJW2019edp, BJW2020, CdCRS2023, RS2024ss}, which is well suited to overdamped Langevin dynamics and can even handle logarithmically attractive interactions \cite{BJW2019crm,BJW2020,CdCRS2023a, CFGW2024}.  Finally, we mention an exciting recent direction focused on weighted estimates for hierarchies of marginals \cite{Lacker2023, LlF2023, BJS2022, Wang2024sharp} and cumulants \cite{hCR2023,BDJ2024}. Although these approaches are currently unable to treat the full Riesz range, they do have advantages in terms of working for both first- and second-order dynamics and can, in certain cases, achieve sharp rates for propagation of chaos (cf. \cref{rem:pc} below). For a proper discussion of contributions and the techniques behind them, we refer to the recent survey \cite{CD2021}, the lecture notes \cite{Golse2022ln,Golse2016ln, JW2017_survey, Jabin2014}, and the introductions of \cite{Serfaty2020,NRS2021}. 

As explained in \cref{ssec:introMot}, the optimal rate of convergence as $N\rightarrow\infty$ of the empirical measure $\mu_N^t$ to the solution $\mu^t$ of \eqref{eq:MFlim} in the distance $\Fr_N$ is $N^{\frac{\s}{\d}-1}$. To the best of our knowledge, achieving this optimal rate has been an open problem, with the only results \cite{SS2015, SS2015log,RS2016, PS2017} being for stationary solutions of \eqref{eq:MFlim} with $\M=-\mathbb{I}$, corresponding to minimizers of the associated Coulomb/Riesz energy. Our first application of \cref{thm:FI} establishes  mean-field convergence at the optimal rate $N^{\frac{\s}{\d}-1}$ for the full Coulomb/super-Coulomb case. As explained in \cref{sec:appMF}, the proof is a straightforward consequence of the sharp first-order estimate \eqref{main1}. {The aforementioned forthcoming joint work \cite{HcRS2024} with Hess-Childs establishes the corresponding $N^{\frac{\s}{\d}-1}$ rate for the remaining sub-Coulomb case.}

\begin{thm}\label{thm:mainMF}
Let $\g$ be of the form \eqref{eq:gmod} and $\|\nabla\mathsf{V}\|_{L^\infty}<\infty$. Assume the equation \eqref{eq:MFlim} admits a solution $\mu \in L^\infty([0,T],\P(\R^\d)\cap L^\infty(\R^\d))$, for some $T>0$, such that 
\begin{equation}\label{nmut}
\|\nabla^{\otimes2}\g\ast\mu^t\|_{L^\infty([0,T], L^\infty)}< \infty.
\end{equation}
If $\s=0$, then also assume that $\int_{\R^\d}\log(1+|x|)d\mu^t(x) < \infty$ for all $t\in [0,T]$.\footnote{This condition is to ensure that the convolution $\g\ast\mu$ is a well-defined function. The assumption is purely qualitative: none of our estimates will depend on it. Through a Gronwall argument, one checks that if $\mu^0$ satisfies this condition, then $\mu^t$ also satisfies this condition uniformly in $[0,T]$.} 

Let $\ux_N$ solve \eqref{eq:MFode}. Then there exists a constant  $C>0$, depending only on $|\M|$, $\d$, $\s$,  such that
\begin{multline}\label{distcoul}
\Fr_N(\ux_N^t,\mu^t) + \frac{\log (N\|\mu^t\|_{L^\infty}) }{2 N\d} \indic_{\s=0} + C\|\mu^t\|_{L^\infty}^{\frac{\s}{\d}}N^{\frac{\s}{\d}-1}\\
 \le e^{C\int_0^t \|\nab u^\tau\|_{L^\infty} d\tau}\Bigg(\Fr_N(\ux_N^0,\mu^0) + \sup_{\tau \in [0,t]}\Big( \frac{\log (N\|\mu^\tau\|_{L^\infty}) }{2 N\d} \indic_{\s=0} + C\|\mu^\tau\|_{L^\infty}^{\frac{\s}{\d}}N^{\frac{\s}{\d}-1}\Big)\Bigg),
\end{multline}
where $u^t \coloneqq -\M\nab\g\ast\mu^t + \mathsf{V}$.

In particular, if $\mu_N^0 \rightharpoonup \mu^0$ in the weak-* topology for measures and
\begin{equation}
\lim_{N\to \infty}\Fr_N(\ux_N^0, \mu^0) =0,
\end{equation}
then
\begin{equation}
\mu_N^t \rightharpoonup \mu^t \qquad \forall t\in [0,T].
\end{equation}
\end{thm}

\begin{remark}\label{rem:pc}
\cref{thm:mainunloc} implies propagation of chaos for the marginals of the system \eqref{eq:MFode} with initial data $\ux_N^\circ$ randomly chosen according to a $\mu^\circ$-chaotic law. For instance, see \cite[Remark 3.7]{Serfaty2020} or \cite[Remark 1.5]{RS2021} for details.  However, this ``global-to-local'' argument in general leads to suboptimal rate. See \cite{Lacker2023} for further discussion.
\end{remark}

\medskip

In the paper \cite{RSlake}, we give another application of \cref{thm:FI} to the convergence of the empirical measure $f_N^t \coloneqq \frac1N\sum_{i=1}^N \delta_{(x_i^t,v_i^t)}$ for the second-order/kinetic system
\begin{equation}\label{eq:NewODE}
\begin{cases}
\dot{x}_i^t = v_i^t \\
\dot{v}_i^t = -\ga v_i^t \displaystyle-\frac{1}{\vep^2 N} \sum_{1\leq j\leq N: j\neq i}\nabla\g(x_i^t-x_j^t) -\frac{1}{\vep^2}\nabla V(x_i^t)\\
(x_i^0,v_i^0)= (x_i^\circ,v_i^\circ)
\end{cases}
\qquad i \in [N],
\end{equation}
in the joint limit $\vep\rightarrow 0$ and $N\rightarrow\infty$. Here, $\g$ is a (super-)Coulomb Riesz interaction as in \eqref{eq:gmod}, $\ga\geq 0$ is the friction coefficient, and $\vep>0$ is a small parameter, possibly depending on $N$, which encodes physical information about the system. Because $\vep^2 N \ll N$, the force in \eqref{eq:NewODE} formally diverges as $N\rightarrow\infty$ and is an example of a \emph{supercritical mean-field limit}, following the terminology of \cite{HkI2021}. Alternatively, the joint limit $\vep\rightarrow 0$ and $N\rightarrow\infty$ may be interpreted as a combined mean-field and \emph{quasi-neutral limit}. 

Let $\mu_\infty$ be the equilibrium measure, i.e.~the probability measure that minimizes the energy
\begin{equation}
  \frac{1}{2}\int_{(\R^\d)^{ 2}}\g(x-y)d\mu^{\otimes 2}(x,y)+\int_{\R^\d}Vd\mu.
\end{equation}
Under suitable assumptions on the external potential $V$, we show using a modulated energy approach that if the initial empirical measure $f_N^\circ(x,v)$ converges to the probability measure $\mu_\infty(x)\delta_{u^\circ(x)}(v)$ in a suitable sense as $\vep\rightarrow 0$ and $N\rightarrow\infty$, then $f_N^t(x,v)$ converges to the probability measure $\mu_\infty(x)\delta_{u^t(x)}(v)$, where $u^t$ is a solution to the \emph{Lake equation}
\begin{equation}\label{eq:Lake}
\begin{cases}
\p_t u +\ga u+ u \cdot\nabla u = -\nabla p\\
\div(\mu_\infty u) = 0\\
u^0 = u^\circ.
\end{cases}
\end{equation}
The pressure $p$ is a Lagrange multiplier to enforce the incompressibility constraint $\div(\mu_\infty u)=0$. Note that if $\mu_\infty$ is constant and $\ga=0$, then \eqref{eq:Lake} is nothing but the incompressible Euler equation. Our result strongly generalizes previous work \cite{HkI2021,Rosenzweig2021ne} for the periodic setting $V=0$ (in which case, $\mu_\infty \equiv 1$) and is a microscopic counterpart to the proof of the quasineutral limit for Vlasov-Poisson with monokinetic data \cite{BCGM2015}.

The sharp first-order commutator estimate \eqref{main1} plays an essential role to show convergence provided $N^{\frac{\s}{\d}-1}/\vep^2 \rightarrow 0$ as $\vep\rightarrow 0$ and $N\rightarrow\infty$. This result is sharp, as when $N^{\frac{\s}{\d}-1}/\vep^2 \not\rightarrow 0$, we show that convergence in the modulated energy distance may fail.

\subsection{Organization of article}\label{ssec:introOrg}
In \cref{sec:ME}, we review the potential truncation/charge smearing renormalization procedure and the modulated energy, proving some Riesz generalizations of results previously known for the Coulomb case.

In \cref{sec:FO}, we give a self-contained treatment of first-order commutators, discussing the relationship between the commutator $\ka^{(1),f}$ and the stress-energy tensor and proving general $L^2$ estimates (\cref{ssec:FOcst}), then using these results and the renormalization procedure to prove the estimate \eqref{main1} of \cref{thm:FI} (\cref{ssec:FOren}).

In \cref{sec:comm}, we turn to $L^2$ and $L^\infty$-based estimates for the higher-order commutators $\ka^{(n),f}$. The main results of this section are \cref{thm:comm2} for the $L^2$ estimates and \cref{th44} for the $L^\infty$ estimates. We start by proving the identity \eqref{eq:introLkanf} for $L\ka^{(n),f}$ (\cref{ssec:L2regeqn}). Next, we turn to the proof of \cref{thm:comm2}, which is divided into general localized estimates corresponding to \eqref{commuordern} (\cref{ssec:L2regloc}), an improved (localized) estimate for $n=2$ (\cref{ssec:L2regimp}), and general unlocalized estimates (\cref{ssec:L2regunloc}). With the proof of \cref{thm:comm2} complete, we consider \cref{th44}. We first prove general $\dot{C}^{k,\alpha}$ estimates for the operator $L$ in $\R^{\d+\k}$ (\cref{ssec:L2regLinf}), and then we apply these estimates together with \cref{thm:comm2} to establish a regularity theory $\ka^{(n),f}$, thereby completing the proof of \cref{th44} (\cref{ssec:L2regind}).

In \cref{sec:FI}, we take up the proof of our main functional inequalities result \cref{thm:FI} by combining the regularity theory of \cref{sec:L2reg} with the charge smearing/potential truncation of \cref{sec:ME}. As the $n=1$ case was previously treated in \cref{sec:FO}, this section focuses on the $n\ge 2$ case. We first introduce the second regularization $\rho_{x_i}^{(\eta_i)}$ for the Dirac $\delta_{x_i}$ that is supported on $\R^\d\times\{0\}^\k$ and show that the electric energy associated to this regularization is controlled by the modulated energy (\cref{ssec:FIpre}).  We then turn to the main proof of the estimate \eqref{mainn}, which then completes the proof of \cref{thm:FI} (\cref{ssec:FImain}). Next, we give an alternative to \eqref{mainn} in the form of \cref{prop:impHOFI}, which has the same additive error when $\ell$ is macroscopic but weaker regularity demands on $v$ (\cref{ssec:FImac}). {Finally, we present the proof of \cref{thm:mainunloc}, i.e.~the analogue of \eqref{mainn} in the case $\ell=\infty$ (\cref{ssec:FIunloc}).}

In \cref{sec:appMF}, we close the main body of the paper with the aforementioned application of \cref{thm:FI} to the  optimal rate of convergence for first-order mean-field dynamics, proving \cref{thm:mainMF}.

Lastly, in \cref{appA}, we give an alternative proof of $L^2$ estimates for the commutators in the Coulomb case, based on the variation by transport argument described in \cref{ssec:intropf}.

\subsection{Acknowledgments}
The first author thanks the Institute for Computational and Experimental Research in Mathematics (ICERM) for its hospitality, where part of the research for this project was carried out during the Fall 2021 semester program ``Hamiltonian Methods in Dispersive and Wave Evolution Equations.'' He also thanks the Courant Institute of Mathematical Sciences at NYU for their hospitality during his visit in April 2024.

Both authors thank Elias Hess-Childs for his careful reading of an earlier version of the manuscript.

\subsection{Notation}\label{ssec:preN}
We close the introduction with the basic notation used throughout the article without further comment. We mostly follow the conventions of \cite{NRS2021,RS2021}.

Given nonnegative quantities $A$ and $B$, we write $A\lesssim B$ if there exists a constant $C>0$, independent of $A$ and $B$, such that $A\leq CB$. If $A \lesssim B$ and $B\lesssim A$, we write $A\sim B$. Throughout this paper, $C$ will be used to denote a generic constant which may change from line to line. Also $(\cdot)_+$ denotes the positive part of a number.

$\N$ denotes the natural numbers excluding zero, and $\N_0$ including zero. {For $N\in\N$, we abbreviate $[N]\coloneqq \{1,\ldots,N\}$.} $\R_+$ denotes the positive reals. Given $x\in\R^\d$ and $r>0$, $B(x,r)$ and $\p B(x,r)$ respectively denote the ball and sphere centered at $x$ of radius $r$. Given a function $f$, we denote its support by $\supp f$. The notation $\nabla^{\otimes k}f$ denotes the $k$-tensor field with components $(\p_{i_1}\cdots\p_{i_k}f)_{1\leq i_1,\ldots,i_k\leq \d}$.

$\P(\R^\d)$ denotes the space of Borel probability measures on $\R^\d$. If $\mu$ is absolutely continuous with respect to Lebesgue measure, we shall abuse notation by writing $\mu$ for both the measure and its density function. When the measure is clearly understood to be Lebesgue, we shall simply write $\int_{\R^{\d}}f$ or $\int_{\R^{\d+\k}}f$ instead of $\int_{\R^\d}fdx$ or $\int_{\R^{\d+\k}}fdxdz$.

{
$C^{k,\alpha}(\R^\d)$ denotes the inhomogeneous space of $k$-times differentiable functions on $\R^\d$ whose $k$-th derivative is $\alpha$-H\"{o}lder continuous, for $\al\in [0,1]$ (i.e. $\alpha=0$ is bounded and $\alpha=1$ is Lipschitz). As per convention, a $\dot{}$ superscript denotes the homogeneous space/seminorm.} $\Sc(\R^\d)$ and $\Sc'(\R^\d)$ denotes the space of Schwartz functions and the space of tempered distributions, respectively.

\section{The modulated energy}\label{sec:ME}
In this section, we review properties of the modulated energy functional introduced in \eqref{eq:modenergy}.

\subsection{Electric formulation}\label{ssec:MEts}
We begin by reviewing the procedures for truncating the interaction potential $\g$ and smearing the Dirac masses $\delta_{x_i}$, which was introduced by the second author in \cite{Serfaty2020}, building on \cite{PS2017}.

When $0<\s<\d$, then it is immediate from basic potential theory (e.g. see \cite[Remark 2.5]{RS2021}) that $\g\ast\mu$ is a bounded, continuous function (it is actually $C^{k,\alpha}$ for some $k\in\N_0$ and $\alpha>0$ depending on the value of $\s$) and therefore the modulated energy is well-defined. If $\s\le 0$, then we need to impose a suitable decay assumption on $\mu$ to compensate for the logarithmic growth of $\g$ at infinity. {The energy condition $\int_{(\R^\d)^2}\g(x-y)d|\mu|^{\otimes 2}(x,y)<\infty$ from the statement of \cref{thm:FI} suffices.}


We note that in the super-Coulomb case $\s>\d-2$, the potential $\g$ fails to be superharmonic. However, as observed in \cite{NRS2021}, superharmonicity is restored if we consider $\g(x)$ as the restriction $\G(x,0)$ of a potential defined in a larger space. To this end, we introduce the space extension $\R^{\d+\k}$, with $\k=0$ if $\s=\d-2$, and $\k=1$ otherwise, and 
\begin{equation}\label{eq:defgamma}
\gamma\coloneqq \s+ 2-\d-\k.
\end{equation}
It is straightforward to check that $\ga\in (-1,1)$.  With an abuse of notation, given $x_i\in \R^\d$, we also denote $x_i = (x_i,0) \in\R^{\d+\k}$, for $\k=1$. We naturally extend $\g$ into a function on $\R^{\d+\k}$ by setting
\begin{equation}
\G:\R^{\d+\k}\setminus\{0\} \rightarrow \R_+, \qquad \G(x,z) \coloneqq \g(|(x,z)|),
\end{equation}
where here and throughout this paper, we use the radial symmetry of $\g$ with an abuse of notation. We will also use the notation $L^2_\gamma$ for the weighted $L^2$ space $L^2(\R^{\d+\k},|z|^\ga dxdz)$, {that is
\begin{align}\label{eq:Lga2}
\|f\|_{L_\ga^2}^2 \coloneqq \int_{\R^{\d+\k}} |z|^\ga |f|^2 dxdz.
\end{align}}
For the record, we note that $|z|^\ga$ is an $A_2$ weight (e.g.~see \cite[Example 7.17]{Grafakos2014c}), a property extensively used later.

Note that any distribution $\mu$ on $\R^{\d}$ may be canonically extended to a distribution $\tilde \mu \coloneqq \mu \delta_{\R^\d\times \{0\}^\k}$ on $\R^{\d+\k}$, which acts on a test function $\phi:\R^{\d+\k}\rightarrow\R$ according to
\begin{equation}
\int_{\R^{\d+\k}}\phi \, d\tl\mu = \int_{\R^\d}\phi(x,0)d\mu(x).
\end{equation}
Equivalently, if $\iota:\R^\d\rightarrow\R^{\d+\k}$ is the trivial embedding $x\mapsto (x,0)$, then $\tl\mu = \iota\#\mu$. 

Following previous work, we truncate the potential $\g$ as follows. Given a parameter $\eta>0$, let 
\begin{equation}\label{eq:defGeta}
\g_\eta\coloneqq \min (\g, \g(\eta)), \qquad \f_\eta \coloneqq \g-\g_\eta,
\end{equation}
We note that $\f_\eta$ is supported in the ball $B(0, \eta)$ of $\R^\d$, respectively $\R^{\d+\k}$ and that
\begin{equation}
\f_\eta(x) = \begin{cases} \left(-\log\Big(\frac{|x|}{\eta}\Big)\right)_+, & {\s=0} \\ \eta^{-\s}\left(\frac{\eta^{\s}-|x|^\s}{|x|^\s}\right)_+, & {\s>0}. \end{cases}
\end{equation}
In particular, $\f_\eta\ge 0$.  
 We will frequently use the bounds
\begin{equation}\label{eq:intf}
\int_{\R^\d} |\nabla^{\otimes k}\f_\eta| \lesssim \eta^{\d-\s-k}, \qquad \s+k<\d.
\end{equation}

An important property of the Riesz case $\d-2\le \s<\d$ \cite{CS2007}, which motivates the dimension extension, is that that $\g$ is, up to normalization, a fundamental solution for the degenerate \emph{local} elliptic operator 
\begin{align}\label{defL}
L\coloneqq -\div\Big(\zg\nab(\cdot)\Big)
\end{align}
in $\R^{\d+\k}$, i.e.
\begin{equation}\label{eq:Gfs}
L\g = \cd\delta_{0} \qquad \text{in} \ \R^{\d+\k},
\end{equation}
with equality in the sense of distributions. {Note that this constant $\mathsf{c}_{\d,\s}$ does not coincide  in general with the constant $\mathsf{c}_{\d,\s}'$ in the fractional Laplacian identity $(-\Delta)^{\frac{\d-\s}{2}}\g = \mathsf{c}_{\d,\s}'\delta_0$.} This implies that {\it if $f$ is a distribution supported on} $\R^\d$, letting  $h^f= \G *f$, we have 
\begin{equation}
\label{relhf}
Lh^f= \cd\tilde f.
\end{equation}

The dimension extension and fundamental solution property leads us to define a smearing of the Dirac mass in $\R^{\d+\k}$ at scale $\eta$ by
\begin{equation}\label{defdeta}
\delta_0^{(\eta)}\coloneqq  \frac{1}{\cd}L\g_\eta, \qquad \delta_x^{(\eta)} \coloneqq \delta_0^{(\eta)}(\cdot-x), \quad x \in\R^\d.
\end{equation}
As shown in \cite[Section 1.3]{PS2017}, this distribution is, in fact, a probability measure supported on $\partial B(0, \eta) \subset \R^{\d+\k}$, and its density with respect to the uniform probability measure on the sphere $\partial B(0, \eta)$ is $\frac{1}{\cd} \frac{|z|^\gamma}{|(x,z)|^{\s+1}}$. 

Given a pairwise distinct configuration of points $\ux_N=(x_1, \dots, x_N)\in(\R^\d)^N$, we define the potential
\begin{equation}
h_N\coloneqq \G\ast\pa*{\frac{1}{N}\sum_{i=1}^N \delta_{x_i}- \mu}.
\end{equation}
Strictly speaking, we should really use the notation $h_N[\ux_N]$ to denote the dependence on the configuration of points, but we choose not to for the sake of lightness of notation. We now introduce a truncated version: given a vector of parameters $\vec{\eta}\in (\R_+)^N$,
we let 
\begin{equation}\label{eq:defHNeta}
h_{N,\vec{\eta}}\coloneqq \frac{1}{N}\sum_{i=1}^N \G_{\eta_i} (\cdot-x_i) - \G\ast \mu,
\end{equation}
regarded as a function on $\R^{\d+\k}$. We also let 
\begin{equation}\label{defhni}
h_N^i\coloneqq h_N-\frac1N\G(\cdot-x_i),
\end{equation}
and we observe in view of \eqref{eq:defGeta} that if the balls $\{B(x_i, \eta_i)\}_{i=1}^N$ are pairwise disjoint, then
\begin{equation}\label{eq:carHneta}
\nabla h_{N,\vec{\eta}}= \begin{cases} \nabla h_N, & {\text{outside} \ \cup_{i=1}^N B(x_i, \eta_i)}\\
\nabla h_N^i, & {\text{in}  \ B(x_i, \eta_i)}.\end{cases}
\end{equation}
From this property, it follows that
\begin{equation}\label{eq:energballs}
\int_{\R^{\d+\k} } |z|^\gamma |\nabla h_{N, \vec{\eta}}|^2 = \int_{\R^{\d+\k} \setminus \cup_i B(x_i, \eta_i)} |z|^\gamma
|\nabla h_{N} |^2 +\sum_{i=1}^N \int_{B(x_i, \eta_i)} |z|^\gamma |\nabla h_N^i|^2.
\end{equation}

Finally, for each $1\leq i\leq N$, we define the nearest-neighbor type distance, 
\begin{equation}\label{eq:defri}
\rs_i \coloneqq \frac{1}{4}\min\Big(\min_{1\leq i\neq j\leq N} |x_i-x_j|,  (N\|\mu\|_{L^\infty})^{-1/d}\Big).
\end{equation}
The balls $B(x_i, \rs_i)$  are tautologically disjoint. Note that in the definition of $\rs_i$, we could replace $\|\mu\|_{L^\infty}$ with any constant larger than it.

One of the improvements in the present paper compared to prior works is that we will carefully track the dependence of the estimates on the density bound $\|\mu\|_{L^\infty}$. Instead of the microscale $N^{-1/\d}$, we will see that a natural lengthscale is actually $(N\|\mu\|_{L^\infty})^{-1/\d}$, which takes into account the density of points.

We also define the localized  lengthscale
\begin{equation}\label{deflamb}
\lambda \coloneqq ( N \|\mu\|_{L^\infty(\Omega)})^{-1/\d},
\end{equation} 
which of course depends on $\Omega$, $\mu$, and $N$. {In the localized case, we replace $(N\|\mu\|_{L^\infty})^{-1/\d}$ by $\lambda$.}

\subsection{Monotonicity and local energy control}\label{ssec:MEloc}
We recall from the introduction the modulated energy \eqref{eq:modenergy}
and state the electric rewriting of this quantity.

\begin{prop}\label{proiden}
For any $\ux_N$ and any $\vec\eta$ as above, such that $\eta_i\le \rs_i$ for each $i$, we have
\begin{equation}\label{iden}
\Fr_N(\ux_N, \mu)= \frac{1}{2\cd}\Bigg(\int_{\R^{\d+\k}} |z|^\gamma |\nabla h_{N,\vec{\eta}} |^2 - \frac{\cd}{N^2} \sum_{i=1}^N \g(\eta_i)\Bigg)- \frac{1}{N} \sum_{i=1}^N \int_{\R^\d} \f_{\eta_i}(x-x_i) d \mu(x),
\end{equation}
where $\f_{\eta}$ is as in \eqref{eq:defGeta}.
\end{prop}
In \cite{PS2017} or \cite{Serfaty2020}, it is proven (by computations  based on \eqref{eq:defHNeta} and integrations by parts)  that the left-hand side is equal to the limit of the  right-hand side as all the $\eta_i$'s tend to $0$. \cref{lem:monoto} below (the equality case) then implies that the right-hand side is constant in the parameters $\eta_i$, provided that $\eta_i\le \rs_i$ for each $i$. 
We now state the monotonicity result from \cite[Lemma 3.2]{Serfaty2023}, generalized to the Riesz case. This result shows that the energy decreases when increasing the truncation radii. However, when working in a sudomain $\Omega$ of the whole space, this is only true when the truncation radii are not changed for balls that intersect the boundary of $\Omega$. {This explains the reason why control is  only obtained for points well inside $\Omega$.
\begin{lemma}\label{lem:monoto}
Assume that $\alpha_i \le \eta_i$ for each $1\leq i\leq N$. Then 
\begin{multline}\label{eq:premono0}
\int_{\R^{\d+\k}} \zg |\nabla h_{N,\vec{\eta}}|^2 - \frac{\cd}{N^2} \sum_{i=1}^N \g(\eta_i) -\frac{2\cd}{N}  \sum_{i=1}^N \int_{\R^\d} \f_{\eta_i}(x-x_i)d\mu\\- \Bigg( \int_{\R^{\d+\k}} \zg |\nabla h_{N,\vec{\alpha}}|^2 - \frac{\cd}{N^2} \sum_{i=1}^{N} \g(\alpha_i) -\frac{2\cd}{N} \sum_{i=1}^N \int_{\R^\d} \f_{\alpha_i}(x-x_i)d\mu(x)\Bigg) \le 0,
\end{multline}
with equality if $\eta_i \le \rs_i$ for each $i$. 

For $\Omega \subset \R^\d $, let us denote
\begin{equation}\label{eq:defFa}
\mathcal F^{\vec{\alpha}}\coloneqq\frac1{2\cd}\Bigg(\int_{\Omega\times \R^\k} \zg |\nabla h_{N,\vec{\alpha}}|^2  -  \frac{\cd}{N^2} \sum_{i\in I_{\Omega}}  \g(\alpha_i)-\frac{2\cd}{N} \sum_{i\in I_{\Omega}}\int_{\R^\d} \f_{\alpha_i}(x-x_i)d\mu(x)\Bigg).
\end{equation} 
Let $\vec{\eta}$, $\vec{\alpha}$, $\vec{\zeta}$  be such that $\zeta_i\le \eta_i\le \alpha_i$ for all $i$, with  $\alpha_i=\eta_i$ if  $\dist(x_i, \p\Omega) \le\alpha_i $ and $\zeta_i=\eta_i$ if $\eta_i>\rs_i$ or if $\dist(x_i, \p \Omega)\le \eta_i$,  then
 we have 
 \begin{itemize}
\item
\begin{equation}\label{eq:supplem0}  \mathcal F^{\vec{\eta}} -  \mathcal F^{\vec{\alpha}}\ge
\frac1{2N^2}\sum_{\substack{ i, j\in I_\Omega, i\neq j \\ \dist(x_i, \partial \Omega) \ge \alpha_i }} 
\left(\g_{\zeta_i}(|x_i-x_j|+\zeta_j) - \g(\alpha_i)\right)_+ .
\end{equation}
\item if $\alpha \ge \frac14\lambda$, 
\begin{equation}\label{eq:controlmic3}
\begin{cases}
\displaystyle \frac{1}{N^2}\sum_{\substack{ i\neq j : x_i , x_j \in \Omega,  \dist(x_i, \p\Omega) \ge 4\alpha,\\ \alpha \le  |x_i-x_j|\le 2 \alpha}}\g(\alpha) \le C \pa*{\mathcal F^{\vec{\zeta}} - \mathcal F^{\vec{\eta}}}, & \text{if} \  \s>0\\
\displaystyle\frac{1}{N^2} \sum_{\substack{ i\neq j : x_i , x_j \in \Omega, \dist(x_i, \p\Omega) \ge 4\alpha ,\\ \alpha \le  |x_i-x_j|\le 2 \alpha}} 1  \le C \pa*{\mathcal F^{\vec{\zeta}} - \mathcal F^{ \vec{\eta}}}, & \text{if}\  \s=0,\end{cases}
\end{equation}
where  $\vec{\eta}=(\eta_1,\ldots,\eta_N)$ and $\vec{\zeta} = (\zeta_1,\ldots,\zeta_N)$ are chosen according to
\begin{equation}\label{228}
\zeta_i = \begin{cases} \al, & {\dist (x_i,\p\Om) \geq \al} \\ \rs_i, & {\text{otherwise}} \end{cases} \qquad \eta_i = \begin{cases} 4\al, & {\dist(x_i,\p\Om)\geq 4\al} \\ \rs_i, & {\text{otherwise}}, \end{cases}
\end{equation}
and $C>0$ is a constant depending only on $\d,\s$.
\end{itemize}
\end{lemma}
\begin{proof}
The relation \eqref{eq:premono0} is proven for instance in \cite[Proof of Lemma B.1]{AS2021} in the Coulomb case. That proof originates in \cite{PS2017}, where the Riesz case is treated.
We follow here very closely \cite[Lemma B.1]{AS2021},  which can be copied with no changes other than the inclusion of the weight $\zg$ and the rescaling. We sketch the main steps.

For $\zeta\le \eta$, we observe  that $\g(\zeta)-\g(\eta)\ge  \G_\zeta-\G_\eta \ge 0 $ and $\G_{\zeta}- \G_{ \eta}$ vanishes outside $B(0,\eta)$. Writing $h_{N, \vec{\eta}}= h_{N, \vec{\zeta}} + \frac1N\sum_{i=1}^N (\G_{\eta_i}- \G_{\zeta_i} ) (\cdot-x_i)$, expanding and using integration by parts and \eqref{defdeta}, we obtain, as in \cite{AS2021}, that if $B(x_i, \eta_i)\subset \Omega$ for all the $i$'s such that $\eta_i\neq \zeta_i$, we have
\begin{equation}\label{above}
\mathcal F^{\vec{\zeta}}-\mathcal F^{\vec{\eta}}= \frac1{2N^2}   \sum_{i\neq j}  \int_{\Omega \times \R^\k}(  \G_{ \zeta_i}-\G_{\eta_i} )  (\cdot-x_i) d\left( \delta_{x_j}^{(\zeta_j)} + \delta_{x_j}^{(\eta_j)} \right).
\end{equation}
Since $(\G_{\zeta_i}-\G_{\eta_i}) (\cdot -x_i)$ is nonnegative and supported only in $B(x_i, \eta_i)$, we find that $\mathcal F^{\vec{\zeta}}-\mathcal F^{\vec{\eta}} \ge 0$ with equality if the $B(x_i, \eta_i)$ are disjoint, which proves \eqref{eq:premono0} and the sentence that follows. 
Moreover,  the couples $i\neq j$ that contribute to the sum are those for which $\zeta_i\neq \eta_i$ (hence there is no contribution for points that do not satisfy $B(x_i, \eta_i) \subset \Omega$)  and $B(x_j,\eta_j)$ intersects $B(x_i, \eta_i)$.

In particular, if $\eta_i\le \rs_i$ for $i$ such that $\eta_i\neq \zeta_i$, then the growing balls are disjoint, and 
we find that the right-hand side of \eqref{above} vanishes hence 
\begin{equation}\label{sft0}
\mathcal F^{\vec{\eta}}= \mathcal F^{\vec{\zeta}} .\end{equation}

Applying  now \eqref{above} to $\vec{\alpha}$ and $\vec{\zeta}$ such that $\zeta_i\le \alpha_i$ with $\zeta_i=\alpha_i$ if  $B(x_i,\alpha_i)  \not\subset \Omega$ and
using the monotonicity of $\g$ and definitions of $\g_{\zeta_i}, \g_{\eta_i}$, we find from \eqref{above} that
 \begin{equation}\label{4.5.9}
\mathcal F^{\vec{\zeta}}- \mathcal F^{\vec{\alpha}} \ge 
\frac1{2N^2} \sum_{\substack{i,j\in I_\Omega :   i\neq j \\ { B(x_i, \alpha_i) \subset \Omega } } } \left(  \g_{\zeta_i}(|x_i-x_j|+\zeta_j)-\g(\alpha_i)\right)_+ .
\end{equation}
Combining \eqref{sft0} and \eqref{4.5.9}, \eqref{eq:supplem0} follows. 

Finally, \eqref{eq:controlmic3} is obtained by applying \eqref{above} to $\vec{\alpha}= \vec{\zeta}$ with the choice \eqref{228} and   using that $\g(3x) = 3^{-\s}\g(x)$ if $\s>0$ and $\g(3x) = \g(x)+\g(3)$ if $\s=0$.
\end{proof}

Next, we state a Riesz-case generalization and improvement of \cite[Corollary 3.4]{Serfaty2023}, which gives a control of the minimal distance interactions.  The improvement is in the $\|\mu\|_{L^\infty}$ dependence and the removal of the nonnegativity assumption on $\mu$.

\begin{prop}\label{prop:MElb}
Assume $\s\in [(\d-2)_+, \d)$.
Let $\mu \in L^1(\R^\d)\cap L^\infty(\R^\d)$ with $\int_{\R^\d}\mu=1$, and let  $\ux_N \in (\R^\d)^N$ be a pairwise distinct configuration.
Let $\Omega \subset \R^\d$ and denote $\hat \Omega \coloneqq \{x\in \R^\d : \dist(x, \Omega) \le   \frac14 \lambda\}$.

For any $\vec{\eta}$ satisfying $ \frac{1}{2} \rs_i \le \eta_i \le \rs_i$ for every $1\leq i\leq N$, it holds that 
\begin{multline}\label{eq:13}
 \frac{1}{\cd}\int_{\Omega \times [-\la, \la]^\k}  \zg|\nab h_{N,\vec{\eta}}|^2 +\left(\frac{\#I_\Omega \log( N\|\mu\|_{L^\infty(\Omega)}) }{N^2\d}\right) \indic_{\s=0} +C \#I_\Omega  N^{-2+\frac{\s}{\d}}\|\mu\|_{L^\infty(\hat{\Omega}) }  \|\mu\|_{L^\infty(\Omega)}^{-1+\frac{\s}{\d}}\\
 \ge 
\begin{cases} 
\displaystyle \frac1{ CN^2}\sum_{i: x_i\in \Omega, \dist(x_i, \p \Omega) \ge \frac14\lambda}\g(\eta_i) & \text{if} \ \s\neq 0\\
\displaystyle \frac{1}{N^2}\sum_{i: x_i\in  \Omega , \dist(x_i, \p \Omega) \ge \frac14\lambda}\g(\eta_i /\lambda ) & \text{if} \ \s=0,\end{cases}
\end{multline}
where $C>0$ depends only on $\d$ and $\s$.

If $\Omega=\R^\d$, then  we also have that 
\begin{equation}\label{eq:pr1}
2 \left(\Fr_N(\ux_N, \mu)+ \frac{ \log (N \|\mu\|_{L^\infty})  }{{2}N \d}  
\indic_{\s=0}\right) + C \|\mu\|_{L^\infty}^{\frac{\s}{\d}} N^{-1+\frac{\s}{\d}} \ge \begin{cases} 
\displaystyle \frac1{ CN^2}\sum_{i=1}^N\g(\eta_i) & \text{if} \ \s\neq 0\\
\displaystyle \frac1{N^2}\sum_{i=1}^N\g(\eta_i /\lambda ) & \text{if} \ \s=0,\end{cases}
\end{equation}
and
\begin{equation}\label{eq:pr2}
\int_{\R^{\d+\k}} |z|^\gamma |\nabla h_{N,\vec{\eta}}|^2 \leq C\pa*{\Fr_N(\ux_N, \mu) +\frac{\log (N\|\mu\|_{L^\infty}) }{2 N\d} \indic_{\s=0}} + C \|\mu\|_{L^\infty}^{\frac{\s}{\d}} N^{-1+\frac{\s}{\d}} .
\end{equation}
\end{prop}
\begin{proof}
First, we remark that we may rewrite $\Fc^{\vec{\alpha}}-\Fc^{\vec{\eta}}$ by restricting integrals to $\Omega \times [-\la,\la]^\k$. Indeed, if  $\alpha_i, \eta_i\le \lambda$, then  $\nab h_{N,\vec{\alpha}}$ and $\nab h_{N,\vec{\eta}}$ by definition coincide outside $\R^\d\times [-\la,\la]^\k$. Hence,  
\begin{align}
\int_{\Omega\times \R^\k} \zg|\nab h_{N,\vec{\eta}}|^2 - \int_{\Omega\times \R^\k} \zg|\nab h_{N,\vec{\alpha}}|^2 =\int_{\Omega\times [-\la,\la]^\k} \zg|\nab h_{N,\vec{\eta}}|^2 - \int_{\Omega\times [-\la,\la]^\k} \zg|\nab h_{N,\vec{\alpha}}|^2.
\end{align}
Let us now choose $\alpha_i=\eta_i$ if $\dist(x_i, \p \Omega)< \frac14\lambda$, $\alpha_i=\frac14\lambda$ otherwise, and $\zeta_i=\eta_i$. We   obtain from \eqref{eq:supplem0}, discarding a nonnegative term,  that 
\begin{multline}\label{236}\frac1{2\cd}\int_{\Omega \times [-\la,\la]^\k} \zg|\nab h_{N,\vec{\eta}}|^2 \\ \ge -\frac{1}{2N^2} \sum_{\substack{i:x_i\in\Omega\\ \alpha_i\neq \eta_i}} \g(\alpha_i)+\frac{1}{2N^2}\sum_{\substack{i,j: x_i, x_j\in\Omega, i\neq j\\ \dist(x_i, \p \Omega) \ge \alpha_i}} \left(\g_{\eta_i}(|x_i-x_j| +\eta_j) -\g(\alpha_i)\right)_+-\frac{C}{N}\|\mu\|_{L^\infty(\hat \Omega)}\sum_{i:x_i\in\Omega} \|\f_{\alpha_i}\|_{L^1}.
\end{multline}

Let $i$ be such that $\dist(x_i,\partial\Omega) \ge \frac\lambda4$. Let us first consider the situation where  $\min_{j\neq i}|x_i-x_j|>\lambda$.  In that case, by definition of $\rs_i$, we have $4\rs_i=\lambda>\frac14\lambda$, and
\begin{equation}
\forall j\neq i, \qquad \left(\g_{\rs_i}(|x_i-x_j|+\rs_j)-\g(\frac\lambda4)\right)_+ = 0>\g(4\rs_i)-\g(\frac\lambda4).
\end{equation}
Next consider the case $\min_{j\neq i} |x_i-x_j|\le \lambda$ and let $j$ achieve the minimum, so that $|x_i-x_j|\le 4\rs_i$.
We then also have $\rs_j\le \rs_i$.
 In all cases,  we can thus assert that there exists $j\neq i$ such that, by definition and monotonicity of $\g_{\rs_i}$,  and the fact that $\frac12 \rs_i\le \eta_i\le \rs_i$ and $\eta_j\le \rs_j$,
\begin{equation}
\left(\g_{\eta_i}(|x_i-x_j| +\eta_j) -\g(\alpha_i)\right)_+\ge\left(\g_{\rs_i}(|x_i-x_j|+\rs_j) - \g(\frac{\la}{4})\right)_+  \geq \g(5\rs_i)-\g(\frac\lambda4) \ge \g(10\eta_i) - \g(\frac\lambda4).
\end{equation}
Keeping only that $j$ in the sum, 
 it follows from \eqref{236}, in view of \eqref{eq:intf} applied to the factor $\|\f_{\alpha_i}\|_{L^1}$ from \eqref{236}, that if $\s\ne 0$,
\begin{multline}
 \frac1{2N^2}\sum_{\substack{i: x_i\in \Omega\\ \dist(x_i, \p\Omega) \ge \frac14\lambda}} \g( 10\eta_i)\le  
 \frac1{2\cd}\int_{\Omega \times [-\la,\la]^\k} \zg|\nab h_{N,\vec{\eta}}|^2 \\ +\frac{1}{2N^2} \sum_{\substack{i:x_i\in\Omega\\ \alpha_i\neq \eta_i}} \g(\alpha_i)
 + \frac{1}{2N^2} \sum_{\substack{i: x_i\in \Omega\\ \dist(x_i, \p\Omega) \ge \frac14\lambda}} \g(\frac\lambda4)
+ \frac{C}{N}\|\mu\|_{L^\infty(\hat\Omega)}\#I_\Omega  \lambda^{\d-\s}.
\end{multline}
Noting that $\al_i\ne \eta_i$ implies $\dist(x_i,\p\Omega)\le \la/4$ and $\al_i=\la/4$, 
in view of the definition of $\g $ and  $\lambda$ in \eqref{deflamb}, the result \eqref{eq:13} follows. If $\s=0$, then starting from \eqref{236} and proceeding by the same reasoning, we instead arrive at
\begin{multline}
 \frac1{2N^2}\sum_{\substack{i: x_i\in \Omega\\ \dist(x_i, \p\Omega) \ge \frac14\lambda}} \g( 40\eta_i/\la)\le  
 \frac1{2\cd}\int_{\Omega \times [-\la,\la]^\k} \zg|\nab h_{N,\vec{\eta}}|^2 \\
 + \frac{1}{2N^2} \sum_{\substack{i: x_i\in \Omega\\ \al_i\ne \eta_i}} \g(\al_i)
+ \frac{C}{N}\|\mu\|_{L^\infty(\hat\Omega)}\#I_\Omega  \lambda^{\d-\s},
\end{multline}
and again the result \eqref{eq:13} follows in view of the choice of $\alpha_i$. Note that we can absorb the numerical factors $\g(a)$, for $a<1$, into the factor $N^{\frac{\s}{\d}}\|\mu\|_{L^\infty(\hat\Omega)} \|\mu\|_{L^\infty(\Omega)}^{-1+\frac{\s}{\d}}$ (up to an absolute constant) because $\|\mu\|_{L^\infty(\hat\Omega)}\ge \|\mu\|_{L^\infty(\Omega)}$ and $N^{\frac{\s}{\d}}\|\mu\|_{L^\infty(\Omega)}^{\frac{\s}{\d}} = (1/\la)^{\frac{\s}{\d}} > 1$, by our assumption that $\la<1$. 

To prove \eqref{eq:pr1}, we apply 
\eqref{eq:supplem0} with $\zeta_i=0$, $\eta_i$ and $\alpha_i=\lambda$. Then $\mathcal{F}^{\vec{\eta}}= \Fr_N(\ux_N, \mu)$  in view of \eqref{iden}, and $\mathcal{F}^{\vec{\alpha}} \ge - \frac{1}{2N}\g(\lambda) - C \|\mu\|_{L^\infty} \lambda^{\d-\s}$ after discarding the nonnegative term  and using \eqref{eq:intf} and $\eta_i\le \rs_i\le  \lambda$.
We thus obtain that 
\begin{equation}\Fr_N(\ux_N, \mu) + \frac{\g(\lambda) }{2N}+ C \|\mu\|_{L^\infty} \lambda^{\d-\s}\ge \frac{1}{2N^2} \sum_{i=1}^N \sum_{j\neq i}\left(\g(x_i-x_j) -\g(\lambda)\right)_+ \ge\frac{1}{2N^2} \sum_{i=1}^N \g(4\rs_i)-\g(\lambda)  \end{equation}
Using that $\rs_i \le 2 \eta_i$, \eqref{eq:pr1} follows. The relation \eqref{eq:pr2} then follows from it and \eqref{iden}.
\end{proof} 
}

\subsection{Mesoscale interaction control}\label{ssec:MEmulti}
We now present an  application of the mesoscopic interaction energy control of \eqref{eq:supplem0} and \eqref{eq:controlmic3}, which allows, by combining the estimates obtained over dyadic scales, to control general inverse powers of the distances between the points. \cref{prop:multiscale} stated below, which is a generalization of a Coulomb-specific result from \cite[Proposition 3.5]{Serfaty2020}, is to be combined with \cref{cor:coro3} in order to estimate the interaction of microscopically close points.  This proposition will only be needed for the higher-order commutator estimates $n\ge 2$.

First, we have a microscale interaction control, following from the relation \eqref{eq:supplem0} applied with $(\zeta_i,\al_i,\eta_i) = (\rs_i,\rs_i,\rs_i)$ if   $\dist(x_i, \p \Omega)\le \lambda$, $(\zeta_i, \al_i,\eta_i) = (0,\la,\rs_i)$ otherwise
and bounding $\Fc^{\vec{\alpha}}$ from below similarly as to in \eqref{236}.

\begin{cor}[Microscale control]\label{cor:coro3}
Under the same assumptions as \cref{prop:MElb}, we have
\begin{equation}
\frac{1}{N^2} \sum_{\substack{i,j \in I_{\Omega}:  i\neq j,\\ \dist(x_i,\p\Omega) \ge { \lambda/4},\\
 |x_i-x_j|\le \lambda }}  \g(x_i-x_j) \le C{\int_{\Omega\times [-\la,\la]^\k}\zg |\nab h_{N,\vec\rs}|^2 +C \frac{\# I_\Omega}{N} \la^{\d-\s}  \|\mu\|_{L^\infty(\hat\Omega)},  \quad \text{if} \ \s>0}
\end{equation}
{
and
\begin{multline}
 \frac{1}{N^2} \sum_{\substack{i,j \in I_{\Omega}:  i\neq j,  \\\dist(x_i,\p\Omega) \ge  \lambda/4,\\  |x_i-x_j|\le \lambda  }}\g\left( |x_i-x_j|/\lambda \right) \le C\left(\int_{\Omega\times [-\la,\la]^\k}\zg |\nab h_{N,\vec\rs}|^2 + \left(\frac{\#I_\Omega \log( N\|\mu\|_{L^\infty(\Omega)}) }{N^2 \d}\right) \indic_{\s=0}  \right) \\
 +C \frac{\# I_\Omega}{N} \la^{\d-\s}  \|\mu\|_{L^\infty(\hat\Omega)}, \quad \text{if} \  \s=0.
\end{multline}
}
\end{cor}

We next turn to the larger scales.
\begin{prop}[Mesoscale  control]\label{prop:multiscale}
Let $a\ge 0$ and $\ell \ge \lambda $. We have
\begin{multline}\label{eq:resmultiscale}
\frac1{N^2}\sum_{\substack{i,j \in I_{\Omega} : i\ne j,\\ \lambda \le |x_i-x_j|\le \ell,  \dist(x_i,\p \Omega) \ge4\ell}} \frac{1}{|x_i-x_j|^{\s+a}} \le C\lambda ^{-a}
{\int_{\Omega \times [-\ell,\ell]^\k}  \zg|\nab h_{N,\vec{\eta}}|^2 
+ C\frac{\#I_\Omega}{N} \|\mu\|_{L^\infty(\hat\Omega)}\la^{\d-\s-a}  } \\
 +C\frac{\#I_\Omega \|\mu\|_{L^\infty(\hat {\Omega})}}{N}\begin{cases}  \ell^{\d-\s-a}, & \text{if} \ \s+a\neq \d\\ \log (\ell/\lambda),  & \text{if} \ \s+a=\d, \end{cases}
\end{multline}
where $C>0$ depends only on $\d,\s,a$.
\end{prop}
\begin{proof}
For the sake of generality, let us  start from any function $f:\R_+\rightarrow \R_+$ such that $f/\g$ is $C^1$ and nonincreasing if $\s\neq 0$, respectively $f$ is $C^1$ and nonincreasing if $\s=0$. In the calculations presented below, $f/\g$ should be replaced by $f$ in the case $\s=0$.

Decomposing over dyadic scales $\le \ell$ and denoting 
\begin{equation}\label{eq:Kdyaddef}
K\coloneqq \left\lceil{\frac{\log (\ell /\lambda)}{\log 2}}\right\rceil,
\end{equation}
$\lceil{\cdot}\rceil$ being the usual ceiling function, we have 
\begin{align}
\sum_{\substack{i,j\in I_{\Omega}: i\ne j,\\ \lambda\le |x_i-x_j|\le \ell, \\ \dist(x_i,\p\Omega) \ge 4\ell}}  f(|x_i-x_j|) &\le \sum_{k=0}^{K-1} \sum_{\substack{i,j\in I_{\Omega}: i\ne j,\\ 2^k \lambda \le |x_i-x_j|\le 2^{k+1} \lambda ,\\  \dist(x_i, \p\Omega) \ge 4\ell}}f(|x_i-x_j|) \nn\\
&\leq \sum_{k=0}^{K-1} \sum_{\substack{i,j\in I_{\Omega}: i\ne j,\\ 2^k \lambda \le |x_i-x_j|\le 2^{k+1} \lambda \\ \dist (x_i, \p\Omega) \ge 4 \cdot 2^k \lambda  }  } \frac{f(|x_i-x_j|)}{\g(|x_i-x_j|{)}} \g(|x_i-x_j|) \nn\\
&\le \sum_{k=0}^{K-1} \frac{f(2^k\lambda)}{\g(2^k\lambda)} \sum_{\substack{i,j\in I_{\Omega}: i\ne j,\\  2^k \lambda \le |x_i-x_j|\le 2^{k+1} \lambda \\ \dist (x_i,\p\Omega) \ge 4 \cdot 2^k \lambda}} \g(2^k\ell),
\end{align}
where to obtain the second line, we use that $4\ell \geq 4\cdot 2^k\la$, for $k\le K-1$, and the third line follows from our assumption that $f/\g$ is nonincreasing and the fact that $\g(x)$ is nonincreasing (note this also implies $f$ is nonincreasing). Applying \eqref{eq:controlmic3} of \cref{lem:monoto} with $\al=2^k\la$ to the inner sum, we deduce 
\begin{equation}\label{eq:sbppre}
\frac{1}{N^2}\sum_{\substack{i,j\in I_{\Omega}: i\ne j, \lambda \le |x_i-x_j|\le \ell, \\ \dist(x_i,\p\Omega)\ge 4\ell}}   f(|x_i-x_j|) \le  C\sum_{k=0}^{K} \frac{f(2^k\lambda)}{\g(2^k\lambda)} \pa*{\Fc^{\vec{\eta}^k} - \Fc^{\vec{\eta}^{k+2}}}
\end{equation}
with for each $0\leq k\leq K$, $\vec{\eta}^k =(\eta_1^k,\ldots,\eta_N^k)$ and
\begin{equation}\label{eq:etaikdef}
\forall 1\leq i\leq N, \qquad \eta_i^k \coloneqq \begin{cases} 2^k\la, & {\dist(x_i,\p\Om)\geq 2^{k}\la} \\ {\rs_i}, & {\text{otherwise}}. \end{cases}
\end{equation}
{In view of the definition of  $\mathcal {F}^{\vec{\eta}}$, the choice \eqref{eq:etaikdef}, and the fact that $h_{N,\vec{\eta}^k}$ and $h_{N,\vec{\eta}^{k+2}}$ coincide when $|z|>2^{k+2}\lambda$, we may also replace $\Fc^{\vec{\eta}^k} - \Fc^{\vec{\eta}^{k+2}}$ by $\bar \Fc^{\vec{\eta}^k} - \bar \Fc^{\vec{\eta}^{k+2}}$ 
where
\begin{equation}
\label{eq:defbFeta}
\bar \Fc^{\vec{\eta}} \coloneqq \frac1{2\cd}\int_{\Omega\times [-\ell,\ell]^\k} \zg |\nabla h_{N,\vec{\rs}}|^2  -  \frac{1}{2N^2} \sum_{i\in I_{\Omega} : \dist (x_i, \p \Omega) \ge \lambda}  \g(\eta_i)-\frac{1}{N} \sum_{i\in I_{\Omega} : \dist(x_i, \p \Omega) \ge \lambda}\int_{\R^\d} \f_{\eta_i}(x-x_i)d\mu(x).
\end{equation}}
Using summation by parts, we  then find from \eqref{eq:sbppre} that
\begin{align}
\frac{1}{N^2}\sum_{\substack{i,j\in I_{\Omega}: i\ne j,  \lambda \le |x_i-x_j|\le \ell,\\ \dist(x_i,\p\Omega) \ge 4\ell}} f(|x_i-x_j|) &\le \sum_{k=0}^{K}\frac{f(2^k \lambda)}{\g(2^k \lambda)}   {\bar \Fc^{\vec{\eta}^k} } - \sum_{k=2}^{K+2} \frac{f(2^{k-2}\lambda)}{\g(2^{k-2}\lambda)}  {\bar \Fc^{\vec{\eta}^k}} \nn\\ 
&\le  \sum_{k=2}^{K}  \pa*{\frac{ f(2^k \lambda )}{\g(2^k \lambda ) }-  \frac{ f(2^{k-2} \lambda )}{\g(2^{k-2} \lambda) }} {\bar \Fc^{\vec{\eta}^k }} \nn\\ 
&\quad+ \frac{f(2\lambda)} {\g(2\lambda)} {\bar \Fc^{\vec{\eta}^1 }}+ \frac{f(\lambda)}{\g(\lambda )} {\bar \Fc^{\vec{\eta}^0}} -\frac{f(2\ell)}{\g(2\ell)}  {\bar \Fc^{\vec{\eta}^{K+2}}}- \frac{f(\ell)}{\g(\ell)} {\bar \Fc^{\vec{\eta}^{K+1}}}. \label{eq:msrhssub}
\end{align}
We next use that {$\bar{\Fc}^{\vec{\alpha}}$} is nonincreasing in each $\alpha_i$, as shown in \eqref{eq:premono0}, hence ${\bar \Fc^{\vec{\eta}^k}}$ is nonincreasing with respect to $k$. This monotonicity also allows us to bound from above each factor ${\bar \Fc^{\vec{\eta}^k}}$ by ${\bar \Fc^{\vec\rs}}$ and from below (by definition and by \eqref{eq:intf}) as follows: ${\bar \Fc^{\vec\eta^k}}\ge {\bar \Fc^{\vec\eta^{K+2}}}$ and
\begin{align}
{\bar \Fc^{{\vec{\eta}^k}} } &\ge -\frac{\cd}{2N^2} \sum_{i\in I_{\Omega} {: \dist(x_i, \p \Omega) \ge \lambda}} \g(\eta_i^{k}) -  \frac{C}{N} \|\mu\|_{L^\infty(\hat{\Omega}) }    \sum_{i\in I_{\Omega} { : \dist(x_i, \p \Omega) \ge \la}} ( \eta_i^{k} )^{\d-\s}. \label{eq:234}
\end{align}
{It follows that}
\begin{multline}\label{eq:msrem}
 \frac{f(2\lambda)} {\g(2\lambda)} {\bar \Fc^{\vec{\eta}^1 }}+ \frac{f(\lambda)}{\g(\lambda )} {\bar \Fc^{\vec{\eta}^0}} -\frac{f(2\ell)}{\g(2\ell)}  {\bar \Fc^{\vec{\eta}^{K+2}}}- \frac{f(\ell)}{\g(\ell)} {\bar \Fc^{\vec{\eta}^{K+1}}} \leq \pa*{\frac{f(\la)}{\g(\la)} + \frac{f(2\la)}{\g(2\la)}}{\bar \Fc^{\vec\rs}} \\
 + 2{\Big(\frac{f(\ell)}{\g(\ell)} + \frac{f(2\ell)}{\g(2\ell)}\Big)}\Big(\frac{\cd}{N^2}\sum_{i\in I_\Omega {: \dist(x_i, \p \Omega) \ge \lambda}} \g(\eta_i^{K+2}) +  \frac{C}{N} \|\mu\|_{L^\infty(\hat{\Omega}) }    \sum_{i\in I_\Omega {: \dist(x_i, \p \Omega) \ge \la}} ( \eta_i^{K+2} )^{\d-\s}\Big).
\end{multline}
By the fundamental theorem of calculus,
\begin{equation}
\frac{ f(2^k \lambda )}{\g(2^k \lambda ) }-  \frac{ f(2^{k-2} \lambda )}{\g(2^{k-2} \lambda) } = \int_{2^{k-2}\lambda}^{2^k \lambda}  \pa*{\frac{f}{\g}}'(t)\, dt.
\end{equation}
Note that $(f/\g)'\leq 0$ by our nonincreasing assumption for $(f/\g)$. Hence, 
\begin{multline}
\sum_{k=2}^{K}  \pa*{\frac{ f(2^k \lambda )}{\g(2^k \lambda ) }-  \frac{ f(2^{k-2} \lambda )}{\g(2^{k-2} \lambda) }} {\bar \Fc^{\vec{\eta}^k} } 
\le-\sum_{k=2}^K ({\bar \Fc^{\vec{\eta}^{k}}})_{-} \int_{2^{k-2}\lambda}^{2^k \lambda}  \pa*{\frac{f}{\g}}'(t)\, dt, \\
\le -\Big(\frac{\cd}{2N^2} \sum_{\substack{i\in I_{\Omega}\\ {\dist(x_i, \p \Omega) \ge \la} }} \g(\eta_i^{k}) + \frac{C}{N} \|\mu\|_{L^\infty(\hat{\Omega}) }    \sum_{\substack{i\in I_{\Omega}\\{\dist(x_i, \p \Omega) \ge \la}}} ( \eta_i^{k} )^{\d-\s}\Big)\sum_{k=2}^K \int_{2^{k-2}\lambda}^{2^k \lambda}  \pa*{\frac{f}{\g}}'(t),  \label{eq:mssumk}
\end{multline}
where $(\cdot)_-$ denotes the negative part $-\min (\cdot, 0)$ {and the final line is by \eqref{eq:234}}.  Inserting \eqref{eq:msrem} and \eqref{eq:mssumk} into the right-hand side of \eqref{eq:msrhssub}, it follows that
\begin{multline}
\frac{1}{N^2}\sum_{\substack{i,j \in I_\Omega : i\ne j,\\ \lambda \le |x_i-x_j|\le \ell ,\dist(x_i,\p\Omega ) \ge 4\ell} } f(|x_i-x_j|)  
\\
\le -\sum_{k=2}^K \Big(\frac{\cd}{2N^2} \sum_{i\in I_{\Omega}} \g(\eta_i^{k}) + \frac{C}{N} \|\mu\|_{L^\infty(\hat{\Omega}) }    \sum_{i\in I_{\Omega}} ( \eta_i^{k} )^{\d-\s}\Big)  \int_{2^{k-2}\lambda}^{2^k \lambda}  \pa*{\frac{f}{\g}}'(t)\, dt \\
+\pa*{\frac{f(\la)}{\g(\la)} + \frac{f(2\la)}{\g(2\la)}}{\bar \Fc^{\vec\rs}}+{\Big(\frac{f(\ell)}{\g(\ell)} + \frac{f(2\ell)}{\g(2\ell)}\Big)}\Big(\frac{\cd}{2N^2} \sum_{\substack{i \in I_\Omega\\ \dist(x_i, \p \Omega) \ge \la}} \g(\eta_i^{K+2}) +  \frac{C}{N} \|\mu\|_{L^\infty(\hat{\Omega}) }    \sum_{\substack{i \in I_\Omega\\ \dist(x_i, \p \Omega) \ge \la}} ( \eta_i^{K+2} )^{\d-\s}\Big).
 \end{multline}
 
From the definition \eqref{eq:etaikdef} of $\eta_i^k$ and \eqref{eq:13}, we see that
\begin{align}
&\frac{1}{N^2}\sum_{\substack{i\in I_\Omega\\ {\dist(x_i, \p \Omega) \ge \la}}} \g(\eta_i^{k}) +  \frac{1}{N} \|\mu\|_{L^\infty(\hat{\Omega}) }    \sum_{\substack{i\in I_\Omega \\{\dist(x_i, \p \Omega) \ge \la}}} ( \eta_i^{k} )^{\d-\s}  \nn\\
&\le \frac{C}{N^2}\sum_{\substack{i\in I_\Omega \\ \dist(x_i,\p\Omega)\ge 2^{k}\la} }{\g(2^k\la)} + \frac{1}{N^2}\sum_{\substack{i\in I_\Omega \\  {\lambda \le}\dist(x_i,\p\Omega) < 2^{k}\la} }\g({\rs_i}) \nn\\
&\ph +  \frac{\|\mu\|_{L^\infty(\hat{\Omega}) }}{N}   \sum_{\substack{i\in I_\Omega \\ \dist(x_i,\p\Omega)\ge 2^{k}\la} } {(2^k\la)^{\d-\s}} + \frac{C\|\mu\|_{L^\infty(\hat{\Omega}) }}{N} \sum_{\substack{i\in I_\Omega\\  {\la\le}\dist(x_i,\p\Omega)< 2^{k}\la} } ({\rs_i})^{\d-\s}  \nn\\
&\le C{\Big( \frac{1}{\cd}\int_{\Omega \times [-\la,\la]^\k}  \zg|\nab h_{N,\vec{\eta}}|^2 
+ C\frac{\#I_\Omega}{N} \|\mu\|_{L^\infty(\hat\Omega)}\la^{\d-\s}  \Big)} 
+ C\frac{\#I_\Omega}{N} \|\mu\|_{L^\infty(\hat\Omega)}{(2^k\la)^{\d-\s}} \nn\\
&\ph + C\frac{\#I_\Omega }{N^2}\|\mu\|_{L^\infty(\hat\Omega)}{\g(2^k\la)},\label{eq:234'}
\end{align}
where we have also used that ${\rs_i}\le \la$. {Evidently from the definition \eqref{eq:Kdyaddef} of $K$, the last two terms are $\le C N^{-1}\#I_\Omega\|\mu\|_{L^\infty(\hat\Omega)}\ell^{\d-\s}$ for $k=K+2$.}  If $\s\neq 0$, specializing to  $f/\g=|x|^{-a}$ with $a\ge 0$, we observe
\begin{equation}
\forall t \in [2^{k-2}\lambda,2^k\lambda], \qquad \pa*{\frac{f}{\g}}'(t) \geq -a (2^k\la)^{-a-1}.
\end{equation}
{
Further observing that
\begin{align}
\sum_{k=2}^K (2^k\la)^{\d-\s-a} \le C\la^{\d-\s-a}\begin{cases} \frac{(\ell/\la)^{\d-\s-a} - 1}{\d-\s-a}, & {\d\ne \s+a} \\ \log(\ell/\la), & {\d=\s+a} \end{cases}
\end{align}
and
\begin{align}
\sum_{k=2}^K (2^k\la)^{-\s-a} \le C\frac{(\la^{-\s-a}-\ell^{-\s-a})}{\s+a},
\end{align}
then combining with \eqref{eq:234'}, we arrive at \eqref{eq:resmultiscale} when  $\s\ne 0$.}

If $\s=0$, we replace the use of $f/\g$ by that of $f$, and obtain   the result in a similar way using this time the second case of  \eqref{eq:controlmic3}.

\end{proof}  

\section{The first-order result}\label{sec:FO}
 In this section, we consider the proof of the first-order estimates \eqref{main2} and \eqref{main1} of \cref{thm:mainunloc} and \cref{thm:FI}.

\subsection{Commutators and stress-energy tensor}\label{ssec:FOcst}

We start by describing the connection between commutators, as introduced in \cite{Rosenzweig2020spv}, and the stress-energy tensor, as introduced in \cite{Duerinckx2016,Serfaty2020}.

Given a Schwartz function $f\in\Sc(\R^{\d})$, define the Riesz potential over $\R^{\d+\k}$ by
\begin{equation}
h^f(x,z)\coloneqq \int_{\R^{\d}} \G((x,z)-y) df(y), \qquad \forall (x,z)\in\R^{\d+\k}.
\end{equation}
Given a Lipschitz vector field $v:\R^{\d}\rightarrow\R^{\d}$, let $\tv: \R^{\d+\k}\rightarrow\R^{\d+\k}$ be an extension such that $\tv(x,0) = (v(x),0)$ and $\tv^{\d+\k} = 0$ if $\k=1$. For integer $n\ge 0$, define the $n$-th order commutator over $\R^{\d+\k}$ by
\begin{equation}\label{eq:kandef}
\ka^{(n),f}(x,z) \coloneqq \int_{\R^{\d}}\nabla^{\otimes n}\G((x,z)-y) : (\tv(x)-\tv(y))^{\otimes n} df(y), \qquad \forall (x,z)\in\R^{\d+\k}.
\end{equation}
Such an extension $\tv$ always exists, since we have the trivial extension $\tv(x,z) \coloneqq (v(x),0)$, which has the same Lipschitz seminorm as $v$.

When assuming the trivial extension, we will abuse notation and drop the $\tilde{\cdot}$ superscript. {Another extension, more appropriate in the local case, is as follows.}

\begin{remark}\label{rem:vextell}
We may extend a vector field $v:\R^\d\rightarrow\R^\d$ to $\R^{\d+\k}$ by
\begin{align}\label{deftv}
\tl{v}(x,z) \coloneqq \chi(\frac{z}{\ell}) (v(x),0), \qquad \forall (x,z)\in\R^{\d+\k},
\end{align}
where $\chi $ is a smooth bump function satisfying 
\begin{align}
\chi(z) = 1 \quad \text{if} \ |z|\le \frac12, \qquad  \chi(z)=0 \quad \text{if} \ |z|\ge 1.
\end{align}
In contrast to the trivial extension seen above, this choice of $\tv$ is also localized in the $z$ coordinate. Moreover, if $\supp v$ is contained in an open ball $B$ of radius $\ell$ in $\R^\d$, then $\supp\tv$ is contained in an open ball $\tl B$ of radius $C\ell$ in $\R^{\d+\k}$.

Since $v$ vanishes at some point $x_0$ if $\ell<\infty$, the mean-value theorem implies that
\begin{align}
\|v\|_{L^\infty}\le C\ell \|\nab v\|_{L^\infty}.
\end{align}
More generally, for any $m\ge 0$,
\begin{align}\label{eq:nabmm+1v}
\|\nab^{\otimes m} v\|_{L^\infty} \le C \ell \|\nab^{\otimes m+1}v\|_{L^\infty}.
\end{align}
The same reasoning applies to any extension $\tv$ supported in $\tl{B}$, hence
\begin{equation}\label{bornetv}
\|\nab^{\otimes m}\tv\|_{L^\infty} \le C\ell \|\nab^{\otimes m+1}\nab\tv\|_{L^\infty}.
\end{equation}

By the product rule,  the extension \eqref{deftv} satisfies
\begin{align}\label{bornenabtv}
\|\nab^{\otimes m}\tv\|_{L^\infty} &\le C_m\|\nab^{\otimes m}v\|_{L^\infty}.
\end{align}
and
\begin{align}
\||z|^{-\ga}L\tv\|_{L^\infty}  &= \left\| \chi(\cdot/\ell) \p_{i}^2 v  + \ell^{-2}\chi''(\cdot/\ell)v + \ell \ga z|z|^{-2}\chi'(\cdot/\ell)v  \right\|_{L^\infty} \le C\|\nab^{\otimes 2} v\|_{L^\infty},
\end{align}
where we have used \eqref{eq:nabmm+1v} and that $\chi'(z/\ell)$ is supported in $1\le |z/\ell|\le 2$.

\end{remark}

For later use in justifying the integration by parts, we note that since $|\nab^{\otimes m}\g(x,z)| \lesssim |(x,z)|^{-\s-m}$, {it follows that if $\tv$ is a $C^m$ extension,
\begin{align}\label{eq:nabmka1dcay}
|\nab^{\otimes m}\ka^{(n),f}(x,z)| \lesssim |(x,z)|^{-\s-n},
\end{align}
and if the derivatives of $\tv$ up to order $m$ are rapidly decreasing, then
\begin{align}
|\nab^{\otimes m}\ka^{(n),f}(x,z)| \lesssim |(x,z)|^{-\s-m-n}.
\end{align}
Indeed, for any multi-index $\vec\al \in \N_0^{\d+\k}$ of order $|\vec\al|=m$, the Leibniz rule yields
\begin{align}
\p_{\vec\al}\ka^{(n),f} &= \sum_{\vec\be \le \vec\al} {\vec\al\choose \vec\be}\int_{\R^{\d}} \p_{\vec\be}\nab^{\otimes n}\g(\cdot-y) : \p_{\vec{\al}-\vec\be}(\tv(\cdot)-v(y))^{\otimes n} df(y) \nn\\
&=\sum_{\vec\be \le \vec\al} \sum_{\vec{\ga}_1+\cdots+\vec{\ga}_n = \vec{\al}-\vec{\be}} {\vec\al\choose \vec\be} {\vec{\be}\choose \vec{\ga}_1,\ldots,\vec{\ga}_n}\int_{\R^{\d}} \p_{\vec\be}\nab^{\otimes n}\g(\cdot-y) : \bigotimes_{i=1}^n \p_{\vec{\ga}_i}(\tv(\cdot)-v(y)) df(y).
\end{align}
If $|(x,z)-y| > \frac12|(x,z)|$, then the singularity of $\p_{\vec\be}\nab^{\otimes n}\g(\cdot-y)$ is irrelevant, and we may crudely bound
\begin{align}
&\int_{|(x,z)-y| > \frac12|(x,z)|} |\p_{\vec\be}\nab^{\otimes n}\g((x,z)-y) : \bigotimes_{i=1}^n \p_{\vec{\ga}_i}(\tv(x,z)-v(y))| d|f|(y) \nn\\
&\le C\prod_{i: \ga_i>0} \|\p_{\vec\ga_i}\tv\|_{L^\infty}\int_{|(x,z)-y| > \frac12|(x,z)|} |(x,z)-y|^{-\s-n-|\vec\be|}  \prod_{i : \ga_i =0} |\tv(x,z)-v(y)| d|f|(y) \nn\\
&\le C\|f\|_{L^1}|(x,z)|^{-\s-n-|\vec\be|}\prod_{i: \ga_i>0} \|\p_{\vec\ga_i}\tv\|_{L^\infty} \prod_{i: \ga_i=0} \|\tv\|_{L^\infty},
\end{align}
where the final line follows from $f\in L^1$. If $|\vec\be| <m$, then there exists an $i$ such that $|\vec\gamma_i|>0$. In which case, if the derivatives of $\tv$ are rapidly decreasing up to order $m$, then $|\p_{\vec\ga_i}\tv(x,z)| \lesssim |(x,z)|^{-(m-|\vec\be|)}$. Hence, the preceding right-hand side may be replaced by $C_{\tv}\|f\|_{L^1} |(x,z)|^{-\s-n-m}$. If $|(x,z)-y| \le \frac12|(x,z)|$, then $|y|\ge \frac12|(x,z)|$, and using the rapid decay of $f$, we see that
\begin{align}
&\int_{|(x,z)-y| \le \frac12|(x,z)|} |\p_{\vec\be}\nab^{\otimes n}\g(\cdot-y) : \bigotimes_{i=1}^n \p_{\vec{\ga}_i}(\tv(\cdot)-v(y))| d|f|(y) \nn\\
&\le C_f\prod_{i: \ga_i>0} \|\p_{\vec\ga_i}\tv\|_{L^\infty}\int_{|(x,z)-y| \le \frac12|(x,z)|}|(x,z)-y|^{-\s-n-|\vec\be|}\prod_{i : \ga_i =0 } \min\Big(\|\tv\|_{L^\infty}, \|\nab\tv\|_{L^\infty}|(x,z)-y|\Big) \jp{y}^{-r}dy  \nn\\
&\le C_{f,\tv}|(x,z)|^{-\s-n-m},
\end{align}
where we have used that $\#\{i : \ga_i=0\} \ge n-(m-|\vec\be|)$ and $r>0$ may be taken arbitrarily large.

We also note that by taking pointwise limits, we may reduce to the case where $\tv$ is smooth and compactly supported. Indeed, let $\rho_\epsilon \coloneqq \epsilon^{-\d}\rho(\cdot/\epsilon)$ be an approximation to the identity. If $\tv$ is Lipschitz, then $\|\rho_\ep\ast\tv -\tv\|_{L^\infty} \le C\epsilon \|\nab \tv\|_{L^\infty}$. Letting $\ka_\epsilon^{(n),f}$ denote the commutator with $\tv$ replaced by $\tv_\epsilon\coloneqq\rho_\epsilon\ast\tv$, we may split the domain of integration to estimate
\begin{align}
|\ka^{(n),f} - \ka_\epsilon^{(n),f}| &\le C\|\nab \tv\|_{L^\infty}^n\int_{|\cdot-y|\le r} |\cdot-y|^{-\s} d|f|(y)  \nn\\
&\ph+ \int_{|\cdot-y|>r} \Big|\nab^{\otimes n}\g(\cdot-y) : \Big((\tv(\cdot)-\tv(y))^{\otimes n} - (\tv_\epsilon(\cdot)-\tv_\epsilon(y))^{\otimes n}\Big)\Big| d|f|(y) \nn\\
&\le C\|\nab\tv\|_{L^\infty}^n \|f\|_{L^\infty} r^{\d-\s} + C\|\nab\tv\|_{L^\infty}\|\tv\|^{n-1}\|f\|_{L^1}r^{-\s-n} \epsilon
\end{align}
Choosing $r^{\d+n} = \epsilon$, the preceding right-hand side tends to zero as $\epsilon\rightarrow 0^+$. On the other hand, if $\chi$ is a smooth bump function in $\R^{\d+\k}$, then letting $\ka_{R}^{(n),f}$ denote the commutator with $\tv$ replaced by $\tv \chi(\cdot/R)$, the same decomposition yields (assuming $R$ is large enough)
\begin{align}
|\ka^{(n),f} - \ka_R^{(n),f}| &\le  C\|\nab\tv\|_{L^\infty}^n \|f\|_{L^\infty} r^{\d-\s} + C\|\tv\|_{L^\infty}^n\int_{\substack{|\cdot-y|>r \\ |y|\ge R}} |\cdot-y|^{-\s-n} d|f|(y) \nn\\
&\le C\|\nab\tv\|_{L^\infty}^n \|f\|_{L^\infty} r^{\d-\s} + C_f \|\tv\|_{L^\infty}^n r^{-\s-n} R^{-1},
\end{align}
where the final line follows from the rapid decay of $f$. Choosing $r^{\d+n} = R^{-1}$, we see that the preceding right-hand tends to zero as $R\rightarrow\infty$.}

For the following proposition, we recall the notation $L_\ga^2$ from \eqref{eq:Lga2}.

\begin{prop}\label{prop:comm}
Given a vector field $v:\R^{\d}\rightarrow\R^{\d}$, let $\tv$ be an extension to $\R^{\d+\k}$ as above. Assume that $\tv$ is smooth and compactly supported. There is a constant $C>0$ depending only on $\d,\s$, such that for any $f\in \Sc_0(\R^\d)$, the space of Schwartz functions with zero mean, it holds that
\begin{equation}\label{eq:comm}
\|\nabla \ka^{(1),f} \|_{L_\ga^2(\R^{\d+\k})} \leq C\|\nab\tv\|_{L^\infty} \|\nab h^f\|_{L_\ga^2(\supp\nab\tv)}.
\end{equation}
Consequently, for any $f,w\in\Sc_0(\R^{\d})$,
\begin{align}\label{eq:comm'}
\Big|\int_{(\R^{\d})^2}(v(x)-v(y))\cdot\nabla\G(x-y)df(x)dw(y)\Big| \le C\|\nab\tv\|_{L^\infty} \|\nab h^f\|_{L_\ga^2(\supp \nab\tv)}\|\nab h^w\|_{L_\ga^2(\supp \nab\tv)},
\end{align}
and restricting $\ka^{(1),f}$ to $\R^\d\times\{0\}^\k \simeq \R^\d$, we have
\begin{align}\label{eq:comm''}
\|\ka^{(1),f}{|_{z=0}}\|_{\dot{H}^{\frac{\d-\s}{2}}(\R^\d)} \le C\|\nab\tv\|_{L^\infty} \|\nab h^f\|_{L_\ga^2(\supp\nab\tv)}.
\end{align}
\end{prop}

\begin{remark}\label{rem:denska1}
The compact support assumption is purely qualitative and may be reduced to assuming only that the extension $\tv$ is Lipschitz. Indeed, one may mollify and cut off and then use Fatou's lemma, appealing to the pointwise convergence shown above.

By density, the map $f\mapsto \ka^{(1),f}$ has an extension from the subspace $\Sc_0(\R^{\d})$ to the space of distributions $f\in \dot{H}^{\frac{\s-\d}{2}}(\R^{\d})$. To see this, note
\begin{equation}
\int_{(\R^{\d})^2}\g(x-y)dw(x)df(y) = \frac{1}{\cd}\int_{\R^{\d+\k}}\zg\nabla h^w\cdot \nabla h^f,
\end{equation}
where the equality is a consequence of the identity \eqref{eq:Gfs}. In particular, if $w=f$, then by Plancherel's theorem,
\begin{equation}\label{eq:GSob}
\|f\|_{\dot{H}^{\frac{\s-\d}{2}}}^2 = \cd'\int_{\R^{\d+\k}}\zg|\nabla h^f|^2.
\end{equation}
Given $f\in\dot{H}^{\frac{\s-\d}{2}}(\R^\d)$, one may take a sequence $f_m\in\Sc_0(\R^\d)$ converging to $f$ in $\dot{H}^{\frac{\s-\d}{2}}$. Applying \eqref{eq:comm} with $f$ replaced by $f_{m'}-f_m$, it follows that $\nab\ka^{(1),f_m}$ is Cauchy in $L_\ga^2$, from which the claim follows.

 Going forward, we will always assume that the test functions $f,w\in\Sc(\R^{\d})$ have Fourier transforms supported away from the origin. This ensures that $h^{f},h^w$ are again Schwartz functions with Fourier transform supported away from the origin. This reduction is justified by the density of such Schwartz functions in $\dot{H}^{\frac{\s-\d}{2}}$.
\end{remark}

\cref{prop:comm} provides a simple proof of an $L^2$-based commutator estimate similar to some found in the literature with more complicated proofs. See, for instance, \cite{CJ1987} or \cite[Appendix]{Rosenzweig2020spv} for the Coulomb case corresponding to the Calder\'{o}n $\d$-commutator.  This simplicity is made possible by a formula expressing $L\ka^{(1),f}$ as the divergence of a vector field consisting of combinations of products of the components of $\nab v,\nab h^f$. This subsumes the stress-energy tensor structure previously used in \cite{Duerinckx2016,Serfaty2017,Serfaty2020}.

We state the crucial identity for $L\ka^{(1),f}$ and its relationship to the stress-energy tensor
\begin{align}\label{eq:stdef}
\comm{\nab h_1}{\nab h_2}_{ij} \coloneqq {\zg\Big(\p_i h_1\p_j h_2 + \p_j h_1\p_i h_2 - \nab h_1\cdot\nab h_2 \delta_{ij}\Big)}, \qquad i,j \in [\d+\k].
\end{align}
\emph{Here and throughout this paper, we follow the convention that repetition of index indicates summation over that index.}

\begin{lemma}\label{lem:Lka1}
Let $v:\R^{\d}\rightarrow\R^{\d}$ be a Lipschitz vector field and $\tv$ be an extension as above. For any test function $f\in\Sc_0(\R^{\d})$, it holds that
\begin{align}\label{eq:Lka1}
L\ka^{(1),f} = {-  \p_i(\zg\p_i \tv^j \p_j h^f)-\p_j(\zg \p_i \tv^j \p_i h^f) + \p_j(\zg \p_i \tv^i \p_j h^f)}.
\end{align}
Moreover, for a test function $\phi$ on $\R^{\d+\k}$,
\begin{align}\label{eq:Lka1hw}
\phi \, L\ka^{(1),f} = \nab \tv : \comm{\nab h^f}{\nab\phi}
-  \p_i(\zg\p_i \tv^j \p_j h^f \phi)-\p_j(\zg \p_i \tv^j \p_i h^f \phi) + \p_j(\zg \p_i \tv^i \p_j h^f \phi).
\end{align}
Consequently,
\begin{align}\label{eq:commst}
\frac{1}{\cd}\int_{\R^{\d+\k}}\nab\ka^{(1),f}\cdot \nab h^w &=  \int_{(\R^{\d})^2}(v(x)-v(y))\cdot\nabla\G(x-y)df(x)dw(y) \nn\\
&= \frac{1}{\cd}\int_{\R^{\d+\k}}\nabla \tv: \comm{\nab h^f}{\nab h^w}.
\end{align}
\end{lemma}
\begin{proof}
By direct computation (consider \eqref{compL} below with $n=1$, and see the proof of \cref{lem:Lkapnf} for the general case of $L\ka^{(n),f}$), one has
\begin{align}\label{eq:Lka1pre}
L\ka^{(1),f} = -\cd \p_i \tv^i \tl{f} -  \p_i(\zg\p_i \tv^j \p_j h^f) - \zg \p_i \tv^j\p_i\p_j h^f.
\end{align}

Applying the product rule to $\p_j$, we find
\begin{align}
- \zg \p_i\tv^j\p_i\p_j h^f = -\p_j(\zg \p_i\tv^j \p_i h^f) + \zg \p_i\p_j\tv^j\p_ih^f.
\end{align}
Writing $\cd\tl{f} = -\p_j(\zg \p_j h^f)$ and using the  product rule, we find
\begin{align}
-\cd \p_i\tv^i\tl{f} = \p_j(\zg \p_i\tv^i \p_j h^f) - \zg \p_j\p_i\tv^i\p_j h^f.
\end{align}
By symmetry with respect to swapping $i\leftrightarrow j$, it follows that
\begin{align}
-\zg\p_i\tv^j \p_i \p_j h^f -\cd \p_i\tv^i\tl{f} &= -\p_j(\zg \p_i\tv^j \p_i h^f) + \p_j(\zg \p_i \tv^i \p_j h^f).
\end{align}
Inserting this identity into the right-hand side of \eqref{eq:Lka1pre} yields the desired \eqref{eq:Lka1}. The identity \eqref{eq:Lka1hw} follows now from the product rule, and \eqref{eq:commst} follows from \eqref{eq:Lka1hw} and integration by parts,  which is justified by the decay \eqref{eq:nabmka1dcay}, using that $Lh^w = \cd \tl{w}$.
\end{proof}

\begin{proof}[Proof of \cref{prop:comm}]
To show \eqref{eq:comm}, we use the identity \eqref{eq:Lka1hw} with $\phi = \ka^{(1),f}$ and apply Cauchy-Schwarz to the right-hand side. To deduce \eqref{eq:comm'}, we use \eqref{eq:commst} and Cauchy-Schwarz. For  \eqref{eq:comm''}, we apply the bound \eqref{eq:comm'} with $w= (-\Delta)^{\d-\s}\ka^{(1),f}$, in which case the left-hand side becomes $\|\ka^{(1),f}\|_{\dot{H}^{\frac{\d-\s}{2}}(\R^\d)}^2$.  Observing that
\begin{align}
\Big(\int_{\supp\nabla\tv} \zg|\nabla h^w|^2\Big)^{1/2} \leq C \|h^w\|_{\dot{H}^{\frac{\d-\s}{2}}(\R^\d)} = C\|(-\Delta)^{{\d-\s}}h^{\ka^{(1),f}}\|_{\dot{H}^{\frac{\d-\s}{2}}(\R^\d)} = C\cd \|\ka^{(1),f}\|_{\dot{H}^{\frac{\d-\s}{2}}(\R^\d)}
\end{align}
yields the desired conclusion.

\end{proof}

We will see in Appendix \ref{appA} how to use iterated stress-energy tensor estimates to deduce higher-order versions of this result in the Coulomb case (or Riesz case $\s=\d-1$ for which $\gamma=0$).

\subsection{Renormalization of the first order commutator estimate}
\label{ssec:FOren}
In this subsection, we give a self-contained proof of the first-order (i.e.~ $n=1$) special case of \cref{thm:mainunloc} and \cref{thm:FI}, which generalizes the Coulomb-specific results \cite[Proposition 4.2]{Serfaty2023}, \cite[Proposition 3.9]{Rosenzweig2021ne} to the super-Coulomb case. Compared to those prior works, the proof is here simpler, avoiding the use of elliptic regularity estimates through an elementary averaging argument. Here and throughout this section, we use the same notation as in \cref{sec:ME}.



We will need the following lemma,  which shows that $\g\ast\delta_y^{(\eta)} = \g_\eta(\cdot-y)$ on $\R^\d\times\{0\}^\k$ when $y\in\R^\d$.
\begin{lemma}\label{lem:lemme}
Let $\eta\ge 0$. If $y \in \R^\d$, then
\begin{equation}\label{eq:lemme}
\int_{\R^{\d+\k}} \G((x,z)-w) d\delta_y^{(\eta)} (x,z) = \G_\eta( w-y), \qquad \forall w \in \R^\d.
\end{equation}
\end{lemma}
\begin{proof}
By translation invariance, it suffices to prove the identity for $y=0$. By definition \eqref{defdeta} of $\delta_0^{(\eta)}$ and by \eqref{eq:Gfs}, we have
\begin{align}
\G*\delta_0^{(\eta)}(w)= \int_{\R^{\d+\k}} \G((x,z)-w) d\delta_0^{(\eta)} (x,z) &= \frac{1}{\cd}\int \G((x,z)-w) L\g_\eta(x,z) \nn\\
& =  \frac{1}{\cd} \int_{\R^{\d+\k}} \zg \nabla \G((x,z)-w) \cdot \nabla \G_\eta(x,z) \nn\\
&= - \frac{1}{\cd} \int_{\R^{\d+\k}}\Big(\div (|z|^\gamma \nabla \G(\cdot-w) )\Big)(x,z) \G_\eta (x,z) \nn\\
& =  \int_{\R^{\d+\k}}  \G_\eta(x,z)d\delta_w (x,z) = \G_\eta(w).
\end{align}
\end{proof}

We now turn to the proof of the estimates \eqref{main2} and \eqref{main1}. Suppose that the vector field extension $\tv$ has the property that $x\notin \supp \nab v$ implies that $(x,z)\notin \supp\nab\tv$. Remark that this property holds for the trivial extension, as well as for the localized extension given in Remark \ref{rem:vextell}.  Let us consider a parameter vector $\vec{\eta}=(\eta_1,\ldots,\eta_N)\in (\R_+)^N$ such that for every $i$, $\eta_i \le \rs_i$. We then desymmetrize
\begin{align}
I&\coloneqq  \int_{(\R^\d)^2\setminus\triangle} ( v(x)-v(y))\cdot \nabla\g(x-y)d\Big(\frac{1}{N} \delta_{x_i} - \mu\Big)^{\otimes 2} (x,y) \nn\\
& = \sum_{i=1}^N \frac{2}{N} \int_{\R^{\d+\k}}  \tv(x_i) \cdot \nabla \G(x_i-\cdot)d\Big( \frac{1}{N} \sum_{j : j\neq i} \delta_{x_j} - \tl\mu\Big)  \nn\\
&\ph - 2 \int_{(\R^{\d+\k})^2} \tv(x) \cdot \nabla \G((x,z)-\cdot) d\tl\mu(x,z)d\Big(\frac{1}{N} \delta_{x_i} - \tl\mu\Big) \nn\\
& = \sum_{i=1}^N \frac2N \int_{\R^{\d+\k}} \tv \cdot \nabla h_N^i d\delta_{x_i} - 2\int_{\R^{\d+\k}} \tv\cdot \nabla h_N d\tl\mu. \label{eq:FIo1pre}
\end{align}
Using that
\begin{align}
h_N^i = h_{N, \vec{\eta} } - \frac1N\G_{\eta_i} (\cdot-x_i) \quad \text{in } \ B(x_i, \eta_i), \\
h_N = h_{N, \vec{\eta}}  +\frac1N \sum_{i=1}^N(\G- \G_{\eta_i}) (\cdot-x_i),
\end{align}
we decompose \eqref{eq:FIo1pre} as $\Te_1 + \Te_2 +\Te_3$, where 
\begin{align}\label{eq:FIo1T1}
 \Te_1 \coloneqq 2 \int_{\R^{\d+\k}} \tv\cdot \nabla h_{N, \vec{\eta}} \, d \Big(\frac 1N \sum_{i=1}^N \delta_{x_i}^{(\eta_i)} - \tl\mu\Big),
\end{align}
\begin{multline}\label{eq:FIo1T2}
\Te_2 \coloneqq  \frac2N \sum_{i=1}^N \int_{\R^{\d+\k}} (\tv(x_i) - \tv) \cdot \nabla h_N^i d\delta_{x_i}^{(\eta_i)} \\
- \frac2{N^2} \sum_{i=1}^N \int_{\R^{\d+\k}} (\tv-\tv(x_i)) \cdot \nabla \G_{\eta_i} (\cdot-x_i) d\delta_{x_i}^{(\eta_i)}  + \frac2{N}\sum_{i=1}^N \int_{\R^{\d+\k}} (\tv-\tv(x_i)) \cdot \nabla (\G_{\eta_i}- \G) (\cdot-x_i) d\tl\mu,
\end{multline}
and
\begin{multline}\label{eq:FIo1T3}
\Te_3 \coloneqq \frac2N \sum_{i=1}^N \int_{\R^{\d+\k}}\tv(x_i) \cdot \nabla h_N^i d\Big(\delta_{x_i} - \delta_{x_i}^{(\eta_i)}\Big) 
 - \frac2{N^2} \sum_{i=1}^N\int_{\R^{\d+\k}} \tv(x_i) \cdot \nabla \G_{\eta_i} (\cdot -x_i) d \delta_{x_i}^{(\eta_i)} \\
+ \frac2N  \sum_{i=1}^N\int_{\R^{\d+\k}} \tv(x_i) \cdot \nabla (\G_{\eta_i}-\G) (\cdot-x_i) d\tl\mu.
\end{multline}
We dispense with $\Te_3$ by showing that it vanishes.

\bigskip

\noindent\textbullet $\Te_3$: Unpacking the definition \eqref{defhni} of $h_N^i$, we write
\begin{multline}\label{eq:FIo1T3pre}
\Te_3= \frac{2}{N^2} \sum_{i=1}^N \sum_{j: j\neq i} \int_{\R^{\d+\k}}\tv(x_i) \cdot \nabla \G(\cdot-x_j) d\Big(\delta_{x_i} - \delta_{x_i}^{(\eta_i)}\Big)
\\
- \frac{2}{N}  \sum_{i=1}^N \int_{\R^{\d+\k}} \tv(x_i) \cdot \nabla \G(\cdot-y) d\tl\mu(y) d\Big(\delta_{x_i} - \delta_{x_i}^{(\eta_i)}\Big) \\
- \frac2{N^2} \sum_{i=1}^N\int_{\R^{\d+\k}} \tv(x_i) \cdot \nabla \G_{\eta_i} (\cdot -x_i) d \delta_{x_i}^{(\eta_i)} 
+ \frac2N  \sum_{i=1}^N\int_{\R^{\d+\k}} \tv(x_i) \cdot \nabla (\G_{\eta_i}-\G) (\cdot-x_i) d\tl\mu.
\end{multline}
Thanks to \cref{lem:lemme}, we have for $i\ne j$,
\begin{align}
\int_{\R^{\d+\k}} \nabla \G(\cdot-x_j) d\Big( \delta_{x_i} - \delta_{x_i}^{(\eta_i)}\Big) =  \nabla \G(x_i-x_j)-  \nabla \G_{\eta_i} (x_i-x_j).
\end{align}
The right-hand side vanishes because $|x_i-x_j |>\eta_i$ for $j \neq i$ (by assumption that $\eta_i \le \rs_i$) and $\G_{\eta_i}$ coincides with $\G$ outside of $B(0, \eta_i)$.  Thus, the first line of \eqref{eq:FIo1T3pre} vanishes. By the same reasoning, the second line of  \eqref{eq:FIo1T3pre} equals
\begin{align}
-\frac2N \sum_{i=1}^N \int_{\R^{\d+\k}} \tv(x_i) \cdot \nabla( \G- \G_{\eta_i}) (x_i-y)  d\tl\mu(y).
\end{align}
Thus, the second line cancels with the last term on the third line of  \eqref{eq:FIo1T3pre}. It remains to show that 
\begin{align}
\int_{\R^{\d+\k}} \tv(x_i) \cdot \nabla \G_{\eta_i} (\cdot-x_i) d \delta_{x_i}^{(\eta_i)} =0.
\end{align}
This is true by the divergence theorem because 
\begin{align}
\nabla \G_{\eta_i} (\cdot-x_i) \delta_{x_i}^{(\eta_i)} &=-\frac{1}{\cd}  \nabla \G_{\eta_i} (\cdot-x_i) \div \left(|z|^\gamma \nabla \G_{\eta_i} (\cdot-x_i) \right)\nn\\
&= -\frac{1}{ 2\cd} \div\comm{\nab \G_{\eta_i} (\cdot-x_i)}{\nab  \G_{\eta_i}(\cdot-x_i)},
\end{align}
where $\comm{\cdot}{\cdot}$ is the stress-energy tensor \eqref{eq:stdef}.


This resolves $\Te_3$, leaving us with the task of estimating $\Te_1,\Te_2$.

\medskip

\noindent\textbullet \ $\Te_1$: Similar to the identity \eqref{eq:commst}, we write $\nab h_{N,\vec\eta}\Big(\frac1N\sum_{i=1}^N \delta_{x_i}^{(\eta_i)}-\tl\mu\Big) = -\frac{1}{2\cd}\div\comm{\nab h_{N,\vec\eta}}{\nab h_{N,\vec\eta}}$ and integrate by parts to obtain
\begin{equation}
\Te_1 = \int_{\R^{\d+\k}} \nabla\tv : \comm{\nab h_{N,\vec{\eta}}}{\nab h_{N,\vec{\eta}}}.
\end{equation}
It follows now from Cauchy-Schwarz that
\begin{equation}\label{eq:T1fin}
|\Te_1| \leq  C\|\nabla\tv\|_{L^\infty} \int_{\supp\nab\tv} \zg |\nabla h_{N,\vec{\eta}}|^2.
\end{equation}

\medskip
\noindent\textbullet \ $\Te_2$: We claim that $\tv(x_i)-\tv = 0$ on the support of $\delta_{x_i}^{(\eta_i)}$ if $\dist(x_i,\supp \nab v)>\la/4$. Indeed, this condition on $x_i$ and the fact that $\eta_i\le \rs_i\le \la/4$ imply that if $(x,z)\in\supp\delta_{x_i}^{(\eta_i)}$, then $\dist(t x_i + (1-t)x,\supp\nab v) >0$ for any $t\in [0,1]$. By the support assumption on $\tv$ and the mean-value theorem, the claim follows. For the remaining $i$, using the mean-value theorem on $\tv - \tv(x_i)$ and the explicit form of the probability measure $\delta_{x_i}^{(\eta_i)}$, we see that\footnote{Note that due to an editing error, there is a discrepancy between \eqref{eq:T21}  below and ensuing computations with the published journal version. There is no change to the final conclusion, but the version presented here is the correct one.}
\begin{multline}\label{eq:T21}
\frac2N \sum_{i=1}^N \Big|\int_{\R^{\d+\k}} (\tv(x_i) - \tv) \cdot \nabla h_N^i d\delta_{x_i}^{(\eta_i)} \Big| \\
\le \frac{C}{N}\|\nab \tv\|_{L^\infty}\sum_{i : \dist(x_i,\supp \nab v) \le {\frac14} \la}\eta_i^{-\s} {\int_{\p B(x_i,\eta_i)}|z|^{\gamma} |\nab h_N^i| d\mathcal{H}^{\d+\k-1}},
\end{multline}
where $\p B(x_i,\eta_i)$ is the sphere in $\R^{\d+\k}$ and $\mathcal{H}^{\d+\k-1}$ is the $(\d+\k-1)$-Hausdorff measure. Similarly,
\begin{align}\label{eq:T22}
\frac2{N^2} \sum_{i=1}^N \Big|\int_{\R^{\d+\k}} (\tv-\tv(x_i)) \cdot \nabla \G_{\eta_i} (\cdot-x_i) d\delta_{x_i}^{(\eta_i)}\Big| \le \frac{C}{N^2}\|\nab \tv\|_{L^\infty}\sum_{i : \dist(x_i,\supp \nab v) \le {\frac14}\la}\eta_i^{-\s}.
\end{align}
Recalling  \eqref{eq:defGeta} and that $\G-\G_\eta = \f_\eta$ is supported in $B(0, \eta)$,
\begin{align}\label{eq:T23}
\frac2N  \sum_{i=1}^N \Big|\int_{\R^{\d+\k}}(\tv- \tv(x_i)) \cdot \nabla (\G_{\eta_i}-\G) (\cdot-x_i) d\mu\Big| \le \frac{C}{N} \|\nab \tv\|_{L^\infty}\|\mu\|_{L^\infty(\hat\Omega)}\sum_{i : \dist(x_i,\supp \nab v) \le {\frac14}\la} \eta_i^{\d-\s}.
\end{align}
Combining \eqref{eq:T21}, \eqref{eq:T22}, \eqref{eq:T23} yields
\begin{multline}\label{eq:T2fin}
|\Te_2| \le \frac{C}{N^2}\|\nab\tv\|_{L^\infty}\sum_{i : \dist(x_i,\supp \nab v) \le {\frac14} \la}\eta_i^{-\s} + \frac{C}{N} \|\nab\tv\|_{L^\infty}\|\mu\|_{L^\infty(\hat\Omega)}\sum_{i : \dist(x_i,\supp \nab v) \le  {\frac14} \la} \eta_i^{\d-\s}\\
+\frac{C}{N}\|\nab\tv\|_{L^\infty}\sum_{i : \dist(x_i,\supp \nab v) \le {\frac14} \la}\eta_i^{-\s} {\int_{\p B(x_i,\eta_i)}|z|^\gamma  |\nab h_N^i| d\mathcal{H}^{\d+\k-1}}.
\end{multline}

\bigskip
Combining the estimates \eqref{eq:T1fin} and  \eqref{eq:T2fin}, we have found that there exists a constant $C>0$ depending only on $\d,\s$, such that for every choice $\vec{\eta}=(\eta_1,\ldots,\eta_N)$ satisfying $\eta_i\leq\rs_i$, we have 
\begin{multline}\label{eq:prefin}
|I|\le C   \|\nabla\tv\|_{L^\infty} \int_{{\supp \nab \tl v}} |z|^{\ga} |\nabla h_{N,\vec{\eta}}|^2 
+\frac{C}{N^2} \|\nabla \tv\|_{L^\infty}\sum_{{ i: \dist(x_i,\supp \nab v)\leq {\frac14}\la }}  \eta_i^{-\s} \\
+ \frac{C}{N} \|\mu\|_{L^\infty(\hat\Om)} \|\nabla \tv\|_{L^\infty} \sum_{{ i: \dist(x_i,\supp \nab v)\leq {\frac14}\la }} \eta_i^{\d-\s} \\
+ \frac{C}{N} \|\nabla\tv\|_{L^\infty}  \sum_{{ i: \dist(x_i,\supp \nab v)\leq {\frac14}\la }} \eta_i^{-\s} {\int_{\partial B(x_i, \eta_i)} |z|^\gamma |\nabla h_N^i| d\mathcal{H}^{\d+\k-1}}.
\end{multline}
For each $t\in[\frac{1}{2},1]$, we apply this relation with $\eta_i = t\rs_i$ and then average both sides of the resulting inequality over $t\in [\frac12,1]$. Using spherical coordinates and a change of variable $t\rs_i \mapsto t$,
\begin{align}
&\frac{C}{N} \|\nabla \tv\|_{L^\infty}  \sum_{{ i: \dist(x_i,\supp \nab v)\leq {\frac14}\la }}  {\int_{\frac12}^1 (t\rs_i)^{-\s} \int_{\partial B(x_i, t \rs_i)} \zg |\nabla h_N^i| d\mathcal{H}^{\d+\k-1} dt} \nn\\
&\leq \frac{C}{N} \|\nabla \tv\|_{L^\infty}\sum_{{i:\dist(x_i,\supp \nab v)\leq {\frac14}\la }} { \rs_i^{-\s-1} \int_{B(x_i, \rs_i)} |z|^\gamma|\nabla h_N^i|} \nn\\
&\le \frac{C}{N} \|\nabla\tv\|_{L^\infty}\sum_{{i:\dist(x_i,\supp \nab v)\leq{\frac14} \la }} {\rs_i^{-\s-1}} \Big(\int_{B(x_i, \rs_i)} |z|^\gamma|\nabla h_N^i|^2\Big)^{\frac12}\Big(\int_{B(0, \rs_i)} {|z|^{\gamma}} \Big)^{\frac12}, \label{eq:finapp'}
\end{align}
where the final line is by Cauchy-Schwarz. Since $\gamma {\in (-1,1)}$, the last integral is convergent and bounded by  $C\rs_i^{\d+\k+\gamma} = C\rs_i^{\s+2}$, by \eqref{eq:defgamma}. Using \eqref{eq:energballs}, Cauchy-Schwarz, and $ab \le \frac12(a^2+b^2)$, we find that
\begin{align}
\eqref{eq:finapp'} &\le \frac{C}{N}  \|\nabla \tv\|_{L^\infty}  \sum_{{i: \dist(x_i,\supp \nab v)\leq{\frac14} \la}}  \rs_i^{-\frac{\s}{2}}\Big(\int_{ B(x_i, \rs_i)}  |z|^\gamma|\nabla h_N^i|^2\Big)^{\frac12} \nn\\
&\leq \frac{C}{N}\|\nabla \tv\|_{L^\infty}\Bigg(\sum_{{ i: \dist(x_i,\supp \nab v)\leq {\frac14}\la}} \rs_i^{-\s}\Bigg)^{\frac12} \Bigg(\int_{ {\{ \dist(x, \supp \nab v)\le \frac12 \la\}  \times [-\la,\la]^\k} }  |z|^\gamma|\nabla h_{N, \vec\rs} |^2\Bigg)^{\frac12}\nn\\
&\leq C\|\nabla \tv\|_{L^\infty} \Bigg(\frac{1}{N^2} \sum_{{ i: \dist(x_i,\supp \nab v)\leq{\frac14} \la}}  \rs_i^{-\s} +\int_{{ \{ \dist(x, \supp \nab v)\le \frac12 \la\} \times [-\la,\la]^{\k}}}  |z|^\gamma|\nabla h_{N, \vec\rs} |^2\Bigg).
\end{align}
Inserting this estimate into \eqref{eq:prefin}, we obtain
\begin{multline}\label{eq:finapp}
|I|\le C   \|\nabla \tv\|_{L^\infty} \Bigr( {\int_{\supp \nab \tl v } \zg
|\nab h_{N,\vec{\eta}}|^2+ 
\int_{\{ \dist(x, \supp \nab v)\le \frac12 \la\}  \times [-\la,\la]^{\k}} |z|^{\ga} |\nabla h_{N,\vec\rs}|^2 } \Bigr)\\
+\frac{C}{N^2} \|\nabla \tv\|_{L^\infty}\sum_{{ i: \dist(x_i,\supp \nab v)\leq {\frac14} \la }}  \rs_i^{-\s}
+ \frac{C}{N} \|\mu\|_{L^\infty(\hat\Om)} \|\nabla \tv\|_{L^\infty} \sum_{{ i: \dist(x_i,\supp \nab v)\leq{\frac14} \la }} \rs_i^{\d-\s}. 
\end{multline}

 Recall from the statement of \cref{thm:FI} that $\Om$ contains a closed $2\la$-neighborhood of $\supp \nab v$, so that the condition $\dist(x_i,\supp \nab v)\leq  \la$ implies $x_i\in\Omega$ and $\dist(x_i,\p\Om)>\la$. 
 Moreover, 
 by definition \eqref{eq:defHNeta}, we have 
\begin{equation}
h_{N,\vec{\eta}}= h_{N,\vec\rs}+ \frac{1}{N}\sum_{i=1}^N (\g_{ \eta_i}- \g_{\rs_i})(x-x_i).
\end{equation}
Thus, using the triangle inequality and the fact that the supports of the $(\g_{ \eta_i}- \g_{\rs_i})(x-x_i)$ are disjoint,  \eqref{eq:intf} and $ \eta_i \in [\frac12\rs_i, \rs_i]$, we have
\begin{align}\label{htrht}\int_{ \supp \nab \tilde v}\zg |\nabla h_{N,\eta}|^2 &\le 2 \int_{\supp \nab\tilde v }\zg |\nabla h_{N,\vec\rs}|^2 
+ \frac{2}{N^2} \sum_{i: \dist(x_i, \supp \nab v)\le \frac14\lambda}\int_{\R^{\d+\k}} \zg |\nab (\g_{ \eta_i}-\g_{\rs_i})(x-x_i)|^2 \nn\\ 
&\le 2 \int_{\supp \nab \tilde v }\zg |\nabla h_{N,\vec\rs}|^2 
+ \frac{C}{N^2} \sum_{i : \dist(x_i, \supp \nab v )\le \frac14\lambda}\rs_i^{-\s}.
\end{align}

Inserting into \eqref{eq:finapp} and using the estimate \eqref{eq:13}  from \cref{prop:MElb}, and recalling $\rs_i\leq \la$, we conclude that 
\begin{equation}\label{eq:Ifin}
|I|\leq C  \|\nabla\tv\|_{L^\infty} \Bigg( { \int_{(\supp \nab \tl v)\cup (\Omega \times [-\la,\la]^{\k})}\zg |\nabla h_{N,\vec\rs}|^2}
  + C \frac{\# I_{\Omega}\|\mu\|_{L^\infty(\hat \Omega) } \la^{\d-\s}}{N}\Bigg).
\end{equation}

In the global case of \cref{thm:mainunloc}, we take $\tv$ to be the trivial extension, and then 
 $\|\nab\tv\|_{L^\infty}=\|\nab v\|_{L^\infty}$. The proof of \eqref{main2} is then complete.
 In the local case of \cref{thm:FI}, we take $\tv$ to be the extension given by Remark \ref{rem:vextell}. The $n=1$ case of \cref{thm:FI} then follows in view of $\lambda\le \ell$, \eqref{bornenabtv} and the fact that $\supp \nab \tv \subset \supp \nab v\times [-\ell, \ell]^\k$. 

\section{Regularity theory and commutator estimates}\label{sec:L2reg}
\label{sec:comm}
In this section, we present our new commutator estimates for the Riesz potential, which arise by considering the variation along linear transport of the Riesz energy of a sufficiently regular distribution. We shall then use these commutator estimates in combination with a renormalization procedure in \cref{sec:FI} in order to estimate such variations of the modulated energy $\Fr_N(\ux_N,\mu)$, in particular proving  \cref{thm:mainunloc} and \cref{thm:FI}. The main result of this section is the following theorem, which may be of independent interest. It corresponds to the functional inequality of \cref{thm:mainunloc} and \cref{thm:FI}, but {\it without the renormalization} needed to include  singular Diracs.

\begin{thm}\label{thm:comm2}
Let $v:\R^\d\rightarrow\R^\d$ be a vector field and $\tv$ be an extension as above. Let $f\in\Sc_0(\R^{\d})$, $h^f=\g*f$, and $\kappa^{(n),f}$ be as in \eqref{eq:kandef}.

If $\supp\nab\tv$ is contained in a ball of radius $\ell$ in $\R^{\d+\k}$ and $\||z|^{-\ga} L\tv\|_{L^\infty} \le C\|\nab^{\otimes 2}\tv\|_{L^\infty}$, then 
\begin{equation}\label{eq:comm2}
\|\nab\ka^{(n),f}\|_{L_\ga^2(\R^{\d+\k})} \le C (\ell \|\nab^{\otimes 2}\tv\|_{L^\infty})^{n}\|\nab h^f\|_{L_\ga^2(\supp\nab\tv)},
\end{equation}
Consequently, for any $f,w\in\Sc_0(\R^{\d})$, it holds that
\begin{multline}\label{eq:comm2dual}
\left|\int_{(\R^{\d})^2}\nabla^{\otimes n}\G(x-y)\cdot(v(x)-v(y))^{\otimes n} df(y)dw(x)\right|\\
\leq C (\ell \|\nab^{\otimes2}\tv\|_{L^\infty})^{n}
\|\nab h^f\|_{L^2_\ga(\supp\nab\tv)} \|\nab h^w\|_{L^2_\ga(\supp\nab\tv)} .
\end{multline}

If $n\le 2$, then we have the stronger estimates
\begin{align}\label{eq:comm2'}
\|\nab\ka^{(n),f}\|_{L_\ga^2(\R^{\d+\k})} \le C \|\nab\tv\|_{L^\infty}^n \|\nab h^f\|_{L_\ga^2(\supp\nab\tv)}
\end{align}
and
\begin{align}\label{eq:comm2dual'}
\Big|\int_{(\R^{\d})^2}\nabla^{\otimes n}\G(x-y)\cdot(v(x)-v(y))^{\otimes n} df(y)dw(x)\Big| \le C \|\nab\tv\|_{L^\infty}^n \|\nab h^f\|_{L_\ga^2(\supp\nab\tv)} \|\nab h^w\|_{L_\ga^2(\supp\nab\tv)}.
\end{align}

Finally, for any $n\ge 1$, 
\begin{align}\label{eq:comm2''}
\|\ka^{(n),f}\|_{\dot{H}^{\frac{\d-\s}{2}}(\R^\d)} \le C\|\nab v\|_{L^\infty}^n \|f\|_{\dot{H}^{\frac{\s-\d}{2}}(\R^\d)}.
\end{align}
and
\begin{align}\label{eq:comm2dual''}
\Big|\int_{(\R^{\d})^2}\nabla^{\otimes n}\G(x-y)\cdot(v(x)-v(y))^{\otimes n} df(y)dw(x)\Big| \le C\|\nab v\|_{L^\infty}^n \|f\|_{\dot{H}^{\frac{\s-\d}{2}}(\R^\d)}\|w\|_{\dot{H}^{\frac{\s-\d}{2}}(\R^\d)}.
\end{align}

In all cases, $C>0$ depends only on $n,\d,\s$. 
\end{thm}


\begin{remark}\label{rem:denskan}
By the same argument as in \cref{rem:denska1}, the estimate \eqref{eq:comm2} implies that $\ka^{(n),f}$ has a unique extension from $\Sc_0(\R^\d)$ to $\dot{H}^{\frac{\s-\d}{2}}(\R^\d)$. 
\end{remark}

{
Let us comment in more detail about the strategy of proof for \cref{thm:comm2}.

We already stated in \eqref{eq:introLkanf} the PDE obeyed by $\kappa^{(n),f}$. What we believe is true is that the right-hand side of \eqref{eq:introLkanf} can in addition be written in divergence form involving terms built from products of the components of $\nab v$ and $\nu^{(m)}$ for $m\le n-1$. More precisely, we conjecture
\begin{align}\label{eq:introLkanfdiv}
L\ka^{(n),f} = - n\div(\zg\nab v^{i}\nu_{i}^{(n-1),f}) +  \sum_{k=1}^{n} C_{k,n}\sum_{\sigma \in \Ss_{k+1}} (-1)^{|\sigma|} \p_{i_{\sigma(k+1)}}\Big(\zg \p_{i_{\sigma(1)}}v^{i_1}\cdots \p_{i_{\sigma(k)}}v^{i_k} \nu_{i_{k+1}}^{(n-k),f}\Big),
\end{align}
where the $C_{k,n}$ are certain combinatorial coefficients and $|\sigma|$ denotes the signature of the permutation. The reader will recall from the introduction that $\nu^{(m),f}$ is a vector field in $\R^{\d+\k}$ that morally is like $\nab \ka^{(m),f}$. The precise definition is given in \eqref{eq:nudef} below. Let us note that by testing \eqref{eq:introLkanfdiv} against $h^w$, using that
\begin{align}
\int_{(\R^\d)^2}(v(x)-v(y))\cdot\nab\g(x-y)df(x)dw(y) = \frac{1}{\cd}\int_{\R^{\d+\k}}L\ka^{(n),f}h^w,
\end{align}
and reversing the product rule, one arrives at a stress-energy tensor structure to the higher-order commutators. If such an identity \eqref{eq:introLkanfdiv} holds and one has a bound
\begin{align}\label{eq:introL2nu}
\int_{\R^{\d+\k}}\zg |\nu^{(m),f}|^2 \leq C_{m}(\|\nab v\|_{L^\infty})^{2m} \int_{\supp\nab v\times\R^\k}\zg|\nab h^f|^2, \qquad m\le n-1,
\end{align}
then it follows immediately from integration by parts and Cauchy-Schwarz that
\begin{align}\label{eq:introL2nabka}
\int_{\R^{\d+\k}}\zg |\nab \ka^{(n),f}|^2  \le C_{n}(\|\nab v\|_{L^\infty})^{2n}\int_{\supp\nab v\times\R^\k}\zg|\nab h^f|^2.
\end{align}
As the $\nu$'s obey a recursion in terms of the $\nab \ka$'s (see \eqref{recnu}), the estimate \eqref{eq:introL2nabka} implies that \eqref{eq:introL2nu} holds for $m=n$ and by induction, the estimates \eqref{eq:introL2nabka}, \eqref{eq:introL2nu} hold for any integer $n\ge 1$. 

Unfortunately, we are only able to prove the identity \eqref{eq:introLkanfdiv} for $n\le 2$ (see \Cref{lem:Lka1,lem:Lka2}), and the computation is already quite involved at second order. If we stay with the non-divergence form of the right-hand side of \eqref{eq:introLkanf}, then we encounter terms that have a factor of $\ka^{(n),f}$. For instance, 
\begin{align}
\int_{\R^{\d+\k}}\ka^{(n),f}\zg\p_i v \cdot \p_i\nu^{(n-1),f} = -\int_{\R^{\d+\k}}\zg\p_i\ka^{(n),f} \p_{i}v\cdot\nu^{(n-1),f}  -\int_{\R^{\d+\k}}\zg\ka^{(n),f}\p_i^2 v\cdot \nu^{(n-1),f}. 
\end{align}
Note there is no contribution from the weight, as $\p_{\d+\k}v = 0$. One cannot simply use Cauchy-Schwarz on the second term on the right-hand side, as in general, there is no way to control $\int_{\supp\nab v\times\R^\k}\zg|\ka^{(n),f}|^2$ by $\int_{\supp\nab v\times\R^\k}\zg|\nab\ka^{(n),f}|^2$. However, recalling that our primary interest is in localized estimates, there is a way out of this issue by exploiting the localization from the start, defining $\ka^{(n),f}$ in terms of a vector field extension $\tv$ which is 
localized in 
 a ball $B$ of radius $C\ell$ in $\R^{\d+\k}$ as in \cref{rem:vextell}.
 Since $L\ka^{(n),f}$ has zero average, we may equivalently test the equation \eqref{eq:introLkanf} against $\ka^{(n),f} -\bar{\ka}^{(n),f}$, where $\bar{\ka}^{(n),f}$ denotes the average in the ball $B$. The strategy is to integrate by parts and use Cauchy-Schwarz as before, setting up an induction argument for estimates satisfied by $\ka^{(m),f}, \nu^{(m),f}, \mu^{(m),f}$ for $m\le 2$ (see \cref{lem:inducknumu}). Crucially, the measure $\zg dxdz$ satisfies a Poincar\'{e} inequality in balls, which allows to control $\int_{B}\zg |\ka^{(n),f}-\bar{\ka}^{(n),f}|^2$ by $\int_{B}\zg |\nab\ka^{(n),f}|^2$. 
 Although this approach requires a bound for $\|\nab^{\otimes 2}v\|_{L^\infty}$, which is a stronger demand than the Lipschitz requirement that would follow if \eqref{eq:introLkanfdiv} holds, this is not problematic for applications, such as to CLTs for the fluctuations of Coulomb/Riesz gases.

}
 

\subsection{Finding the PDE solved by $\kappa^{(n),f}$}\label{ssec:L2regeqn}
{We pay our debt to the reader by proving the identity \eqref{eq:introLkanf} for $L\ka^{(n),f}$.}


First, we define the vector field $\nu^{(n),f}: \R^{\d+\k}\rightarrow \R^{\d+\k}$ specified by the components 
\begin{equation}\label{eq:nudef}
\nu_i^{(n),f}(\cdot) \coloneqq  \int_{\R^{\d}}  \partial_i\nab^{\otimes n}\G(\cdot-y) : (\tv(\cdot)-\tv(y))^{\otimes n} f(y)dy, \qquad i\in [\d+\k].
\end{equation}
The $\nu^{(n),f}$  are seen by integration by parts to satisfy the recursion
\begin{align}\label{recnu}
\nu^{(n),f}_i &=  \partial_i \kappa^{(n),f}-\int_{\R^\d}\Big( \partial_{i_1 }\dots \partial_{i_n} \G(\cdot-y) \prod_{k \in [n]\backslash \{l\} } ( \tv-\tv(y))^{i_k}  \Big)\partial_i \tv^{i_l} f(y) dy \nn\\
&=  \partial_i \kappa^{(n),f}- n\partial_i \tv\cdot \nu^{(n-1), f}.
\end{align}

Second, we define the matrix field $\mu:\R^{\d+\k} \rightarrow (\R^{\d+\k})^{\otimes 2}$ specified by the components 
\begin{equation}\label{eq:defmu}
\mu_{ij}^{(n), f}(\cdot) \coloneqq  \int_{\R^{\d}}  \partial_i \p_j \nab^{\otimes n} \G(\cdot-y) :  (\tv(\cdot) -\tv(y))^{\otimes n} f(y)dy, \qquad i,j\in [\d+\k].
\end{equation}
Evidently, $\mu$ is symmetric. The $\mu^{(n), f}$ also satisfy a recursion
 \begin{align}\label{recmu} 
 \mu_{ij}^{(n), f} &= \partial_i \nu_j^{(n), f}- n\int_{\R^{d}}  \p_j \p_k\nab^{\otimes n-1} \G(\cdot-y) : \p_i \tv^k \otimes ( \tv-\tv(y))^{\otimes (n-1)} f(y)dy \nn\\
&= \partial_i \nu_j^{(n), f}- n \p_i  \tv^k \mu_{jk}^{(n-1), f}.
 \end{align}
Remark that $\nu^{(0),f} = \nab h^f$ and $\mu^{(0),f} = \nab^{\otimes 2} h^f$.

\begin{lemma}\label{lem:Lkapnf}
For $f\in \mathcal{S}_0(\R^\d)$, we have
\begin{multline}\label{eqLk}
L\kappa^{(n), f} =\cd(-1)^n\sum_{\sigma \in \Ss_n} \p_{i_{\sigma_1}}\tl{v}^{i_1}\cdots\p_{i_{\sigma_n}}\tl{v}^{i_n}\tl{f}   -n\p_i(\zg\p_i \tv \cdot\nu^{(n-1),f}) - n\zg\p_i\tv \cdot \p_i\nu^{(n-1),f} \\
+ n(n-1)\zg(\p_i \tv)^{\otimes 2}:\mu^{(n-2),f}.
\end{multline}
\end{lemma}
\begin{proof}
By definition of $\kappa^{(n), f}$, differentiating inside the integral, plus using the product rule, we compute
\begin{align}\label{compL}
L\kappa^{(n),f}  &=\int_{\R^{\d}}  L\Big(\nab^{\otimes n} \G(\cdot-y)\Big) :(\tl{v}-\tl{v}(y))^{\otimes n} f(y)dy \nn\\
&\ph+ n\int_{\R^{\d}} \nab^{\otimes n} \G(\cdot-y)  :L\tv \otimes (\tv -\tv(y))^{\otimes (n-1)} f(y)dy \nn\\
&\ph - n(n-1)\zg\int_{\R^{\d}} \nab^{\otimes n} \G(\cdot-y) : \p_i \tv\otimes \p_i \tv  \otimes (\tv -\tv(y))^{\otimes (n-2)} f(y) dy \nn\\
&\ph -2n\zg \int_{\R^{\d}}  \partial_i \nab^{\otimes n} \G(\cdot-y) : \p_i \tv  \otimes (\tv -\tv(y))^{\otimes (n-1)} f(y) dy.
\end{align}

In view of  \eqref{eq:Gfs}, we have for any $X=(x^1,\ldots,x^{\d+\k}),Y=(y^1,\ldots,y^{\d+\k})$ that
\begin{align}
-\div (  |x^{\d+\k}-y^{\d+\k}|^\ga \nab \G(X-Y))= \cd \delta_{X=Y}.
\end{align}
Specializing this identity to $y^{\d+\k}=0$, it follows that
\begin{align}
-\div (  |x^{\d+\k}|^\ga   \nab \G(X-y))=\cd \delta_{X=y},
\end{align}
i.e.~ $L\Big(\g(\cdot-y)\Big)(X)= \cd \delta_{X=y}$ if $y\in\R^\d$. Since $\tv^{\d+\k} = 0$ if $\k=1$ and $\p_1,\ldots,\p_\d$ commute with $L$, it follows that
\begin{align}
\int_{\R^{\d}}  L\Big(\nab^{\otimes n} \G(\cdot-y)\Big) :(\tl{v}-\tl{v}(y))^{\otimes n} f(y)dy  &= \int_{\R^{\d}}\nab^{\otimes n} (L\G)(\cdot-y):(\tl{v}-\tl{v}(y))^{\otimes n} f(y)dy \nn\\
&=\cd\int_{\R^{\d}} \nabla^{\otimes n}\delta_{y} : (\tl{v}-\tl{v}(y))^{\otimes n} f(y)dy \nn\\
&=\cd(-1)^n\sum_{\sigma \in \Ss_n} \p_{i_{\sigma_1}}\tv^{i_1}\cdots\p_{i_{\sigma_n}}\tv^{i_n}\tl{f},
\end{align}
where the final equality follows from the definition of the distributional derivative, the Leibniz rule, and the fact that $\tv-\tv(y)$ is zero on the support of $\delta_{y}$. 
 
Recalling the definition \eqref{eq:nudef} of $\nu^{(n),f}$, the second term on the right-hand side of \eqref{compL} is 
\begin{align}
n L\tv^i  \cdot \int_{\R^{\d}}  \p_i \nab^{\otimes n-1} \G(\cdot-y) :  (\tv -\tv(y))^{\otimes (n-1)} f(y)dy
 = nL\tv \cdot \nu^{(n-1),f}.
\end{align}

Recalling the definition \eqref{eq:defmu} of $\mu^{(n),f}$, fourth term on the right-hand side of \eqref{compL} is 
\begin{align}
& -2n\zg    \int_{\R^{\d}}  \partial_i \p_j \nab^{\otimes (n-1)} \G(\cdot-y) : \p_i \tv^j   (\tv -\tv(y))^{\otimes (n-1)} f(y)dy \nn\\
& = -2n\zg    \p_i \tv^j \int_{\R^{\d}}  \partial_i \p_j \nab^{\otimes (n-1)} \G(\cdot-y) :  (\tv -\tv(y))^{\otimes (n-1)} f(y)dy \nn\\
& = -2n\zg  \nab\tv : \mu^{(n-1),f}.
\end{align}
 
Similarly, the third term on the right-hand side of \eqref{compL} is
\begin{multline}
- n(n-1)\zg  \p_i \tv^j \p_i \tv^k \int_{\R^{\d}} \p_{j}\p_{k}\nab^{\otimes (n-2)} \G(\cdot-y) : (\tv -\tv(y))^{\otimes (n-2)} f(y)dy \\
 =-n(n-1)\zg  (\p_i\tv)^{\otimes 2} : \mu^{(n-2),f}.
\end{multline}

Assembling the prior relations, we arrive at
\begin{multline}\label{eq:Lkapfin'}
L\kappa^{(n),f}  = \cd(-1)^n\sum_{\sigma \in \Ss_n} \p_{i_{\sigma_1}}\tv^{i_1}\cdots\p_{i_{\sigma_n}}\tv^{i_n}\tl{f} + nL\tv \cdot \nu^{(n-1),f} -2n\zg  \nab\tv : \mu^{(n-1),f}\\
-n(n-1)\zg  (\p_i\tv)^{\otimes 2} : \mu^{(n-2),f}.
\end{multline}
Reversing the product rule,
\begin{align}\label{eq:Lkap'}
 nL\tv \cdot \nu^{(n-1),f} &= -n\p_i(\zg\p_i\tv \cdot \nu^{(n-1),f}) + n\zg\p_i\tv \cdot \p_i\nu^{(n-1),f},
 \end{align}
and using the recursion \eqref{recmu} for $\mu^{(n-1),f}$,
\begin{align}\label{eq:Lkap''}
 -2n\zg  \nab\tv : \mu^{(n-1),f} = -2n\zg\p_i\tv^j \Big(\p_i\nu_j^{(n-1),f} - (n-1)\p_i\tv^k\mu_{jk}^{(n-2),f}\Big).
\end{align}
Inserting \eqref{eq:Lkap'}, \eqref{eq:Lkap''} into \eqref{eq:Lkapfin'} and simplifying, we arrive at \eqref{eqLk}.

\end{proof}

\subsection{(Localized) $L^2$ commutator estimates}\label{ssec:L2regloc}
In this subsection, we prove the estimates \eqref{eq:comm2}, \eqref{eq:comm2dual} of \cref{thm:comm2}.

We first prove the estimate \eqref{eq:comm2}  by the standard method for the $L^2$ regularity of elliptic equations, i.e.~testing the  PDE \eqref{eqLk}  against its solution.

\begin{lemma}\label{lem:inducknumu}
Under the same assumptions as in \cref{thm:comm2}, for any $n\ge 1$,\footnote{ The claim is also true for $n=0$ provided that one replaces $\R^{\d+\k}$ in the left-hand side integrals by $\supp\nab\tv$.}
\begin{align}\label{induck}
\|\nab\ka^{(n),f}\|_{L_\ga^2(\R^{\d+\k})} &\le  C_n(\ell \|\nab^{\otimes 2}\tv\|_{L^\infty})^{n}\|\nab h^f\|_{L_\ga^2(\supp\nab\tv)},\\ \label{inducnu}
\|\nu^{(n),f}\|_{L_\ga^2(\R^{\d+\k})} &\le C_n(\ell \|\nab^{\otimes 2}\tv\|_{L^\infty})^{n}\|\nab h^f\|_{L_\ga^2(\supp\nab\tv)},
\end{align}
and for any matrix field $X = (X_{ij})_{i,j=1}^{\d+\k}$ such that $X_{(\d+\k)j}, X_{i(\d+\k)}$ are supported in $\supp \nab v \times \{\ell\le |z|\le 2\ell\}^\k $,
\begin{align}\label{inducmu}
\int_{\R^{\d+\k}}\zg \mu^{(n),f} : X \le  C_n(\ell \|\nab^{\otimes 2}\tv\|_{L^\infty})^{n}\|\nab h^f\|_{L_\ga^2(\supp\nab\tv)}  \Big(\|\nab  X\|_{L^2_{\gamma}(\R^{\d+\k})}   +\ell^{-1}  \|X\|_{L^2_{\gamma}(\R^{\d+\k})}  \Big).
\end{align}
Consequently, for any test function $\phi$ on $\R^{\d+\k}$,
\begin{align}\label{inducphi}
\int_{\R^{\d+\k}}\zg\nab\phi\cdot\nab\ka^{(n),f} \le C_n(\ell \|\nab^{\otimes 2}\tv\|_{L^\infty})^{n}\|\nab h^f\|_{L_\ga^2(\supp\nab\tv)} \|\nab\phi\|_{L_\ga^2(\supp\nab\tv)}.
\end{align}
\end{lemma}

\begin{proof}
We prove the lemma by induction on $n\ge 1$.

Let us start with the base case $n={1}$.  The inequality \eqref{induck} is implied by \eqref{eq:comm} of \cref{prop:comm}. Applying the recursion \eqref{recnu} for $\nu$ and recalling $\nu^{(0),f} = \nab h^f$, we have by triangle inequality that
\begin{align}
\|\nu^{(1),f}\|_{L_\ga^2} \le  \|\nab\kappa^{(1),f}\|_{L_\ga^2} + \|\nab \tv\cdot \nab h^f\|_{L_\ga^2} &\le C\|\nab\tv\|_{L^\infty} \|\nab h^f\|_{L_\ga^2(\supp\nab\tv)}.\label{eq:nu1L2}
\end{align}
This implies \eqref{inducnu} for $n=1$. Finally, using the recursion \eqref{recmu} for $\mu$, the triangle inequality, and recalling that $\mu^{(0),f} = \nab^{\otimes 2}h^f$, 
\begin{align}
\int_{\R^{\d+\k}}\zg \mu^{(1),f} : X = \int_{\R^{\d+\k}}\zg\partial_i \nu_j^{(1), f} X_{ij} - \int_{\R^{\d+\k}}\zg \p_i  \tv^k \p_{j}\p_{k}h^{f} X_{ij}.
\end{align}
Integrating by parts,
\begin{align}
\int_{\R^{\d+\k}}\zg\partial_i \nu_j^{(1), f} X_{ij} = -\int_{\R^{\d+\k}}\zg\nu_j^{(1),f}\p_i X_{ij} - \ga\delta_{i(\d+\k)}\int_{\R^{\d+\k}}z|z|^{\ga-2}\nu_j^{(1),f}X_{ij}.
\end{align}
By Cauchy-Schwarz and \eqref{eq:nu1L2},
\begin{align}
\int_{\R^{\d+\k}}\zg|\nu_j^{(1),f}\p_i X_{ij}| \le \|\nu_j^{(1),f}\|_{L_\ga^2} \|\p_i X_{ij}\|_{L_\ga^2} \le C\|\nab\tv\|_{L^\infty}\|\nab h^f\|_{L_\ga^2(\supp\nab\tv)}\|\nab X\|_{L_\ga^2}.
\end{align}
Similarly, also using the support hypothesis for $X$,
\begin{align}
\delta_{i(\d+\k)}\int_{\R^{\d+\k}}|z|z|^{\ga-2}\nu_j^{(1),f}X_{ij}| &\le C\ell^{-1} \|\nu_j^{(1),f}\|_{L_\ga^2} \|X_{(\d+\k)j}\|_{L_\ga^2} \nn\\
&\le C\ell^{-1}\|\nab\tv\|_{L^\infty}\|\nab h^f\|_{L_\ga^2(\supp\nab\tv)}\|X\|_{L_\ga^2}.
\end{align}
Integrating by parts $\p_k$ and using that $\tv^{\d+\k} = 0$ if $\k=1$,
\begin{align}
 - \int_{\R^{\d+\k}}\zg \p_i  \tv^k \p_{j}\p_{k}h^{f} X_{ij} = \int_{\R^{\d+\k}}\zg \p_i\p_k\tv^k \p_j h^{f} X_{ij}  +  \int_{\R^{\d+\k}}\zg \p_i\tv^k \p_j h^f \p_k X_{ij}. 
\end{align}
By Cauchy-Schwarz,
\begin{align}
\int_{\R^{\d+\k}}\zg |\p_i\p_k\tv^k \p_j h^{f} X_{ij}| \le C\|\nab^{\otimes 2}\tv\|_{L^\infty} \|\nab h^f\|_{L_\ga^2(\supp\nab\tv)} \|X\|_{L_\ga^2(\supp\nab\tv)}, \\
 \int_{\R^{\d+\k}}\zg |\p_i\tv^k \p_j h^f \p_k X_{ij}| \le C \|\nab \tv\|_{L^\infty} \|\nab h^f\|_{L_\ga^2(\supp\nab\tv)} \|\nab X\|_{L_\ga^2(\supp\nab\tv)}.
\end{align}
Combining the preceding relations, we arrive at
\begin{multline}
\int_{\R^{\d+\k}}\zg \mu^{(1),f} : X  \le C\|\nab h^f\|_{L_\ga^2(\supp\nab\tv)}\Big(\|\nab\tv\|_{L^\infty}\|\nab X\|_{L_\ga^2} +\ell^{-1}\|\nab\tv\|_{L^\infty}\|X\|_{L_\ga^2}\\
+ \|\nab^{\otimes 2}\tv\|_{L^\infty}  \|X\|_{L_\ga^2(\supp\nab\tv)}  +   \|\nab\tv\|_{L^\infty}  \|\nab X\|_{L_\ga^2(\supp\nab\tv)}\Big),
\end{multline}
which establishes \eqref{inducmu} for $n=1$. 

\medskip
Let us now assume that \eqref{induck}, \eqref{inducnu}, \eqref{inducmu} hold up to order $n-1$. Let $\phi$ be a test function on $\R^{\d+\k}$, and let $\bar\phi$ denote the average of $\phi$ over $\tl{B}$, which we remind the reader is a ball of radius $C\ell$ in $\R^{\d+\k}$ containing $\supp\nab\tv$. Integrating \eqref{eqLk} against $\phi-\bar\phi$ and replacing $\tl f$ by $\frac{1}{\cd}Lh^f$, we obtain 
\begin{multline}\label{petb}
\int_{\R^{\d+\k} } (\phi-\bar\phi)  L\kappa^{(n), f} =(-1)^n\sum_{\sigma \in \Ss_n} \int_{\R^{\d+\k}} \p_{i_{\sigma_1}}\tv^{i_1}\cdots\p_{i_{\sigma_n}}\tv^{i_n}  (\phi-\bar\phi)  Lh^f\\ 
 - n\int_{\R^{\d+\k}}\p_i(\zg\p_i\tv \cdot\nu^{(n-1),f})(\phi -\bar\phi)  +n(n-1)\int_{\R^{\d+\k}}\zg (\p_i \tv)^{\otimes 2} : \mu^{(n-2),f} (\phi-\bar\phi) \\
-n\int_{\R^{\d+\k}} \zg\p_i\tv \cdot \p_i\nu^{(n-1),f} (\phi-\bar\phi).
\end{multline}

Integrating by parts and using that $\nab\ka^{(n),f}$ decays like $|(x,z)|^{-\s-n-1}$ at infinity, the left-hand side of \eqref{petb} equals $\int_{\R^{\d+\k}}\zg\nab\phi\cdot\nab\ka^{(n),f}$. Further integrating by parts in the right-hand side of \eqref{petb} and using the triangle inequality, we obtain
\begin{align}
\int_{\R^{\d+\k} }\zg \nab\phi\cdot\nab\ka^{(n),f} & \le C_n\int_{\supp \nab \tv} \zg |\nab \tv|^n  |\nab\phi| | \nab h^f| \nn\\  
 &\ph + C_n\int_{\supp \nab \tv} \zg |\nab^{\otimes2}\tv| |\nab\tv|^{n-1}  |\phi-\bar\phi | | \nab h^f| \nn\\  
&\ph  +C_n\int_{\R^{\d+\k}}   \zg |\nab \tv| |\nu^{(n-1),f}| |\nab\phi| \nn\\
&\ph + C_n\int_{\R^{\d+\k}} |L\tv| |\nu^{(n-1),f}|  |\phi-\bar\phi| \nn\\
&\ph + C_n \Big|\int_{\R^{\d+\k}}\zg (\p_i \tv)^{\otimes 2} : \mu^{(n-2),f} (\phi-\bar\phi)\Big|. \label{eq:nabknpre}
\end{align}
Suppose that the estimates \eqref{induck}, \eqref{inducnu}, \eqref{inducmu} hold up to $n-1$.

By Cauchy-Schwarz and the induction hypothesis \eqref{inducnu} for $\nu^{(n-1),f}$,
\begin{align}
\int_{\supp \nab \tv} \zg |\nab \tv|^n  |\nab\phi|| \nab h^f| \le C_n \|\nab \tv\|_{L^\infty}^n \|\nab\phi\|_{L_\ga^2} \|\nab h^f\|_{L_\ga^2(\supp\nab\tv)}, \label{eq:nabknpre1}
 \end{align}
 and
 \begin{align}
\int_{\R^{\d+\k}}   \zg |\nab \tv| |\nu^{(n-1),f}| |\nab \phi| &\le C_n\|\nab\tv\|_{L^\infty} \|\nu^{(n-1),f}\|_{L_\ga^2} \|\nab\phi\|_{L_\ga^2} \nn\\
 &\le C_n(\ell \|\nab^2 \tv\|_{L^\infty})^{n} \|\nab h^f\|_{L_\ga^2(\supp\nab\tv)}\|\nab\phi\|_{L_\ga^2}. \label{eq:nabknpre2}
\end{align}

For the remaining terms in \eqref{eq:nabknpre},
we recall the Poincar\'e inequality in $L_\ga^2$ \cite[Theorem 1.5]{FKS1982}
\begin{equation}\label{poincare}
\Big(\int_{B} \zg |\phi -\bar\phi|^2\Big)^{1/2} \le C \ell\Big(\int_{B} \zg |\nab\phi|^2\Big)^{1/2}.
\end{equation}
Note this inequality is valid because our $\ga\in (-1,1)$ (recall the definition \eqref{eq:defgamma}) and therefore $|z|^{\ga}$ is an $A_2$ weight.

Together with the Cauchy-Schwarz inequality and the induction hypothesis \eqref{inducnu} for $\nu^{(n-1),f}$, \eqref{poincare} implies that
\begin{align}
\int_{\R^{\d+\k}} |L\tv| |\nu^{(n-1),f}|  | \phi-\bar\phi| &\le C_n \||z|^{-\ga} L\tv\|_{L^\infty} \|\nu^{(n-1),f}\|_{L_\ga^2}\Big(\int_{B} \zg |\phi -\bar\phi|^2\Big)^{1/2} \nn\\
 &\le C_n\|\nab^{\otimes 2}\tv\|_{L^\infty} (\ell \|\nab^{\otimes 2}\tv\|_{L^\infty})^{n-1} \|\nab h^f\|_{L_\ga^2(\supp\nab\tv)} \ell\|\nab\phi\|_{L_\ga^2(B)} \nn\\
 &= C_n(\ell \|\nab^{\otimes 2}\tv\|_{L^\infty})^{n}\|\nab h^f\|_{L_\ga^2(\supp\nab\tv)}  \|\nab\phi\|_{L_\ga^2(B)}, \label{eq:nabknpre3}
\end{align}
where we have implicitly used the hypothesis that $\||z|^{-\ga}L\tv\|_{L^\infty} \le C\|\nab^{\otimes 2}\tv\|_{L^\infty}$. Finally, we note that $\mu^{(n-2), f}$ is tested against the tensor
\begin{align}
X_{jk} \coloneqq \p_i \tv^j\p_i\tv^k(\phi-\bar\phi), 
\end{align}
whose $(\d+\k,k)$ or $(j,\d+\k)$ components are zero if $\k=1$, and which satisfies the bounds (by consequence of \eqref{poincare})
\begin{align}
\|X\|_{L_\ga^2} \le  \ell\|\nab \tv\|_{L^\infty}^2 \|\nab\phi\|_{L_\ga^2},\label{eq:XLga2}\\
\|\nab X\|_{L_\ga^2} \le \|\nab\tv\|_{L^\infty}(\ell\|\nab^{\otimes 2}v\|_{L^\infty}) \|\nab\phi\|_{L_\ga^2}. \label{eq:nabXLga2}
\end{align}
Thus, we may use the induction hypothesis \eqref{inducmu} for $\mu^{(n-2),f}$ together with Cauchy-Schwarz and \eqref{poincare} to obtain
\begin{align}
&\Big|\int_{\R^{\d+\k}}\zg (\p_i \tv)^{\otimes 2} : \mu^{(n-2),f} (\phi-\bar\phi)\Big| \nn\\
&\le C_n (\ell\|\nab^{\otimes 2}\tv\|_{L^\infty})^{n-2}\|\nab h^f\|_{L_\ga^2(\supp\nab\tv)}(\|\nab X\|_{L_\ga^2} + \ell^{-1}\|X\|_{L_\ga^2})\nn\\
&\le C_n (\ell\|\nab^{\otimes 2}\tv\|_{L^\infty})^{n-2}\|\nab h^f\|_{L_\ga^2(\supp\nab\tv)} \Big( \|\nab\tv\|_{L^\infty}(\ell\|\nab^{\otimes 2}\tv\|_{L^\infty})\|\nab\phi\|_{L_\ga^2} \nn\\
&\ph\qquad+ \ell^{-1}(\ell\|\nab\tv\|_{L^\infty}^2 \|\nab\phi\|_{L_\ga^2})\Big) \nn\\
&\le C_n(\ell\|\nab^{\otimes 2}\tv\|_{L^\infty})^{n}\|\nab h^f\|_{L_\ga^2(\supp\nab\tv)}  \|\nab\phi\|_{L_\ga^2}.\label{eq:nabknpre4}
\end{align}

Combining the estimates \eqref{eq:nabknpre1}, \eqref{eq:nabknpre2}, \eqref{eq:nabknpre3}, \eqref{eq:nabknpre4}, we arrive at
\begin{multline}\label{eq:nabknffin}
\int_{\R^{\d+\k} }\zg \nab\phi \cdot \nab \ka^{(n),f} \le  C_n \|\nab\tv\|_{L^\infty}^n \|\nab\phi\|_{L_\ga^2} \|\nab h^f\|_{L_\ga^2(\supp\nab\tv)}\\
+C_n(\ell \|\nab^{\otimes 2}\tv\|_{L^\infty})^{n} \|\nab h^f\|_{L_\ga^2(\supp\nab\tv)}\|\nab\phi\|_{L_\ga^2}.
\end{multline}
This shows that \eqref{inducphi} holds at order $n$, assuming that \eqref{induck}, \eqref{inducnu}, \eqref{inducmu} hold up to order $n-1$. In particular, taking $\phi = \ka^{(n),f}$, dividing both sides by $\|\nab\ka^{(n),f}\|_{L_\ga^2}$, and simplifying the resulting right-hand side, using \eqref{bornetv}, yields \eqref{induck} at order $n$. 

\medskip

Using the recursion \eqref{recnu} and the triangle inequality, we deduce from \eqref{eq:nabknffin} that 
\begin{align}
\|\nu^{(n),f}\|_{L_\ga^2} &\le \|\nab\ka^{(n),f}\|_{L_\ga^2} + C_n \|\nab \tv\|_{L^\infty}  \|\nu^{(n-1),f}\|_{L_\ga^2} \nn\\
&\le C_n(\ell\|\nab^{\otimes 2}\tv\|_{L^\infty})^{n}\|\nab h^f\|_{L_\ga^2(\supp\nab\tv)}  + C_n\|\nab \tv\|_{L^\infty}(\ell\|\nab^{\otimes 2}\tv\|_{L^\infty})^{n-1} \|\nab h^f\|_{L_\ga^2(\supp\nab\tv)} \nn\\
&\le C_n (\ell\|\nab^{\otimes 2}\tv\|_{L^\infty})^n\|\nab h^f\|_{L_\ga^2(\supp\nab\tv)},\label{eq:nunffin}
\end{align}
where we also use \eqref{bornenabtv} and the induction hypothesis \eqref{inducnu} for $\nu^{(n-1),f}$ in obtaining the second inequality. This shows that \eqref{inducnu} holds at order $n$.

Finally, let $X=(X_{ij})_{i,j=1}^{\d+\k}$ be a matrix field satisfying the support condition in the statement of the theorem. Using the recursion \eqref{recmu} for $\mu^{(n),f}$,
\begin{align}
\int_{\R^{\d+\k}} \zg \mu^{(n), f} :X =
\int_{\R^{\d+\k}} \zg \partial_i \nu_j^{(n),f}X_{ij}- n  \int_{\R^{\d+\k}} \zg  \partial_i \tv^k \mu_{jk}^{(n-1),f}X_{ij}. \label{eq:munXpre}
\end{align}
Integrating by parts $\p_i$ and using Cauchy-Schwarz, the support assumption for $X$, and \eqref{eq:nunffin},
\begin{align}
\int_{\R^{\d+\k}} \zg \partial_i \nu_j^{(n),f}X_{ij} &= -\delta_{(\d+\k)i}\ga \int_{\R^{\d+\k}}z |z|^{\ga-1} \nu_j^{(n),f} X_{ij} - \int_{\R^{\d+\k}}\zg \nu_j^{(n),f}\p_i X_{ij} \nn\\
&\le C \ell^{-1} \|\nu^{(n),f}\|_{L_\ga^2} \|X\|_{L_\ga^2} + \|\nu^{(n),f}\|_{L_\ga^2} \|\nab X\|_{L_\ga^2} \nn\\
&\le C(\ell\|\nab^{\otimes 2}\tv\|_{L^\infty})^{n} (\ell^{-1} \|X\|_{L_\ga^2}+ \|\nab X\|_{L_\ga^2}).  \label{eq:munffin'}
\end{align}
Next, define the tensor
\begin{align}
\tl{X}_{jk} \coloneqq \p_i\tv^k X_{ij}, \qquad j,k\in[\d+\k],
\end{align}
which satisfies the support condition because $X_{i(\d+\k)}$ is supported in $\supp\nab v\times \{\ell \le |z| \le 2\ell\}^\k$ and $\tv^{\d+\k} = 0$ if $\k=1$. Moreover, $\tl{X}$ satisfies
\begin{align}
\|\tl{X}\|_{L_\ga^2} \le \|\nab \tv\|_{L^\infty} \|X\|_{L_\ga^2}, \label{eq:tXLga2}\\
\|\nab \tl{X}\|_{L_\ga^2} \le \|\nab^{\otimes 2}\tv\|_{L^\infty}\|X\|_{L_\ga^2} + \|\nab \tv\|_{L^\infty} \|\nab X\|_{L_\ga^2}. \label{eq:nabtXLga2}
\end{align}
Therefore, we may apply the induction hypothesis for $\mu^{(n-1),f}$ to obtain
\begin{align}
- n  \int_{\R^{\d+\k}} \zg  \partial_i \tv^k \mu_{jk}^{(n-1),f}X_{ij} &= -n\int_{\R^{\d+\k}} \zg \mu^{(n-1),f} : \tl{X} \nn\\
&\le C_n   (\ell \|\nab^{\otimes2} v\|_{L^\infty})^{n-1} \|\nab h^f\|_{L_\ga^2(\supp\nab\tv)}(\|\nab  \tl{X}\|_{L^2_{\gamma}}   +\ell^{-1}\|\tl X\|_{L^2_{\gamma}}) \nn\\
&\le C_n (\ell \|\nab^{\otimes2} \tv\|_{L^\infty})^{n}\|\nab h^f\|_{L_\ga^2(\supp\nab\tv)}(\|\nab{X}\|_{L^2_{\gamma}}   +\ell^{-1}\| X\|_{L^2_{\gamma}}), \label{eq:munffin}
\end{align}
where the final inequality follows from applying the bounds \eqref{eq:tXLga2}, \eqref{eq:nabtXLga2} to the second line. Applying \eqref{eq:munffin'}, \eqref{eq:munffin} to our starting point \eqref{eq:munXpre} shows \eqref{inducmu} at order $n$. With this last estimate, the proof of the induction step---and therefore, the proof of the lemma---is complete.
\end{proof}

\cref{lem:inducknumu} takes care of the estimate \eqref{eq:comm2}. For \eqref{eq:comm2dual}, we write the left-hand side as
\begin{align}
\frac{1}{\cd}\int_{\R^{\d+\k}}\ka^{(n),f}Lh^w &= \frac{1}{\cd}\int_{\R^{\d+\k}}\nab\ka^{(n),f}\cdot \nab h^w,
\end{align}
where we have integrated by parts once to obtain the right-hand side. Appealing to the estimate \eqref{inducphi} with $\phi = h^w$ then yields the desired \eqref{eq:comm2dual}.

\subsection{An improved estimate for $n\le 2$}\label{ssec:L2regimp}

In this subsection, we prove the estimates \eqref{eq:comm2'}, \eqref{eq:comm2dual'} of \cref{thm:comm2}. These estimates improve upon \eqref{eq:comm2}, \eqref{eq:comm2dual} for $n\le 2$, in that they only depend on $\|\nab\tv\|_{L^\infty}$ and they require no assumption of localization for the transport $\tv$. The $n=1$ case has already been treated in \cref{ssec:FOcst}, so we only consider $n=2$.

We start by massaging the right-hand side of $L\ka^{(2),f}$ into divergence form---a rather involved exercise with the product rule and the recursions for $\ka,\nu,\mu$.

\begin{lemma}\label{lem:Lka2}
For any $f\in \Sc_0(\R^\d)$, we have 
\begin{multline}\label{eq:Lkap2f}
L\kappa^{(2),f} = - 2\p_{i_1}(\zg \p_{i_1}\tv^{i_2}\nu_{i_2}^{(1),f}) -2\p_{i_2}(\zg \p_{i_1}\tv^{i_2}\nu_{i_1}^{(1),f})+2\p_{i_1}(\zg \p_{i_2}\tv^{i_2}\nu_{i_1}^{(1),f})\\
-2 \p_{i_2}\Big(\zg\p_{i_1}\tv^{i_1} \p_{i_3}\tv^{i_2}\p_{i_3}h^f\Big)  +2\p_{i_2}\Big(\zg\p_{i_3}\tv^{i_1}\p_{i_1}v^{i_2}\p_{i_3}h^{f}\Big)  \\
+ \p_{i_3}\Big(\zg\p_{i_1}\tv^{i_1}\p_{i_2}\tv^{i_2}\p_{i_3}h^f\Big) - \p_{i_3}\Big(\zg\p_{i_2}\tv^{i_1}\p_{i_1}\tv^{i_2}\p_{i_3}h^f\Big).
\end{multline}
\end{lemma}
\begin{proof}
By \cref{lem:Lkapnf}, we have
\begin{multline}\label{eq:Lkappa2sp}
L\ka^{(2),f} = \cd \Big(\p_{i_{1}}\tv^{i_1} \p_{i_2}\tv^{i_2} + \p_{i_2}\tv^{i_1}\p_{i_1}\tv^{i_2}\Big)\tl{f} - 2\p_{i_1}(\zg \p_{i_1}\tv^{i_2}\nu_{i_2}^{(1),f})\\
-2\zg \p_{i_1}\tv^{i_2} \p_{i_1}\nu_{i_2}^{(1),f} + 2\zg \p_{i_1}\tv^{i_2}\p_{i_1}v^{i_3} \mu_{i_2i_3}^{(0),f}.
\end{multline}
Applying the recursion \eqref{recnu} for $\nu^{(1),f}$, 
\begin{align}\label{eq:Lkappa21}
-2\zg \p_{i_1}\tv^{i_2} \p_{i_1}\nu_{i_2}^{(1),f} = -2\zg\p_{i_1}\tv^{i_2}\p_{i_2}\p_{i_1}\ka^{(1),f} + 2\zg \p_{i_1}\tv^{i_2}\p_{i_1}\Big(\p_{i_2}v^{i_3}\nu_{i_3}^{(0),f}\Big).
\end{align}
Reverse distributing $\p_{i_2}$,
\begin{align}
-2\zg\p_{i_1}\tv^{i_2}\p_{i_2}\p_{i_1}\ka^{(1),f} &= -2\p_{i_2}(\zg\p_{i_1}\tv^{i_2}\p_{i_1}\ka^{(1),f}) + 2\zg \p_{i_1}\p_{i_2}\tv^{i_2} \p_{i_1}\ka^{(1),f}, \label{eq:Lkappa22}
\end{align}
where we have implicitly used that $\tv^{i_2} = 0$ for $i_2=\d+\k$, if $\k=1$, so there is no contribution from $\p_{i_2}\zg$. 
Distributing $\p_{i_1}$,
\begin{align}\label{eq:Lkappa23}
2\zg \p_{i_1}\tv^{i_2}\p_{i_1}\Big(\p_{i_2}\tv^{i_3}\nu_{i_3}^{(0),f}\Big) &= 2\zg\p_{i_1}\tv^{i_2}\p_{i_1}\p_{i_2}\tv^{i_3}\nu_{i_3}^{(0),f} + 2\zg \p_{i_1}\tv^{i_2}\p_{i_2}\tv^{i_3}\p_{i_1}\nu_{i_3}^{(0),f}.
\end{align}
Reverse distributing $\p_{i_2}$,
\begin{align}
& 2\zg\p_{i_1}\tv^{i_2}\p_{i_1}\p_{i_2}\tv^{i_3}\nu_{i_3}^{(0),f} \nn\\
 &=2 \p_{i_2}\Big(\zg\p_{i_1}\tv^{i_2}\p_{i_1}\tv^{i_3} \nu_{i_3}^{(0),f}\Big) - 2\zg \p_{i_1}\p_{i_2}\tv^{i_2}\p_{i_1}\tv^{i_3}\nu_{i_3}^{(0),f}-2\zg\p_{i_1}\tv^{i_2}\p_{i_1}\tv^{i_3}\p_{i_2}\nu_{i_3}^{(0),f} \nn\\
 &=2\p_{i_2}\Big(\zg\p_{i_1}\tv^{i_2}\p_{i_1}\ka^{(1),f}\Big) -2\p_{i_2}\Big(\zg \p_{i_1}\tv^{i_2}\nu_{i_1}^{(1),f}\Big) - 2\zg \p_{i_1}\p_{i_2}\tv^{i_2}\p_{i_1}\ka^{(1),f}  \nn\\
 &\ph+ 2\zg \p_{i_1}\p_{i_2}\tv^{i_2}\nu_{i_1}^{(1),f}-2\zg\p_{i_1}\tv^{i_2}\p_{i_1}\tv^{i_3}\p_{i_2}\nu_{i_3}^{(0),f}, \label{eq:Lkappa24}
\end{align}
where the final equality follows from the recursion for $\nu_{i_1}^{(1),f}$ and we have again used that $\tv^{i_2} = 0$ for $i_2=\d+\k$, if $\k=1$.
Combining \eqref{eq:Lkappa21}-\eqref{eq:Lkappa24}, we obtain
\begin{align}
-2\zg \p_{i_1}\tv^{i_2} \p_{i_1}\nu_{i_2}^{(1),f} &= -\underbrace{2\p_{i_2}(\zg\p_{i_1}\tv^{i_2}\p_{i_1}\ka^{(1),f})}_{T_1} + \underbrace{2\zg \p_{i_1}\p_{i_2}\tv^{i_2} \p_{i_1}\ka^{(1),f}}_{T_2} \nn\\
&\ph+\underbrace{2\p_{i_2}(\zg\p_{i_1}\tv^{i_2}\p_{i_1}\ka^{(1),f})}_{-T_1} -2\p_{i_2}(\zg \p_{i_1}\tv^{i_2}\nu_{i_1}^{(1),f}) \nn\\
&\ph- \underbrace{2\zg \p_{i_1}\p_{i_2}\tv^{i_2}\p_{i_1}\ka^{(1),f}}_{-T_2} + 2\zg \p_{i_1}\p_{i_2}\tv^{i_2}\nu_{i_1}^{(1),f} \nn\\
 &\ph-2\zg\p_{i_1}\tv^{i_2}\p_{i_1}\tv^{i_3}\p_{i_2}\nu_{i_3}^{(0),f}+ 2\zg \p_{i_1}\tv^{i_2}\p_{i_2}\tv^{i_3}\p_{i_1}\nu_{i_3}^{(0),f} \nn\\
 &=-2\p_{i_2}(\zg \p_{i_1}\tv^{i_2}\nu_{i_1}^{(1),f})+2 \p_{i_1}(\zg\p_{i_2}\tv^{i_2}\nu_{i_1}^{(1),f}) - 2\p_{i_2}\tv^{i_2}\div(\zg \nu^{(1),f}) \nn\\
 &\ph-2\zg\p_{i_1}\tv^{i_2}\p_{i_1}\tv^{i_3}\p_{i_2}\nu_{i_3}^{(0),f}+ 2\zg \p_{i_1}\tv^{i_2}\p_{i_2}\tv^{i_3}\p_{i_1}\nu_{i_3}^{(0),f}. \label{eq:Lkappa25}
\end{align}
By the recursion \eqref{recmu},
\begin{align}\label{eq:Lkappa25'}
\p_{i_2}\nu_{i_3}^{(0),f} =  \mu_{i_2i_3}^{(0),f},
\end{align}
which in turn implies that
\begin{align}
-2\zg \p_{i_1}\tv^{i_2} \p_{i_1}\nu_{i_2}^{(1),f} + 2\zg \p_{i_1}\tv^{i_2}\p_{i_1}\tv^{i_3} \mu_{i_2i_3}^{(0),f} &= 2\p_{i_1}(\zg\p_{i_2}\tv^{i_2}\nu_{i_1}^{(1),f}) - 2\p_{i_2}\tv^{i_2}\div(\zg \nu^{(1),f}) \nn\\
&\ph+ 2\zg \p_{i_1}\tv^{i_2}\p_{i_2}\tv^{i_3}\p_{i_1}\nu_{i_3}^{(0),f}. \label{eq:Lkappa26}
\end{align}
Using that
\begin{align}\label{eq:Lnu1f}
-\div(\zg\nu^{(1),f}) =  -\p_{i_1}(\zg\p_{i_2}\tv^{i_1}\nu_{i_2}^{(0),f}) + \p_{i_2}(\zg\p_{i_1}\tv^{i_1}\nu_{i_2}^{(0),f}),
\end{align}
by consequence of the identity \eqref{eq:Lka1} for $L\ka^{(1),f}$ and the recursion \eqref{recnu} for $\nu^{(1),f}$, it follows that
\begin{align}
- 2\p_{i_1}\tv^{i_1}\div(\zg \nu^{(1),f}) &= -2\p_{i_1}\tv^{i_1}\p_{i_2}(\zg\p_{i_3}\tv^{i_2}\nu_{i_3}^{(0),f}) + 2\p_{i_1}\tv^{i_1}\p_{i_3}(\zg\p_{i_2}\tv^{i_2}\nu_{i_3}^{(0),f}) \nn\\
&=-2 \p_{i_2}\Big(\zg\p_{i_1}\tv^{i_1} \p_{i_3}\tv^{i_2}\nu_{i_3}^{(0),f}\Big) + \underbrace{2\zg \p_{i_2}\p_{i_1}\tv^{i_1}\p_{i_3}\tv^{i_2}\nu_{i_3}^{(0),f}}_{T_3} \nn\\
&\ph +2\p_{i_3}\Big(\zg\p_{i_1}\tv^{i_1}\p_{i_2}\tv^{i_2}\nu_{i_3}^{(0),f}\Big) -\underbrace{2\zg\p_{i_3}\p_{i_1}\tv^{i_1}\p_{i_2}\tv^{i_2}\nu_{i_3}^{(0),f}}_{T_4}, \label{eq:Lkappa27}
\end{align}
where the final equality follows from reverse distributing $\p_{i_2}$ in the first term, respectively $\p_{i_3}$ in the second. Since $\p_{i_1}\nu_{i_3}^{(0),f} = \p_{i_3}\nu_{i_1}^{(0),f}$, we may reverse product rule $\p_{i_3}$ to obtain
\begin{align}
2\zg \p_{i_1}\tv^{i_2}\p_{i_2}\tv^{i_3}\p_{i_1}\nu_{i_3}^{(0),f} &= 2\p_{i_3}\Big(\zg\p_{i_1}\tv^{i_2}\p_{i_2}v^{i_3}\nu_{i_1}^{(0),f}\Big) - 2\zg\p_{i_1}\p_{i_3}\tv^{i_2}\p_{i_2}\tv^{i_3}\nu_{i_1}^{(0),f}  \nn\\
&\ph- 2\zg\p_{i_1}\tv^{i_2}\p_{i_2}\p_{i_3} \tv^{i_3}\nu_{i_1}^{(0),f} \nn\\
&=2\p_{i_2}\Big(\zg\p_{i_3}\tv^{i_1}\p_{i_1}\tv^{i_2}\nu_{i_3}^{(0),f}\Big) -\underbrace{2\zg\p_{i_3}\p_{i_2}\tv^{i_1}\p_{i_1}\tv^{i_2}\nu_{i_3}^{(0),f}}_{T_5} \nn\\
&\ph-\underbrace{2\zg\p_{i_3}\tv^{i_1}\p_{i_1}\p_{i_2}\tv^{i_2}\nu_{i_3}^{(0),f}}_{-T_3}, \label{eq:Lkappa28}
\end{align}
where the final equality follows from the relabeling $(1,2,3) \mapsto (3,1,2)$, and we have implicitly used that $v^{\d+\k}=0$, so that there is no contribution from the derivative hitting the weight. Finally, writing $\cd \tl{f} = -\p_{i_3}(\zg\nu_{i_3}^{(0),f})$ and redistributing $\p_{i_3}$, we see that
\begin{multline}
 \cd \Big(\p_{i_{1}}\tv^{i_1} \p_{i_2}\tv^{i_2} + \p_{i_2}\tv^{i_1}\p_{i_1}\tv^{i_2}\Big) \tl f = -\p_{i_3}\Big(\zg (\p_{i_{1}}\tv^{i_1} \p_{i_2}\tv^{i_2} + \p_{i_2}\tv^{i_1}\p_{i_1}\tv^{i_2})\nu_{i_3}^{(0),f}\Big)  \\
 +\underbrace{2 \zg\p_{i_3}\p_{i_1}\tv^{i_1}\p_{i_2}\tv^{i_2}\nu_{i_3}^{(0),f}}_{-T_4}+ \underbrace{2\zg\p_{i_3}\p_{i_2}\tv^{i_1}\p_{i_1}\tv^{i_2} \nu_{i_3}^{(0),f}}_{-T_5}. \label{eq:Lkappa29}
\end{multline}
After a little bookkeeping, we arrive at
\begin{align}
&\cd \Big(\p_{i_{1}}\tv^{i_1} \p_{i_2}\tv^{i_2} + \p_{i_2}\tv^{i_1}\p_{i_1}\tv^{i_2}\Big)\tl{f}-2\zg \p_{i_1}\tv^{i_2} \p_{i_1}\nu_{i_2}^{(1),f} + 2\zg \p_{i_1}\tv^{i_2}\p_{i_1}v^{i_3} \mu_{i_2i_3}^{(0),f} \nn\\
&= 2\p_{i_1}(\zg \p_{i_2}\tv^{i_2}\nu_{i_1}^{(1),f}) -2\p_{i_2}\Big(\zg \p_{i_1}\tv^{i_1} \p_{i_3}\tv^{i_2}\nu_{i_3}^{(0),f}\Big)  +2\p_{i_3}\Big(\zg\p_{i_1}\tv^{i_1}\p_{i_2}\tv^{i_2}\nu_{i_3}^{(0),f}\Big) \nn\\
&\ph+2\p_{i_2}\Big(\zg\p_{i_3}\tv^{i_1}\p_{i_1}\tv^{i_2}\nu_{i_3}^{(0),f}\Big) -\p_{i_3}\Big(\zg (\p_{i_{1}}\tv^{i_1} \p_{i_2}\tv^{i_2} + \p_{i_2}\tv^{i_1}\p_{i_1}\tv^{i_2})\nu_{i_3}^{(0),f}\Big) \nn\\
&=2\p_{i_1}(\zg\p_{i_2}\tv^{i_2}\nu_{i_1}^{(1),f}) -2 \p_{i_2}\Big(\zg\p_{i_1}\tv^{i_1} \p_{i_3}\tv^{i_2}\nu_{i_3}^{(0),f}\Big)  +2\p_{i_2}\Big(\zg\p_{i_3}\tv^{i_1}\p_{i_1}\tv^{i_2}\nu_{i_3}^{(0),f}\Big) \nn\\
&\ph + \p_{i_3}\Big(\zg\p_{i_1}\tv^{i_1}\p_{i_2}\tv^{i_2}\nu_{i_3}^{(0),f}\Big) - \p_{i_3}\Big(\zg\p_{i_2}\tv^{i_1}\p_{i_1}\tv^{i_2}\nu_{i_3}^{(0),f}\Big). 
\end{align}
Recalling our starting point \eqref{eq:Lkappa2sp}, relabeling indices, and recalling that $\nu^{(0),f} = \nab h^f$, we conclude the desired \eqref{eq:Lkap2f}.
\end{proof}

Integrating both sides of the equation \eqref{eq:Lkap2f} against a test function $\phi$ over $\R^{\d+\k}$, we obtain
\begin{multline}
\int_{\R^{\d+\k}}\phi L\ka^{(2),f} = \int_{\R^{\d+\k}}\phi\Bigg(- 2\p_{i_1}(\zg \p_{i_1}\tv^{i_2}\nu_{i_2}^{(1),f}) -2\p_{i_2}(\zg \p_{i_1}\tv^{i_2}\nu_{i_1}^{(1),f}) \\
+2\p_{i_1}(\zg \p_{i_2}\tv^{i_2}\nu_{i_1}^{(1),f})-2 \p_{i_2}\Big(\zg\p_{i_1}\tv^{i_1} \p_{i_3}\tv^{i_2}\p_{i_3}h^f\Big)  +2\p_{i_2}\Big(\zg\p_{i_3}\tv^{i_1}\p_{i_1}\tv^{i_2}\p_{i_3}h^{f}\Big)  \\
+ \p_{i_3}\Big(\zg\p_{i_1}\tv^{i_1}\p_{i_2}\tv^{i_2}\p_{i_3}h^f\Big) - \p_{i_3}\Big(\zg\p_{i_2}\tv^{i_1}\p_{i_1}\tv^{i_2}\p_{i_3}h^f\Big)\Bigg).
\end{multline}
Integrating by parts in the right-hand side and using the triangle inequality plus Cauchy-Schwarz, we obtain
\begin{multline}
\int_{\R^{\d+\k}}\phi L\ka^{(2),f} \le C\|\nab\tv\|_{L^\infty}\|\nab\phi\|_{L_\ga^2(\supp\nab \tv)}\Big(\|\nab\tv\|_{L^\infty} \|\nab h^f\|_{L_\ga^2(\supp\nab\tv)} \\
+  \| \nu^{(1),f}\|_{L_\ga^2(\supp\nab \tv)}\Big).
\end{multline}
Inserting the bound \eqref{eq:nu1L2} for $\|\nu^{(1),f}\|_{L_\ga^2}$, we arrive at
\begin{align}
\int_{\R^{\d+\k}}\phi L\ka^{(2),f} \le C\|\nab\tv\|_{L^\infty}^2 \|\nab\phi\|_{L_\ga^2(\supp\nab \tv)} \|\nab h^f\|_{L_\ga^2(\supp\nab\tv)}.
\end{align}
Taking $\phi = \ka^{(2),f}$ and integrating by parts in the left-hand side yields \eqref{eq:comm2'}. Taking $\phi = \frac{1}{\cd}h^w$, for $w\in\Sc_0(\R^\d)$, integrating by parts, and using that $Lh^w = \cd\tl{w}$ yields \eqref{eq:comm2dual'}.

\subsection{(Unlocalized) $L^2$ commutator estimates}\label{ssec:L2regunloc}
In this subsection, we prove the estimates \eqref{eq:comm2''}, \eqref{eq:comm2dual''}, which will then complete the proof of \cref{thm:comm2}. The difference between these estimates and the preceding ones is that there is no localization to the support of $\nab\tv$ in the right-hand side, which is suitable when we do not care about localizing the modulated energy. There is a quick way to obtain unlocalized $L^2$ commutator estimates (without needing the extension to $\R^{\d+\k}$) following the approach of the authors' work \cite{NRS2021} with Q.H. Nguyen. Although the cited paper only considered estimates up to second order, the argument works just as well at higher order. Accordingly, we only sketch it, omitting the steps to justify the integration by parts and instead referring to \cite[Section 6.2]{NRS2021} for the details.


To compactify the notation, we set
\begin{align}
\k_v(x,y) \coloneqq (v(x)-v(y))^{\otimes n} : \nab^{\otimes n}\g(x-y), \qquad \forall x,y\in(\R^\d)^2\setminus\triangle
\end{align} 
for the kernel of $\ka^{(n)}$. For $f,g\in\Sc_0(\R^\d)$, write
\begin{align}
f = \div \nab \Delta^{-1}f \eqqcolon \div f_1,\label{eq:f1def}\\
g = \div \nab\Delta^{-1}g \eqqcolon \div g_1, \label{eq:g1def}
\end{align}
where $f_1,g_1$ are vector-valued.

Integrating by parts in $x$ and $y$, we see that
\begin{align}\label{eq:f1g1presym}
\int_{(\R^\d)^2}\k_v(x,y)f(x)g(y) = \int_{(\R^\d)^2}\nab_x\nab_y\k_v(x,y) : f_1(x)\otimes g_1(y).
\end{align}
Since
\begin{align}
\int_{\R^\d}\nab_x\nab_y\k_v(x,y)dx =\int_{\R^\d} \nab_x\nab_y\k_v(x,y) dy = 0,
\end{align}
by the fundamental theorem of calculus, we may symmetrize the right-hand side of \eqref{eq:f1g1presym} to obtain
\begin{align}\label{eq:f1g1sym}
\int_{(\R^\d)^2}\k_v(x,y)f(x)g(y) = \frac12\int_{(\R^\d)^2}\nab_x\nab_y\k_v(x,y) : (f_1(x)-f_1(y))\otimes (g_1(x)-g_1(y)).
\end{align} 

Next, observe from the product rule that
\begin{multline}\label{eq:nabxykv}
\nab_x\nab_y \k_v(x,y) = -\nab^{\otimes( n+2)}\g(x-y) : (v(x)-v(y))^{\otimes n} \\
-n\nab^{\otimes(n+1)}\g(x-y) : (\nab v(x)+\nab v(y))\otimes (v(x)-v(y))^{\otimes( n-1)} \\
-n(n-1) \nab^{\otimes n}\g(x-y) : \nab v(x)\otimes \nab v(y) \otimes (v(x)-v(y))^{\otimes(n-2)}.
\end{multline}
So, by the triangle inequality and mean-value theorem, it follows that
\begin{align}
|\nab_x\nab_y\k_v(x,y)| \le C\frac{\|\nab v\|_{L^\infty}^{n}}{|x-y|^{\s+2}}.
\end{align}

In the super-Coulomb case $\d-2<\s<\d$, we may directly estimate
\begin{align}
&\int_{(\R^\d)^2}|\nab_x\nab_y\k_v(x,y) : (f_1(x)-f_1(y))\otimes (g_1(x)-g_1(y))| \nn\\
&\leq C\|\nab v\|_{L^\infty}^n \int_{(\R^\d)^2} \frac{|f_1(x)-f_1(y)| |g_1(x)-g_1(y)|}{|x-y|^{\s+2}} \nn\\
&\le  C\|\nab v\|_{L^\infty}^n \Big(\int_{(\R^\d)^2} \frac{|f_1(x)-f_1(y)|^2}{|x-y|^{\s+2}}\Big)^{1/2}\Big(\int_{(\R^\d)^2} \frac{|g_1(x)-g_1(y)|^2}{|x-y|^{\s+2}}\Big)^{1/2} \nn\\
&\le C\|\nab v\|_{L^\infty}^n  \|f_1\|_{\dot{H}^{\frac{2+\s-\d}{2}}}\|g_1\|_{\dot{H}^\frac{2+\s-\d}{2}} \label{eq:kvf1g1bnd}
\end{align}
where the third line follows from Cauchy-Schwarz and the fourth line from the difference quotient formulation of the $\dot{H}^{\frac{2+\s-\d}{2}}$ seminorm (e.g. see \cite[Proposition 3.4]{dNPV2012}). Recalling the definitions \eqref{eq:f1def}, \eqref{eq:g1def} of $f_1,g_1$ and appealing to Plancherel's theorem,
\begin{align}
\|f_1\|_{\dot{H}^{\frac{2+\s-\d}{2}}} \lesssim \|f\|_{\dot{H}^{\frac{\s-\d}{2}}} = \|h^f\|_{\dot{H}^{\frac{\d-\s}{2}}}, \label{eq:f1Hdot}\\
\|g_1\|_{\dot{H}^{\frac{2+\s-\d}{2}}} \lesssim \|g\|_{\dot{H}^{\frac{\s-\d}{2}}} = \|h^g\|_{\dot{H}^{\frac{\d-\s}{2}}}. \label{eq:g1Hdot}
\end{align}
Combining \eqref{eq:f1g1sym}, \eqref{eq:kvf1g1bnd}, we arrive at the desired conclusion.

In the more delicate Coulomb case $\s=\d-2$, we use the fundamental theorem of calculus to rewrite the identity \eqref{eq:nabxykv} as 
\begin{multline}
\p_{x_i}\p_{y_j}\k_v(x,y) = -\underbrace{\k_{v,ijl_1\cdots l_n}^{p_1\cdots p_n}(x-y) \prod_{r=1}^n \int_0^1 \p_{p_r}v^{l_r}(tx+(1-t)y)dt}_{\eqqcolon \k_{v,ij}^{1}(x,y)} \\
-n\p_i v^{l_1}(x)\underbrace{\k_{v,jl_1\cdots l_n}^{p_2\cdots p_n}(x-y) \prod_{r=2}^n \int_0^1 \p_{p_r}v^{l_r}(tx+(1-t)y)dt}_{\eqqcolon \k_{v,jl_1}^{2}(x,y)} \\
-n\p_j v^{l_1}(y)\k_{v,il_1\cdots l_n}^{p_2\cdots p_n}(x-y) \prod_{r=2}^n \int_0^1 \p_{p_r}v^{l_r}(tx+(1-t)y)dt\\
-n(n-1)\p_i v^{l_1}(x)\p_j v^{l_2}(y) \underbrace{\k_{v,l_1\cdots l_n}^{p_3\cdots p_n}(x-y)\prod_{r=3}^n \int_0^1 \p_{p_r}v^{l_r}(tx+(1-t)y)dt}_{\eqqcolon \k_{v,l_1l_2}^{3}(x,y)},
\end{multline}
where
\begin{align}
\k_{v,ijl_1\cdots l_n}^{p_1\cdots p_n}(x-y) \coloneqq  \p_i\p_j\p_{l_1}\cdots\p_{l_n}\g(x-y) \prod_{r=1}^n (x-y)^{p_r} \\
\k_{v,il_1\cdots l_n}^{p_2\cdots p_n}(x-y) \coloneqq  \p_i\p_{l_1}\cdots\p_{l_n}\g(x-y)\prod_{r=2}^n (x-y)^{p_r} \\
\k_{v,l_1\cdots l_n}^{p_3\cdots p_n}(x-y) \coloneqq\p_{l_1}\cdots\p_{l_n}\g(x-y)\prod_{r=3}^n (x-y)^{p_r} 
\end{align}
are standard kernels associated to Calder\'{o}n-Zygmund operators \cite[Theorem 5.4.1]{Grafakos2014c}. Consequently, the kernels $\k_{v,ij}^{1}(x,y), \k_{v,il_1}^{2}(x,y), \k_{v,l_1l_2}^{3}(x,y)$ respectively define Calder\'{o}n $\d$-commutators $T_{v,ij}^{1},T_{v,il_1}^{2},T_{v,l_1l_2}^{3}$, to which we can apply the Christ-Journ\'{e} theorem \cite{CJ1987}. Using this result together with Cauchy-Schwarz, we find that
\begin{align}
\int_{(\R^\d)^2}\k_v(x,y)f(x)g(y) &= -\int_{\R^\d}f_1^i T_{v,ij}^{1}(g_1^j)   - n\int_{\R^\d}\p_i v^{l_1} f_1^i T_{v,jl_1}^{2}(g_1^j) \nn\\
&\ph - n\int_{\R^\d}\p_j v^{l_1}g_1^j T_{v,il_1}^{2}(f_1^i)  -n(n-1)\int_{\R^\d}\p_j v^{l_2}g_1^j T_{v,l_1l_2}^{3}(\p_i v^{l_1} f_1^i) \nn\\
&\le C\|\nab v\|_{L^\infty}^n \|f_1\|_{L^2} \|g_1\|_{L^2},
\end{align}
where we have implicitly used the $L^2$ operator norm bounds
\begin{align}
\|T_{v,ij}^{1}\|_{L^2\rightarrow L^2} \leq (C\|\nab v\|_{L^\infty})^n ,\\
\| T_{v,il_1}^{2} \|_{L^2\rightarrow L^2} \le (C\|\nab v\|_{L^\infty})^{n-1},\\
\| T_{v,l_1l_2}^{3}\|_{L^2\rightarrow L^2} \le (C\|\nab v\|_{L^\infty})^{n-2},
\end{align}
where $C>0$ depends only on $\d$. Inserting the bounds \eqref{eq:f1Hdot}, \eqref{eq:g1Hdot} for $\|f_1\|_{L^2},\|g_1\|_{L^2}$, we arrive at the desired conclusion. This completes the proof.

\subsection{$L^\infty$ local regularity theory for fractional Laplacians in extended space form}\label{ssec:L2regLinf}
In this subsection, we turn to establishing the local regularity theory for commutators based on the elliptic regularity theory for divergence-form operators with $A_2$ weights developed by Fabes \emph{et al.} \cite{FKS1982}.

The following regularity result for the operator $L$ is  adapted from \cite{FKS1982}.

\begin{lemma} \label{lem:regularity}
Suppose that $w$ satisfies
\begin{align}\label{eq:Lwdivphi}
Lw =   \div \phi\quad \text{in} \ B(0,\eta)\subset \R^{\d+\k},
\end{align}
where $|z|^{-\ga}\phi $ is $m$-times differentiable in the $x$ variables,  for $m\ge 1$.  Then for any ball $B(0,\eta)\subset \R^{\d+\k}$, we have
\begin{equation}\label{resultl}
\|\nab_x^{\otimes m} w\|_{L^\infty(B(0, \eta/2))} \le \frac{C}{\eta^{\frac{\s}2+m}}\|\nab w\|_{L_\ga^2(B(0,\eta))}  +C  \sum_{k=1}^m \eta^{k-m+1} \||z|^{-\ga} \nab^{\otimes k}_x\phi\|_{L^\infty(B(0,\eta))},
\end{equation}
where $\nab_x\coloneqq(\p_1,\ldots,\p_{\d})$.

\end{lemma}
\begin{proof}
First, we note that each partial derivative $\partial_{i} w$ satisfies the equation \eqref{eq:Lwdivphi} with $\partial_{i}\phi$ instead of $\phi$. Next, we  write $\partial_{j} w=u+v$ where 
\begin{align}\label{equ}
\left\{\begin{array}{ll}
Lu= \div \partial_{j}\phi, &\text{in } B(0,\frac34\eta)\\
u=0, & \text{on} \ \partial B(0,\frac34\eta)\end{array}\right.
\end{align}
  and $v$ solves
\begin{equation}\label{eqv}
Lv=0 \quad  \text{in} \ B(0, \frac34\eta).
\end{equation}

As $\zg$ is an $A_2$ weight, \cite[Theorem 2.3.12]{FKS1982} applies, and with the help of  \eqref{eq:defgamma},  we find there exists $\sigma>0$ such that 
\begin{equation}\label{eqv'}
\| v\|_{\dot{C}^{0,\sigma}(B(0, \frac\eta2 ) ) } \le \frac{C}{\eta^{\sigma+ \frac{\s}2} } \|v\|_{L_\ga^2(B(0,\frac{3\eta}{4}))} \le \frac{C}{\eta^{\sigma+ \frac{\s}2} }\Big( \|\nab w\|_{L_\ga^2(B(0,\frac{3\eta}{4}))} + \|u\|_{L_\ga^2(B(0,\frac{3\eta}{4}))}\Big).
\end{equation}
where we remind the reader that $\dot{C}^{0,\sigma}$ is the homogeneous H\"older space.
On the other hand, \cite[Theorem 2.2.3]{FKS1982} ensures that 
\begin{equation}
 \|u\|_{L^\infty(B(0,\frac34\eta))}\le C\eta \| |z|^{-\ga}\nab \phi\|_{L^\infty (B(0,\eta))}.
 \end{equation}
Using this bound to majorize $\|u\|_{L_\ga^2(B(0,3\eta/4))}$ in the right-hand side of \eqref{eqv'}, we obtain
\begin{equation}
\label{eqv''}
\|v\|_{\dot{C}^{0,\sigma}(B(0, \frac\eta2 ) ) } \le \frac{C}{\eta^{\sigma+ \frac{\s}2+1} } \Big( \|\nab w\|_{L_\ga^2(B(0,\frac{3\eta}{4}))}+  \| |z|^{-\gamma} \nab \phi\|_{L^\infty (B(0,\eta))}\eta^{1+\frac{\d+\k+\gamma}{2}}   \Big).
\end{equation}
Letting $[\partial_{j} w]$ denote the average of $\partial_{j} w$ in the ball  $B(0,\frac12\eta)$, we  deduce from \eqref{eqv} and the definition \eqref{eq:defgamma} of $\ga$ that 
\begin{align}\label{wwav}
\|\partial_{j} w-[\partial_{j} w]\|_{L^\infty(B(0, \frac12\eta))} &\le \|v-[v]\|_{L^\infty(B(0,\frac12\eta))} + 2\|u\|_{L^\infty(B(0,\frac12\eta))}\nn\\
&\le \frac{C}{\eta^{\frac{\s}2+1}} \|\nab w\|_{L_\ga^2(B(0,\frac34\eta))}+C\eta \| |z|^{-\gamma} \nab \phi\|_{L^\infty (B(0,\eta))}.
\end{align}
Using Cauchy-Schwarz's inequality, we moreover have 
\begin{align}
|[\partial_{j} w]|\le \frac{C}{\eta^{\d+\k} } \|\nab w\|_{L_\ga^2(B(0,\frac34\eta))}\Big( \int_{B(0, \frac34\eta)} \frac{1}{\zg}\Big)^{\frac12}\le \frac{C}{\eta^{\frac{\s}2+1}} \|\nab w\|_{L_\ga^2(B(0,\frac34\eta))}.
 \end{align} 
From \eqref{wwav} and the triangle inequality, it then follows that
\begin{equation}\label{wwav2}
\|\partial_{j} w\|_{L^\infty(B(0, \frac12\eta))} \le \frac{C}{\eta^{\frac{\s}2+1}}\|\nab w\|_{L_\ga^2(B(0,\frac34\eta))} +C\eta \| |z|^{-\gamma} \nab \phi\|_{L^\infty (B(0,\eta))}.
\end{equation}
Differentiating the equation \eqref{eq:Lwdivphi} $m$ times, with $m\ge 2$, we obtain in the same way that
\begin{equation}\label{Dmconc}
\|\nab^{\otimes m}w\|_{L^\infty(B(0, \frac12\eta))} \le \frac{C}{\eta^{\frac{\s}2+1}}\|\nab_x^{\otimes m}w\|_{L_\ga^2(B(0,\frac34\eta))}+C \eta \| |z|^{-\gamma} \nab_x^{\otimes m} \phi\|_{L^\infty(B(0,\eta))}.
\end{equation}
We next prove a Caccioppoli inequality to control $\|\nab_x^{\otimes m}w\|_{L_\ga^2(B(0,\frac34\eta))}$ in terms of $\|\nab_x^{\otimes (m-1)}w\|_{L_\ga^2(B(0,\frac78\eta))}$.

To this end, let $\chi$ be a smooth cutoff function equal to $1$ in $B(0, \frac34 \eta)$ and vanishing outside $B(0,\frac78\eta)$, with $|\nab \chi|\le \frac{C}{\eta}$.
Let  $f=\p_{i_1}\cdots\p_{i_{m-1}}w$ be some $(m-1)$-th order partial derivative of $w$ for $i_1,\ldots,i_{m-1}\in[\d]$, and let $\varphi= \p_{i_1}\cdots\p_{i_{m-1}} \phi$. Testing the equation 
\begin{align}
Lf = \div \varphi
\end{align}
against $\chi^2 f$, we obtain, after integrating by parts, 
\begin{align}
\int_{\R^{\d+\k}} \zg\chi^2| \nab f|^2 +\int_{\R^{\d+\k}}\zg f\nab(\chi^2) \cdot \nab f =  \int_{\R^{\d+\k}}\chi^2 fLf &=\int_{\R^{\d+\k}}\chi^2f\div\varphi \nn\\
&=-\int_{\R^{\d+\k}}\Big(\nab(\chi^2) f\cdot \varphi +\chi^2 \nab f \cdot \varphi\Big).
\end{align}
Moving the second term on the left-hand side over to the right-hand side, then applying Cauchy-Schwarz to the new right-hand side, we obtain
\begin{multline}
\int_{\R^{\d+\k}} \zg\chi^2| \nab f|^2  \le 2\|\chi\nab f\|_{L_\ga^2} \|f\nab \chi\|_{L_\ga^2} \\
  + \Big(\frac{C}{\eta}\|\chi f\|_{L_\ga^2} + \|\chi\nab f\|_{L_\ga^2} \Big)  \||z|^{-\ga} \varphi\|_{L^\infty(B(0,\frac78\eta))}\Big(\int_{B(0,\eta)} \zg\Big)^{\frac12},
\end{multline}
where we have implicitly used $\|\nab\chi\|_{L^\infty}\le C\eta^{-1}$. Hence, evaluating the last integral on the right-hand side and using $ab \le \frac12(\epsilon a + \epsilon^{-1}b)$ to absorb the factors of $\|\chi \nab f\|_{L_\ga^2}$ on the right-hand side into the left,
\begin{align}
\int_{\R^{\d+\k}} \zg\chi^2| \nab f|^2  & \le \frac{C}{\eta^2} \int_{B(0, \frac78\eta)} 
\zg f^2   +C \eta^{\s+2} \||z|^{-\ga} \varphi\|_{L^\infty (B(0,\frac78\eta))}^2.
\end{align}
Summing both sides of this inequality over all choices $i_1,\ldots,i_{m-1}$ in the definition of $f$ and recalling that $\chi\equiv 1$ on $B(0,3/4)$ then yields
\begin{equation} 
\int_{B(0, \frac34\eta)} \zg|\nab_x^{\otimes m} w|^2 \le  \frac{C}{\eta^2} \int_{B(0, \frac78\eta)} 
\zg |\nab_x^{\otimes (m-1)} w|^2  +C \eta^{\s+2} \||z|^{-\ga} \nab_x^{\otimes (m-1)}\phi\|_{L^\infty(B(0,\frac78\eta))}^2.
\end{equation}
Applying this bound into the right-hand side of \eqref{Dmconc}, we obtain
\begin{multline}
\|\nab_x^{\otimes m} w\|_{L^\infty(B(0, \frac12\eta))} \le \frac{C}{\eta^{\frac{\s}2+2}} 
\|\nab_x^{\otimes (m-1)} w\|_{L_\ga^2(B(0,\frac78\eta))} \\
 +C  \||z|^{-\ga} \nab_x^{\otimes (m-1)}\phi\|_{L^\infty(B(0,\frac78\eta))}
+C \eta \| |z|^{-\gamma} \nab_x^{\otimes m} \phi\|_{L^\infty (B(0,\frac78\eta))}.
\end{multline}
We may then iterate the argument to obtain \eqref{resultl}, {replacing $B(0,\frac78 \eta)$ by $B(0,(\frac12 + \frac{3}{8(m-1)})\eta)$.}
\end{proof}

In the next subsection, we will use \cref{lem:regularity} in the following form. 
\begin{cor}\label{reglemma}
Assume $w$ satisfies
\begin{equation}
Lw=   \div(\zg\varphi_1) + \zg \varphi_2\quad \text{in} \ B(0,\eta) \subset \R^{\d+\k} .
\end{equation}   Then, we have 
\begin{multline}\label{resultl'}
\|\nab_x^{\otimes m} w\|_{L^\infty(B(0, \frac12\eta))} \le \frac{C}{\eta^{\frac{\s}2+m}}\|\nab w\|_{L_\ga^2(B(0,\eta))} + C\sum_{k=1}^m \eta^{k-m+1}\|\nab_x^{\otimes k}\varphi_1\|_{L^\infty(B(0,\eta))}\\
+ C  \sum_{k=0}^m \eta^{k-m+2} \|\nab_x^{\otimes k}\varphi_2\|_{L^\infty(B(0, \eta))}.
\end{multline}
\end{cor}
\begin{proof}
Define a  vector field in $\R^{\d+\k}$ by
\begin{align}
\phi_2(x^1,\ldots,x^{\d+\k}) \coloneqq \Big(\zg \int_0^{x^1} \varphi_2(t, x^2, \dots, x^{\d+\k}) dt, 0, \dots, 0\Big).
\end{align}
One easily checks that $\div \phi_2= \zg \varphi_2$ and 
\begin{equation}\label{eq:zganabxphi2}
\| |z|^{-\ga} \nab_x^{\otimes k}  \phi_2\|_{L^\infty(B(0,\eta))} \le \eta\| \nab_x^{\otimes k} \varphi_2\|_{L^\infty(B(0,\eta))}+  \|\nab_x^{\otimes (k-1)}\varphi_2\|_{L^\infty(B(0,\eta))}.
\end{equation}

Now define the vector field $\phi\coloneqq \zg\varphi_1 + \phi_2$ in $\R^{\d+\k}$, which, by the triangle inequality and \eqref{eq:zganabxphi2}, satisfies
\begin{align}
\||z|^{-\ga}\nab_x^{\otimes k}\phi\|_{L^\infty(B(0,\eta))} \le \|\nab_x^{\otimes k}\varphi_1\|_{L^\infty(B(0,\eta))} + \eta\| \nab_x^{\otimes k} \varphi_2\|_{L^\infty(B(0,\eta))}+  \|\nab_x^{\otimes (k-1)}\varphi_2\|_{L^\infty(B(0,\eta))}.
\end{align}
Applying the result of \cref{lem:regularity}, we obtain
\begin{align}
\|\nab_x^{\otimes m} w\|_{L^\infty(B(0, \frac12\eta))} &\le \frac{C}{\eta^{\frac{\s}2+m}}\|\nab w\|_{L_\ga^2(B(0,\eta))}  +C  \sum_{k=1}^m \eta^{k-m+1} \||z|^{-\ga} \nab^{\otimes k}_x\phi\|_{L^\infty(B(0,\eta))} \nn\\
&\le \frac{C}{\eta^{\frac{\s}2+m}}\|\nab w\|_{L_\ga^2(B(0,\eta))}  + C\sum_{k=1}^m \eta^{k-m+1}\Big(\|\nab_x^{\otimes k}\varphi_1\|_{L^\infty(B(0,\eta))} \nn\\
&\ph\qquad+ \eta\| \nab_x^{\otimes k} \varphi_2\|_{L^\infty(B(0,\eta))}+  \|\nab_x^{\otimes (k-1)}\varphi_2\|_{L^\infty(B(0,\eta))}\Big).
\end{align}
Observing that
\begin{multline}
\sum_{k=1}^m \eta^{k-m+1}\Big(\eta\| \nab_x^{\otimes k} \varphi_2\|_{L^\infty(B(0,\eta))}+  \|\nab_x^{\otimes (k-1)}\varphi_2\|_{L^\infty(B(0,\eta))}\Big)\\
 = \sum_{k=1}^m \eta^{k-m+2}\|\nab_x^{\otimes k} \varphi_2\|_{L^\infty (B(0,\eta))}+ \sum_{k=0}^{m-1}\|\nab_x^{\otimes k} \varphi_2\|_{L^\infty(B(0,\eta))},
\end{multline}
we arrive at the desired \eqref{resultl'}.
\end{proof}

\begin{remark}\label{lemstandard}
By standard considerations, the estimate \eqref{resultl'} holds in any ball $B(0, \rho\eta)$ with $\rho<1$, up to changing the constant in the right-hand side.
\end{remark}

\subsection{Regularity for commutators by induction}\label{ssec:L2regind}
The following theorem shows that the commutators enjoy  {\it local} regularity properties similar to those of $h^f$. This is a vast generalization of the prototypical result for $h^f$ in \cite[Lemma A.2]{LS2018,Serfaty2023}. We are showing it under the strong assumption that $f$ vanishes in $B(x, \eta)$, which will be sufficient for our purposes. However, an analogous result would be true without this assumption if $f$ is regular enough, up to additional terms including derivatives of $f$ in the estimate.

It will be crucial for us later that the control can be obtained in terms of integrals restricted to arbitrarily small balls.
\begin{thm}\label{th44}
Let $v:\R^\d\rightarrow\R^\d$ be a vector field and $\tv$ be an extension to $\R^{\d+\k}$ as above. Let $f \in \Sc_0(\R^\d)$ and $\kappa^{(n), f}$ be as in \eqref{eq:kandef}. If, for any $x\in\R^\d$ and $\eta>0$,  $f$ vanishes in $B(x, \eta) \subset \R^\d$, then for any integer $m\ge 1$,  
\begin{multline}\label{regkappa}
\| \nab_x^{\otimes m} \kappa^{(n), f}\|_{L^\infty(B(x, \frac12 \eta))} \\ \le \frac{C}{\eta^{\frac{\s}2+m}} \sum_{p=0}^n    \sum_{\substack{c_1,\dots,  c_{n-p}\ge 0\\
\sum c_i \le m-1+2n} } \eta^{\sum c_i  } \|\nab^{\otimes(1+c_1)}\tv\|_{L^\infty(B(x,\eta))} \dots \|\nab ^{\otimes(1+c_{n-p})}\tv\|_{L^\infty(B(x,\eta))} \|\nab\ka^{(p),f}\|_{L_\ga^2(B(x,\eta))},
 \end{multline}
where $C>0$ depends only on  $m,n,\d,\s$ and the summation over the $c_i$ is vacuous if $p=n$.
\end{thm}

\begin{remark}\label{rem:th44}
For the particular extension $\tv$ introduced in \cref{rem:vextell}, if $\eta \le \ell$, then $\chi \equiv 0$ in $B(0,\eta)$. Hence, $\tv$ coincides with the trivial extension in $B(x,\eta)$. Consequently, all factors $\|\nab^{\otimes(1+c_i)}\tv\|_{L^\infty(B(x,\eta))}$ may be replaced by $\|\nab^{\otimes(1+c_i)}v\|_{L^\infty(B(x,\eta))}$.
\end{remark}

\begin{proof}[Proof of \cref{th44}]
Fix an integer $n_0 \ge 0$, and set $\rho\coloneqq \frac{1}{8(n_0+1)}$. For $0\le n\le n_0$, we inductively prove the following statement: for any $m\ge 1$, it holds that
\begin{multline}\label{indkappa} 
\|\nab_x^{\otimes m} \kappa^{(n), f}\|_{L^\infty(B(x, (1- (n+1)\rho ) \eta))}  \\
\le \frac{C_{n,m}}{\eta^{\frac{\s}2+m}} \sum_{p=0}^n    \sum_{\substack{c_1,\dots,  c_{n-p}\ge 0\\\sum c_i \le m-1+ 2n} } \eta^{\sum c_i  } \|\nab^{\otimes(1+c_1)}\tv \|_{L^\infty(B(x,\eta))} \dots \|\nab ^{\otimes(1+c_{n-p})} \tv\|_{L^\infty(B(x,\eta))} \|\nab\ka^{(p),f}\|_{L_\ga^2(B(x,\eta))},
\end{multline}
\begin{multline}\label{indnu}
\max_{i\in[\d]}\|\nab_x^{\otimes m}\nu_i^{(n), f}\|_{L^\infty(B(x, (1-(n+1)\rho )\eta))} \\
\le \frac{C_{n,m}}{\eta^{\frac{\s}2+m+1}} \sum_{p=0}^n    \sum_{\substack{c_1,\dots,  c_{n-p}\ge 0\\
\sum c_i \le m+ 2n} } \eta^{\sum c_i } \|\nab^{\otimes(1+c_1)}\tv \|_{L^\infty(B(x,\eta))}\dots \|\nab ^{\otimes(1+c_{n-p})} v\|_{L^\infty(B(x,\eta))}
\|\nab\ka^{(p),f}\|_{L_\ga^2(B(x,\eta))},
\end{multline}
\begin{multline}\label{indmu}
\max_{i,j\in[\d]}\| \nab_x^{\otimes m} \mu_{ij}^{(n), f}\|_{L^\infty(B(x, (1-(n+1)\rho)  \eta))}, \\
\le \frac{C_{n,m}}{\eta^{\frac{\s}2+m+2}} \sum_{p=0}^n    \sum_{\substack{c_1,\dots,  c_{n-p}\ge 0\\
\sum c_i \le m+1 +2n} } \eta^{\sum c_i  } \|\nab^{\otimes(1+c_1)}\tv \|_{L^\infty(B(x,\eta))} \dots \|\nab ^{\otimes(1+c_{n-p})}\tv\|_{L^\infty(B(x,\eta))} 
\|\nab\ka^{(p),f}\|_{L_\ga^2(B(x,\eta))}.
\end{multline}

Consider the base case $n=0$. For $p=0$, $\kappa^{(p),f}= h^f $ which solves $L h^f = \cd \tl{f} = 0$ in $B(x, \eta)\subset \R^{\d+\k}$ by assumption on $f$. Applying \eqref{resultl'} ($\varphi=0$) with \cref{lemstandard}, we find that 
\begin{equation}
\|\nab_x^{\otimes m} h^f\|_{L^\infty(B(x, (1-\rho)\eta))} \le  \frac{C_{0,m}}{\eta^{\frac{\s}2+m}} \|\nab h^f\|_{L_\ga^2(B(x,\eta))},
\end{equation}
 with $C_{0,m}$ also depending on $\rho$. Thus, the result \eqref{indkappa} is true for $n=0$. Since $\nu^{(0),f}= \nab h^{ f}$, $\mu^{(0), f}= \nab^{\otimes 2} h^f$, the inequalities  \eqref{indnu}, \eqref{indmu} for $n=0$ follow from \eqref{indkappa} for $n=0$. Suppose  that for some $n \ge 1$, the relations \eqref{indkappa}, \eqref{indnu}, \eqref{indmu}  are true up to order $n-1$. We will now show they are true for order $n$.
 
Recalling the identity \eqref{eqLk}, we may write

\begin{align}\label{eq:Lkanfvphi}
L\ka^{(n),f} =\div(\zg\varphi_1) + \zg\varphi_2,
\end{align}
where
\begin{align}
& \varphi_1^i \coloneqq -n\p_i\tv\cdot\nu^{(n-1),f}, \\
& \varphi_2 \coloneqq   - n\p_i\tv \cdot \p_i\nu^{(n-1),f} + n(n-1)(\p_i \tv)^{\otimes 2}:\mu^{(n-2),f}.
\end{align}
 We note that since $\tv^{\d+\k}=0$ if $\k=1$ by assumption, the preceding right-hand only depends on the components $\nu_i^{(n-1),f}, \mu_{ij}^{(n-1),f}$ for $i,j\in[\d]$.

We apply the estimate \eqref{resultl'} of \cref{reglemma} to $\kappa^{(n), f}$  as a solution of \eqref{eq:Lkanfvphi} in $B(x, (1-n\rho)\eta)$ to obtain
\begin{multline}\label{eq:inducknab0}
\|\nab_x^{\otimes m} \kappa^{(n), f} \|_{L^\infty(B(x, (1-(n+1)\rho)\eta))} \\
\le \frac{C}{\eta^{\frac{\s}2+m}}\|\nab \ka^{(n),f}\|_{L_\ga^2(B(x,\eta))} + C\sum_{k=1}^m \eta^{k-m+1} \|\nab_x^{\otimes k}(\nab\tv\cdot\nu^{(n-1),f})\|_{L^\infty(B(x,(1-n\rho)\eta))}\\  
 + C  \sum_{k=0}^m \eta^{k-m+2}\Big(\|\nab_x^{\otimes k}(\p_i\tv \cdot \p_i\nu^{(n-1),f})\|_{L^\infty(B(x,(1-n\rho)\eta))}  + \|\nab_x^{\otimes k}((\p_i\tv)^{\otimes 2}  :\mu^{(n-2), f} )   \|_{L^\infty(B(x, (1-n\rho)\eta))} \Big).
\end{multline}
By the Leibniz rule and the induction hypothesis \eqref{indnu} for $\nu^{(n-1),f}$, we see that
\begin{align}
 &\sum_{k=1}^m \eta^{k-m+1}\|\nab_x^{\otimes k}(\nab\tv\cdot\nu^{(n-1),f})\|_{L^\infty(B(x,(1-n\rho)\eta))}\nn\\
 &\le  C\sum_{k=1}^m \eta^{k-m+1}\sum_{q=0}^k \|\nab_x^{\otimes q}\nab\tv\|_{L^\infty(B(x,(1-n\rho)\eta))} \max_{i\in [\d]} \|\nab_x^{\otimes (k-q)} \nu_i^{(n-1),f}\|_{L^\infty(B(x,(1-n\rho)\eta))}  \nn\\
 &\le C\sum_{k=1}^m \sum_{q=0}^k \|\nab_x^{\otimes q}\nab\tv\|_{L^\infty(B(x,(1-n\rho)\eta))} \frac{\eta^{k-m+1}}{\eta^{\frac{\s}{2}+k-q+1}}\sum_{p=0}^{n-1}\sum_{\substack{c_1,\ldots,c_{n-1-p}\ge 0 \\ \sum c_i\le k-q+2(n-1)}} \nn\\
 &\ph\qquad \eta^{\sum c_i} \|\nab^{\otimes(1+c_1)}\tv \|_{L^\infty(B(x,\eta))} \dots \|\nab ^{\otimes(1+c_{n-1-p})}\tv\|_{L^\infty(B(x,\eta))} \|\nab\ka^{(p),f}\|_{L_\ga^2(B(x,\eta))}.
\end{align}
Letting $c_{n-p}\coloneqq q$ and interchanging the order of summations, we obtain
\begin{multline}
\sum_{k=1}^m \eta^{k-m+1}\|\nab_x^{\otimes k}(\nab\tv\cdot\nu^{(n-1),f})\|_{L^\infty(B(x,(1-n\rho)\eta))} \\
\le \frac{C}{\eta^{\frac{\s}{2}+m}}\sum_{p=0}^{n-1} \sum_{\substack{c_1,\ldots,c_{n-p}\ge 0 \\ \sum c_i\le m+2n-1}}\eta^{\sum c_i}\|\nab^{\otimes(1+c_1)}\tv \|_{L^\infty(B(x,\eta))} \dots \|\nab ^{\otimes(1+c_{n-p})}\tv\|_{L^\infty(B(x,\eta))} \|\nab\ka^{(p),f}\|_{L_\ga^2(B(x,\eta))}. \label{eq:inducknab1}
\end{multline}
By similar reasoning,
\begin{align}
 &\sum_{k=0}^m \eta^{k-m+2}\|\nab_x^{\otimes k}(\p_i\tv \cdot \p_i\nu^{(n-1),f})\|_{L^\infty(B(x,(1-n\rho)\eta))} \nn\\
 &\le  C\sum_{k=0}^m \sum_{q=0}^k \|\nab_x^{\otimes q}\nab \tv\|_{L^\infty(B(x,(1-n\rho)\eta))} \frac{\eta^{k-m+2}}{\eta^{\frac{\s}{2}+k-q+2}}\sum_{p=0}^{n-1}\sum_{\substack{c_1,\dots,  c_{n-1-p}\ge 0\\ \sum c_i \le k-q+1+ 2(n-1)} } \nn\\
&\ph\qquad \eta^{\sum c_i  } \|\nab^{\otimes(1+c_1)}\tv\|_{L^\infty(B(x,\eta))} \dots \|\nab ^{\otimes(1+c_{n-1-p})}\tv\|_{L^\infty(B(x,\eta))}\|\nab\ka^{(p),f}\|_{L_\ga^2(B(x,\eta))} \nn\\
&\le \frac{C}{\eta^{\frac{\s}{2}+m}}\sum_{p=0}^{n-1}\sum_{\substack{c_1,\dots,  c_{n-p}\ge 0\\ \sum c_i \le m+2n-1}}\eta^{\sum c_i  } \|\nab^{\otimes(1+c_1)}\tv \|_{L^\infty(B(x,\eta))} \dots \|\nab ^{\otimes(1+c_{n-p})}\tv\|_{L^\infty(B(x,\eta))}\|\nab\ka^{(p),f}\|_{L_\ga^2(B(x,\eta))}, \label{eq:inducknab2}
\end{align}
where we again let $c_{n-p} \coloneqq q$ to obtain the final inequality. Finally, proceeding as before, except now using the induction hypothesis \eqref{indmu} for $\mu^{(n-2),f}$, we obtain 
\begin{align}
&\sum_{k=0}^m \eta^{k-m+2} \|\nab_x^{\otimes k}((\p_i \tv)^{\otimes 2}  :\mu^{(n-2), f} )   \|_{L^\infty(B(x, (1-n\rho)\eta))} \nn\\
&\le C\sum_{k=0}^m \sum_{0\le r\le q\le k}\|\nab_x^{\otimes r}\p_i \tv\|_{L^\infty(B(x, (1-n\rho)\eta))} \|\nab_x^{\otimes(q-r)}\p_i\tv\|_{L^\infty(B(x, (1-n\rho)\eta))} \frac{\eta^{k-m+2}}{\eta^{\frac{\s}{2}+k-q+2}}\sum_{p=0}^{n-2} \sum_{\substack{c_1,\dots,  c_{n-2-p}\ge 0  \\  
\sum c_i \le k-q+1+ 2(n-2)} } \nn\\
&\ph\qquad \eta^{\sum c_i  } \|\nab^{\otimes(1+c_1)}\tv \|_{L^\infty(B(x,\eta))} \dots \|\nab ^{\otimes(1+c_{n-2-p})}\tv\|_{L^\infty(B(x, \eta))}\|\nab\ka^{(p),f}\|_{L_\ga^2(B(x,\eta))} \nn\\
&\le \frac{C}{\eta^{\frac{\s}{2}+m}}\sum_{p=0}^{n-2}\sum_{\substack{c_1,\dots,  c_{n-p}\ge 0  \\ \sum c_i \le m+2n-1} } \eta^{\sum c_i  }\|\nab^{\otimes(1+c_1)}\tv \|_{L^\infty(B(x,\eta))} \dots \|\nab ^{\otimes(1+c_{n-p})} \tv\|_{L^\infty(B(x,\eta))}\|\nab\ka^{(p),f}\|_{L_\ga^2(B(x,\eta))}, \label{eq:inducknab3}
\end{align}
where we let $c_{n-p-1} \coloneqq r$, $c_{n-p}\coloneqq q-r$ to obtain the final inequality. Applying the estimates \eqref{eq:inducknab1}, \eqref{eq:inducknab2}, \eqref{eq:inducknab3} to the right-hand side of our starting point \eqref{eq:inducknab0} then shows that \eqref{induck} holds at order $n$. The relations \eqref{indnu} and \eqref{indmu} then follow from \eqref{recnu} and  \eqref{recmu}, respectively, thereby completing the proof of the induction step.
\end{proof}

\section{The higher-order functional inequalities}\label{sec:FI}
We now prove the sharp functional inequalities of  \cref{thm:FI} for the $n$-th order variation of the modulated energy by linear transport for $n\ge 2$. As described in the introduction, this proof is based on a combination of the inequalities of \cref{sec:comm} with a renormalization procedure using the potential truncation/charge smearing of \cref{sec:ME}. One decomposes the main quantity we seek to estimate into a first term corresponding to the same quantity with the singular charges replaced by smeared charges, and second and third terms corresponding to the errors generated by this smearing/renormalization. The first term is handled by  \cref{thm:comm2}. The remaining terms require the full power of the regularity theory of \cref{th44} if one wishes to preserve the estimate's sharpness and localization to the support of $v$.  As shown later in \cref{ssec:FImac}, cruder estimates, which would be sufficient when the support of $v$ is macroscopic (i.e. $\ell \asymp 1$), can be obtained by bypassing this regularity theory. {We also show analogous estimates without localization in \cref{ssec:FIunloc}.}
 

Of course, we cannot apply \cref{thm:comm2} directly with $f=\frac{1}{N}\sum_{i=1}^N\delta_{x_i}-\mu$, since then $h^{f}\notin L_\ga^2$. Therefore, much of the hard work lies in renormalizing the commutator expressions by smearing the Dirac masses $\delta_{x_i}$, so that $f$ is now sufficiently regular, and then optimally estimating the error introduced by the smearing. A significant difficulty is that in the non-Coulomb Riesz case, we cannot use the smearing $\delta_{x_i}^{(\eta_i)}$ from \eqref{defdeta} because it is not supported in $\R^\d\times\{0\}^\k$ and would preclude applying  \cref{thm:comm2}.

As in the statement of \cref{thm:FI}, we assume throughout this section that $\Om$ contains a closed $5\ell$-neighborhood of $\Omega'$, where $\Omega'$ is a ball of radius $\ell$ containing  a $2\lambda$-neighborhood of $\supp v$.  In particular, this implies that if $x_i\in\Omega'$, then $\dist(x_i,\p\Omega) \ge 4\ell$. Moreover, $\min(\la,\la/\ell)<\frac12$.

For $1\le i\le N$, we now let 
\begin{equation}\label{choiceeta}
\eta_i \coloneqq \begin{cases}\frac12\cre{ \rs_i}, & \text{if} \  x_i\in \Omega'\\
0, & \text{otherwise,}\end{cases}
\end{equation}
where ${\rs_i}$ is as in \eqref{eq:defri}. Importantly, the Dirac masses corresponding to points not in $\Omega'$ are not smeared. Thus, there will be no contribution to the renormalization error for points outside $\Omega'$, which is crucial to obtaining localized estimates.

To compactify the notation, we abbreviate
\begin{equation}\label{defXi}
\Xi \coloneqq { C\cre \int_{(\supp \nab \tl v)\cup(\Omega \times [-\la,\la]^\k)} \zg |\nab h_{N,\vec\rs}|^2
} + C \frac{\#I_\Omega\|\mu\|_{L^\infty(\hat \Omega)}\lambda^{\d-\s}}{N},
\end{equation}
which is the factor we wish to obtain on the right-hand side of \eqref{mainn}.  

\subsection{Preliminary estimates}\label{ssec:FIpre}
{Let $\rho^{(\eta)}$ be a mollifier in $\R^\d$ at lengthscale $\eta$ supported in $B(0,\frac12\eta)$. We will assume throughout this section that the moments of mollifier $\rho^{(\eta)}$ vanish up to order $m$. The exact value of $m$ needed will be specified later. Such a mollifier may be constructed as follows.

Fix $m\ge 1$. Given a smooth bump function $\psi$ with $\int\psi = 1$, define a new function $\rho$ in Fourier space by
\begin{align}
\hat{\rho}(\xi) \coloneqq \hat{\psi}(\xi)  - \hat{\zeta}(\xi)\sum_{k=1}^m \frac{1}{k!} \xi^{\otimes k}\cdot\nab^{\otimes k}\hat{\psi}(0),
\end{align}
where $\zeta$ is $C^\infty$ with compact support and such that $\hat{\zeta}(0) = 1$. By Fourier inversion, it follows that $\rho$ is $C^\infty$ with support in $B(0,r)$, for some $r>0$, has unit mean, and $\int x^{\otimes k}\rho = 0$ for $1\le k\le m$. Now take $\rho^{(\eta)} \coloneqq (\frac{2r}{\eta})^{\d}\rho(\frac{2r}{\eta}\cdot)$.}

We set
\begin{align}\label{eq:trhodef}
\rho_{x_0}^{(\eta)} \coloneqq \rho^{(\eta)}(\cdot-x_0), \qquad \tilde\rho_{x_0}^{(\eta)} \coloneqq \widetilde{\rho_{x_0}^{(\eta)}}.
\end{align}
 The essential point is that $\tilde \rho^{(\eta)}_{x_0}$, unlike $\delta_{x_0}^{(\eta)}$, is supported in $\R^\d \times \{0\}^\k$. 
Consequently, we have, by \eqref{relhf}, that
\begin{equation}
\label{eqcrucre}
L(\g\ast\rho^{(\eta)}_{x_0}) = \cd \tl\rho^{(\eta)}_{x_0}.
\end{equation}

The following lemma allows us to compare the electric potentials generated by the two different regularizations and will be needed only in the non-Coulomb case.
 
\begin{lemma}\label{lem:barHN}
For any pairwise distinct configuration $\ux_N \in (\R^\d)^N$, let 
\begin{align}
\bar h_{N,\vec{\eta}} \coloneqq \G* \Big( \frac{1}{N}\sum_{i=1}^N \rho_{x_i}^{(\eta_i)} -\mu\Big),
\end{align}
where $\vec{\eta}$ is as in \eqref{choiceeta}. We have
\begin{equation} \label{HH}
 \int_{{( \supp \nab \tl v)\cup (\Omega\times  [-\la,\la]^\k)}}|z|^\gamma |\nabla \bar h_{N,\vec{\eta}}|^2 \le C \Xi,
\end{equation}
 where $C>0$ depends only on $\d,\s$.

\end{lemma}
\begin{proof}
Using integration by parts and the identity \eqref{eqcrucre}, we have, by definition of $\eta_i$, 
\begin{align}\label{et1}
\int_{\R^{\d+\k}}|z|^\gamma |\nabla (\bar h_{N,\vec{\eta}}- h_{N,\vec{\eta}})|^2 
&=  \int_{\R^{\d+\k}} ( \bar h_{N,\vec{\eta}}- h_{N,\vec{\eta}}) L(\bar h_{N,\vec{\eta}}- h_{N,\vec{\eta}}) \nn\\
&= \frac{\cd}N\sum_{i\in I_{\Omega'}} \int_{\R^{\d+\k}} ( \bar h_{N,\vec{\eta}}- h_{N,\vec{\eta}})d(\tilde\rho_{x_i}^{(\eta_i)} - \delta_{x_i}^{(\eta_i)}) \nn\\
&= \frac{\cd}{N^2}  \sum_{i,j\in I_{\Omega'}}\int_{\R^{\d+\k} }(\G*\rho_{x_j}^{(\eta_j)} -  \G_{\eta_j}(\cdot-x_j)) d(\tilde\rho_{x_i}^{(\eta_i)} - \delta_{x_i}^{(\eta_i)}).
\end{align}
For the terms corresponding to $i=j$,  we easily check (by scaling) that 
\begin{equation}
|\G_{\eta_j}(\cdot-x_j) - \G*\rho_{x_j}^{(\eta_j)}|\le C \eta_i^{-\s}.
\end{equation}
Thus, we can bound
\begin{equation}\label{et2}
\frac1{N^2}  \sum_{i\in I_{\Omega'}}\Big|\int_{\R^{\d+\k} }( \G_{\eta_j}(\cdot-x_i) - \G* \rho_{x_i}^{(\eta_i)}) d(\tilde\rho_{x_i}^{(\eta_i)} - \delta_{x_i}^{(\eta_i)})\Big|\le 
\frac{C}{N^2} \sum_{i \in I_{\Omega'} } \eta_i^{-\s}\le C \Xi,
\end{equation}
where we used \eqref{eq:13} for the last inequality. For $j\neq i$ with $|x_i-x_j|\le \lambda$, we instead use that {$\eta_j\le \frac14\min_{i\ne j}|x_i-x_j|$}, hence
\begin{equation}
|\G_{\eta_j}(\cdot-x_j) - \G* \rho_{x_j}^{(\eta_j)}|\le \frac{C}{|x_i-x_j|^\s} \quad\text{in} \ B(x_i, \eta_i).
\end{equation}
Thus, 
\begin{multline}\label{et3}
\frac1{N^2}  \sum_{i,j \in I_{\Omega'} : i\ne j, |x_i-x_j|\le \lambda} \Big|\int_{\R^{\d+\k} }( \G_{\eta_j}(\cdot-x_i) - \G*\rho_{x_i}^{(\eta_i)}) d(\tilde\rho_{x_i}^{(\eta_i)} - \delta_{x_i}^{(\eta_i)})\Big|\\
\le \frac{C}{N^2}\sum_{i,j \in I_{\Omega'} : i\ne j, |x_i-x_j|\le \lambda} \frac{1}{|x_i-x_j|^\s}\le C \Xi,
\end{multline} where this time, we used \cref{cor:coro3} to conclude. 

{{
We now turn to the remaining terms, for which $|x_i-x_j|> \lambda$. Using that $\g_{\eta_j}(\cdot-x_j) = \g(\cdot-x_j)$ in $B(x_i,\eta_i)$, approximating $\g(\cdot-x_j)$ by its degree one Taylor polynomial centered at $x_i$, and using that $\rho_0^{(\eta)}, \delta_0^{(\eta_i)}$ have vanishing first moment by assumption (if $m\ge 1$) and rotational symmetry, respectively, we find that in $B(x_i,\eta_i)$,}
\begin{align}
&\frac1{N^2}  \sum_{i,j \in I_{\Omega'}: i\ne j, |x_i-x_j|\ge \lambda} \Big|\int_{\R^{\d+\k} }(\G*\rho_{x_j}^{(\eta_j)} -  \G_{\eta_j}(\cdot-x_j)) d(\tilde\rho_{x_i}^{(\eta_i)} - \delta_{x_i}^{(\eta_i)})\Big| \nn \\
&\le \frac{C}{N^2}  \sum_{ i,j \in I_{\Omega'}: i\ne j,  |x_i-x_j|\ge \lambda} \|\nab^{\otimes 2}\g(\cdot-x_j)\|_{L^\infty(B(x_i,\eta_i+\eta_j))} \eta_i^2 \nn\\
&\le \frac{C}{N^2} \sum_{i,j \in I_{\Omega'}: i\ne j,  |x_i-x_j|\ge \lambda}\frac{\eta_i^2}{|x_i-x_j|^{\s+2}}\le C\Xi. \label{et4}
\end{align}
In obtaining the penultimate inequality, we have used \cref{prop:multiscale}  (noting that we only need to consider $\s \neq \d-2$ for this lemma)  and $\eta_i\le \lambda= (N\|\mu\|_{L^\infty(\Omega)})^{-1/\d}$, $\ell \ge 2\lambda$ to obtain that
\begin{multline}\label{eq5}
\frac1{N^2}\sum_{i,j \in I_{\Omega'}: i\ne j,  |x_i-x_j|\ge \lambda}\frac{\eta_i^2}{|x_i-x_j|^{\s+2}}\le C
 {\int_{\Omega \times [-\ell,\ell]^\k} \zg|\nab h_{N,\vec\rs}|^2 } \\ + C\frac{\lambda^{\d-\s}\#I_\Omega \|\mu\|_{L^\infty(\hat{\Omega})}}{N} + C\frac{\lambda^2  \ell^{\d-\s-2} \#I_\Omega \|\mu\|_{L^\infty(\hat \Omega)}}{N} \le C \Xi.
\end{multline}
Here, we have implicitly used that $\d-\s-2\le 0$ and $\la \le \ell$, hence $\la^2\ell^{\d-\s-2} \le \la^{\d-\s}$.} Applying the bounds \eqref{et2}, \eqref{et3}, and \eqref{et4} together with the triangle inequality to \eqref{et1}, we conclude that
\begin{align}
\int_{\R^{\d+\k}} \zg |\nab (\bar h_{N,\vec{\eta}}-h_{N,\vec{\eta}})|^2\le C \Xi.
\end{align}
The desired conclusion \eqref{HH} follows readily from the estimate \eqref{eq:pr2} of \cref{prop:MElb} and the $L_\ga^2$ triangle inequality.
\end{proof}

\subsection{Main proof}\label{ssec:FImain}
With \cref{lem:barHN} in hand, let us now turn to the main proof of \eqref{mainn}. We decompose
\begin{multline}
\int_{(\R^\d)^2\setminus\triangle}\nabla^{\otimes n}\g(x-y) : (v(x)-v(y))^{\otimes n} d\Big(\frac{1}{N}\sum_{i=1}^N\delta_{x_i}-\mu\Big)^{\otimes 2}(x,y)\\
 \eqqcolon \Te_1+\Te_2+\Te_3 \label{eq:rcommhorhs}
\end{multline}
where 
\begin{equation}
\Te_1 \coloneqq \int_{(\R^\d)^2} \nabla^{\otimes n}\g(x-y) : (v(x)-v(y))^{\otimes n} 
 d \Big(\frac{1}{N}\sum_{i=1}^N\rho_{x_i}^{(\eta_i)}-\mu\Big)^{\otimes 2}(x,y),
 \end{equation}
\begin{equation}
\Te_2\coloneqq \frac1N\sum_{i\in I_{\Omega'}} \int_{(\R^\d)^2}\nabla^{\otimes n}\g(x-y) : (v(x)-v(y))^{\otimes n}  d\Big(\frac{1}{N}\sum_{j=1}^N\rho_{x_j}^{(\eta_j)}-\mu\Big)(x) d( \delta_{x_i}-\rho_{x_i}^{(\eta_i)}) (y),
\end{equation}
and 
\begin{equation}
\Te_3\coloneqq \frac1N\sum_{i\in I_{\Omega'}}  \int_{(\R^\d)^2\setminus \triangle} \nabla^{\otimes n}\g(x-y) : (v(x)-v(y))^{\otimes n} d \Big(\frac{1}{N}\sum_{j=1}^N\delta_{x_j}-\mu\Big) (x)d( \delta_{x_i}-\rho_{x_i}^{(\eta_i)}) (y).
\end{equation}

Throughout the proof, we make the specific choice of extended vector field $\tv$ from \cref{rem:vextell}.

\medskip

{\bf Step 1: the first term}.
Abbreviating
\begin{align}\label{eq:fabrev}
f \coloneqq \frac1N\sum_{j=1}^N \rho_{x_j}^{(\eta_j)}-\mu,
\end{align}
{
we may apply the estimate \eqref{eq:comm2dual} of \cref{thm:comm2}  to obtain
\begin{align}
|\Te_1| \le C(\ell\|\nab^{\otimes 2}\tv\|_{L^\infty})^n \|\nab h^{ f}\|_{L_\ga^2(\supp\nab\tv)}^2 
\end{align}
where the constant $C>0$ depends only on $\d,\s,n$. Noting that $h^f = \bar h_{N,\vec\eta}$, and combining with \eqref{HH} and \eqref{eq:13}, we deduce that
\begin{equation}\label{eq:HOFIT1fin}
|\Te_1|\le  C{(\ell\|\nab^{\otimes 2}\tv\|_{L^\infty})^n}\Xi,
\end{equation}
with $\Xi$ as in \eqref{defXi}.}



{\bf Step 2: the second term}.
Continuing to use the notation \eqref{eq:fabrev}, we rewrite $\Te_2$ as 
\begin{equation}\label{redeft2}
\Te_2= \frac1N\sum_{i\in I_{\Omega'}} \int_{\R^{\d+\k}} \kappa^{(n),f} d(\delta_{x_i}-\tl\rho_{x_i}^{(\eta_i)}).
\end{equation} 

For each $i \in I_{\Omega'}$, let us split 
\begin{align}\label{eq:ffisplit}
f= \frac1N \rho_{x_i}^{(\eta_i)} - \mu\indic_{B(x_i, R_i)} + f_i,\qquad f_i \coloneqq \frac1N\sum_{j\neq i}\rho_{x_j}^{(\eta_j)}-\mu \indic_{B(x_i ,R_i)^c},
\end{align}
and
where the radius $R_i$ is chosen so that $\int_{B(x_i, R_i)} d\mu=\frac1N$.  This choice ensures that each difference $\frac1N\rho_{x_i}^{(\eta_i)} - \mu\indic_{B(x_i, R_i)}$ has zero average in $\R^\d$. We claim that $R_i \ge 2 \eta_i$. Indeed, by definition \eqref{choiceeta} of $\eta_i$ and definition of $\mathsf{r}_i$, we have
\begin{align}
\eta_i \le \frac12 \mathsf{r}_i \le \frac{\la}{8}  \Longrightarrow \int_{B(x_i, 2\eta_i) } d\mu \le \frac{\pi^{\d/2} \|\mu\|_{L^\infty(\Omega)} (2\eta_i)^\d}{\Gamma(\d/2+1)}  \le \frac{\pi^{\d/2}\|\mu\|_{L^\infty(\Omega)} \lambda^\d}{\Gamma(\d/2+1) 4^{\d} }   < \frac{1}{N}.
\end{align}

Using this decomposition, we bound
\begin{multline}\label{eq:T2kndcomp}
\Big|\int_{\R^{\d+\k}} \kappa^{(n), f}    d(\delta_{x_i}-\tl\rho_{x_i}^{(\eta_i)})\Big| \le \Big|\int_{\R^{\d+\k}} \kappa^{(n), f_i}    d(\delta_{x_i}-\tl\rho_{x_i}^{(\eta_i)})\Big| \\
+ \Big|\int_{\R^{\d+\k}} \kappa^{(n),(\frac{1}{N}\rho_{x_i}^{(\eta_i)}-\mu\indic_{B(x_i,R_i)})}    d(\delta_{x_i}-\tl\rho_{x_i}^{(\eta_i)})\Big|.
\end{multline}
We separately estimate each of the two terms on the right-hand side, beginning with the second one.

By the mean-value theorem,
\begin{align}\label{eq:Kvbnd}
|\nab^{\otimes n}\g(X-Y)\cdot (\tv(X)-\tv(Y))^{{\otimes n}}| \leq C\frac{\|\nab\tv\|_{L^\infty}^n}{|X-Y|^{\s}}, \qquad X\neq Y\in \R^{\d+\k}.
\end{align}
Therefore, we have the bound
\begin{align}
\Big|\ka^{(n),(\frac{1}{N}\rho_{x_i}^{(\eta_i)}-\mu\indic_{B(x_i,R_i)})}\Big|\le  C\int_{\R^\d}\frac{\|\nab\tv\|_{L^\infty}^n}{|\cdot-y|^\s}d\Big(\frac1N\rho_{x_i}^{(\eta_i)} + \mu\indic_{B(x_i,R_i)}\Big)(y).
\end{align}
If $\s=0$, the preceding right-hand side is evidently $\le C\|\nab\tv\|_{L^\infty}/N$, using that $\int_{B(x_i,R_i)}d\mu = \frac 1N$. If $\s\ne 0$, then dividing into regions $|X-y|\le \eta_i$ and $|X-y|>\eta_i$:
\begin{align}
\int_{|X-y|\le \eta_i}|X-y|^{-\s}d\Big(\frac1N\rho_{x_i}^{(\eta_i)} + \mu\indic_{B(x_i,R_i)}\Big)(y) &\le C\Big(\frac1N\|\rho_{x_i}^{(\eta_i)}\|_{L^\infty} + \|\mu\|_{L^\infty(\Omega)}\Big)\eta_i^{\d-\s} \nn\\
&\le C\Big(\frac{\eta^{-\d}}{N}+\|\mu\|_{L^\infty(\Omega)}\Big)\eta_i^{\d-\s}\nn\\
&\le \frac{C}{N\eta_i^{\s}},
\end{align} 
where the final inequality follows from $\|\mu\|_{L^\infty(\Omega)}\eta^{\d}\le 1/N$, and
\begin{align}
\int_{|X-y|> \eta_i}|X-y|^{-\s}d\Big(\frac1N\rho_{x_i}^{(\eta_i)} + \mu\indic_{B(x_i,R_i)}\Big)(y) \le \eta_i^{-\s}\int_{\R^\d}d\Big(\frac1N\rho_{x_i}^{(\eta_i)} + \mu\indic_{B(x_i,R_i)}\Big)(y) = \frac{2}{N\eta_i^{\s}}.
\end{align}
All together, we conclude that
\begin{equation} \label{propk2}
\|\ka^{(n), \frac1N \rho_{x_i}^{(\eta_i)} - \mu\indic_{B(x_i, R_i)} }\|_{L^\infty}\le \frac{C \|\nab\tv\|_{L^\infty}^n}{N \eta_i^{\s}}.
\end{equation}
Hence,
\begin{align}\label{propk2'}
\Big|\int_{\R^{\d+\k}} \kappa^{(n),(\frac{1}{N}\rho_{x_i}^{(\eta_i)}-\mu\indic_{B(x_i,R_i)})}    d(\delta_{x_i}-\tl\rho_{x_i}^{(\eta_i)})\Big| &\le \| \kappa^{(n),(\frac{1}{N}\rho_{x_i}^{(\eta_i)}-\mu\indic_{B(x_i,R_i)})}\|_{L^\infty}\int_{\R^{\d+\k}}d(\delta_{x_i}+\tl\rho_{x_i}^{(\eta_i)}) \nn\\
&\le \frac{C \|\nab\tv\|_{L^\infty}^n}{N \eta_i^{\s}}.
\end{align}

For the remaining term on the right-hand side of \eqref{eq:T2kndcomp}, we use the mean-value theorem to crudely estimate
\begin{align}\label{eq:T2knfiMVT}
\Big|\int_{\R^{\d+\k}} \kappa^{(n), f_i}    d(\delta_{x_i}-\tl\rho_{x_i}^{(\eta_i)})\Big| \le \|\nab_x\kappa^{(n), f_i} \|_{L^\infty}\eta_i,
\end{align}
where we have implicitly used that $\rho_{x_i}^{(\eta_i)}$ is supported in $B(x_i, \eta_i/2)$. We note that $f_i$ vanishes in $B(x_i, \eta_i)$, by disjointness of the balls, that $f_i$ has zero average in $\R^\d$, and that $\eta_i \le \lambda \le \ell$. Thus, \cref{th44} applies to $\ka^{(n),f_i}$, yielding    
\begin{multline}\label{kfi}
\|\nab_x\kappa^{(n), f_i} \|_{L^\infty (B(x_i, \eta_i/2))} \\ \le
\frac{C}{\eta_i^{\frac{\s}{2}+1}} \sum_{p=0}^n    \sum_{\substack{c_1,\dots,  c_{n-p}\ge 0\\
\sum c_k \le 2n} } \eta_i^{\sum c_k  } \|\nab^{\otimes(1+c_1)}\tv\|_{L^\infty(B(x_i,\eta_i))} \dots \|\nab ^{\otimes(1+c_{n-p})}\tv\|_{L^\infty(B(x_i,\eta_i))} \|\nab\ka^{(p),f_i}\|_{L_\ga^2(B(x_i,\eta_i))}.
\end{multline}
Note that by our specific choice of extension $\tv$ and \cref{rem:th44}, the factor $\|\nab^{\otimes(1+c_i)}\tv\|_{L^\infty(B(x_i,\eta_i))} \le \|\nab^{\otimes(1+c_i)}v\|_{L^\infty}$. By the $L_\ga^2$ triangle inequality,
\begin{align}\label{eq:nabkpfiTI}
 \|\nab\ka^{(p),f_i}\|_{L_\ga^2(B(x_i,\eta_i))} \le \|\nab\ka^{(p),f}\|_{L_\ga^2(B(x_i,\eta_i))} + \|\nab\ka^{(p),f-f_i}\|_{L_\ga^2(B(x_i,\eta_i))}.
\end{align}
In view of the estimate \eqref{eq:comm2} from \cref{thm:comm2}, we have
\begin{align}\label{propk1}
\int_{\R^{\d+\k}} \zg |\nab \kappa^{(p),f-f_i }|^2 &= \int_{\R^{\d+\k}} \zg |\nab \kappa^{(p),(\frac1N \rho_{x_i}^{(\eta_i)} - \mu\indic_{B(x_i, R_i)}) }|^2  \nn\\
&\le C(\ell\|\nab^{\otimes 2}\tv\|_{L^\infty})^{2p}\int_{\supp\nab\tv} \zg |\nab h^{(\frac1N \rho_{x_i}^{(\eta_i)} - \mu\indic_{B(x_i, R_i)}) }|^2 \nn\\
&\le C (\ell\|\nab^{\otimes 2}\tv\|_{L^\infty})^{2p}\int_{(\R^\d)^2}\g(x-y)d\Big(\frac1N\rho_{x_i}^{(\eta_i)}- \mu\indic_{B(x_i, R_i)}\Big)^{\otimes 2}(x,y).
\end{align}
To bound this last line, we argue as follows. Set $\mu_{\eta_i}(x) \coloneqq \eta_i^\d \mu(\eta_i x)$. Then rescaling and using that $\int d\Big(\frac1N\rho_{x_i}^{(\eta_i)}- \mu\indic_{B(x_i, R_i)}\Big) = 0$, we see that
\begin{multline}
\int_{(\R^\d)^2}\g(x-y)d\Big(\frac1N\rho_{x_i}^{(\eta_i)}- \mu\indic_{B(x_i, R_i)}\Big)^{\otimes 2}(x,y)  \\
=\eta_i^{-\s}\int_{(\R^\d)^2}\g(x-y)d\Big(\frac1N\rho_{x_i}^{(1)}- \mu_{\eta_i}\indic_{B(x_i, R_i/\eta_i)}\Big)^{\otimes 2}(x,y). \label{propk1'}
\end{multline}
Using that $\|\g\ast\rho_0^{(1)}\|_{L^\infty} \le C$, we directly bound
\begin{align}
&\Big|\int_{(\R^\d)^2}\g(x-y)d\Big(\frac1N\rho_{x_i}^{(1)}\Big)(y)d\Big(\frac1N\rho_{x_i}^{(1)}- \mu_{\eta_i}\indic_{B(x_i, R_i/\eta_i)}\Big)(x)\Big| \nn\\
&\le \frac1N\int_{\R^\d}\|\g\ast\rho_{x_i}^{(1)}\|_{L^\infty}d\Big|\frac1N\rho_{x_i}^{(1)}- \mu_{\eta_i}\indic_{B(x_i, R_i/\eta_i)}\Big| \nn\\
&\le \frac{C}{N^2}.\label{propk1''}
\end{align}
For the remaining term, we use that $\g(x-y)\le C$, for some $C\ge 0$, if $|x-y|\ge 1$ to estimate
\begin{align}
&\int_{(\R^\d)^2}\g(x-y)d\Big(\mu_{\eta_i}\indic_{B(x_i, R_i/\eta_i)}\Big)^{\otimes 2}(x,y) \nn\\
&\le  \int_{|x-y|\le 1}\g(x-y) d\Big(\mu_{\eta_i}\indic_{B(x_i, R_i/\eta_i)}\Big)^{\otimes 2}(x,y)  +  C\int_{|x-y|\ge 1}d\Big(\mu_{\eta_i}\indic_{B(x_i, R_i/\eta_i)}\Big)^{\otimes 2}(x,y)\nn\\
&\le \|\mu_{\eta_i}\|_{L^\infty}\int_{\R^\d}d\Big(\mu_{\eta_i}\indic_{B(x_i, R_i/\eta_i)}\Big)(x)\int_{|x-y|\le 1}\g(x-y) + \frac{C}{N^2} \nn\\
&\le \frac{C\|\mu_{\eta_i}\|_{L^\infty}}{N} + \frac{C}{N^2} \le \frac{C}{N^2}, \label{propk1'''}
\end{align}
where the second line follows from Fubini-Tonelli, the third line from $\int_{B(x_i,R_i)}d\mu= \frac1N$, and the fourth line from $\eta^{\d}\|\mu\|_{L^\infty(\Omega)} \le \la^\d\|\mu\|_{L^\infty(\Omega)} = \frac1N$. Combining \eqref{propk1'}, \eqref{propk1''}, \eqref{propk1'''} and recalling our starting point \eqref{propk1}, we arrive arrive at
\begin{align}\label{propk1''''}
\int_{\R^{\d+\k}} \zg |\nab \kappa^{(p),f-f_i }|^2  \le \frac{C (\ell\|\nab^{\otimes 2}v\|_{L^\infty})^{2p}}{N^2\eta_i^{\s}}.
\end{align}
Applying this bound to the second term on the right-hand side of \eqref{eq:nabkpfiTI}, then inserting into the right-hand side of \eqref{kfi}, we obtain
\begin{multline}\label{eq:nabxknfi}
\|\nab_x\kappa^{(n), f_i} \|_{L^\infty (B(x_i, \eta_i/2))} \le \frac{C}{\eta_i^{\frac{\s}{2}+1}} \sum_{p=0}^n    \sum_{\substack{c_1,\dots,  c_{n-p}\ge 0\\
\sum c_k \le 2n} } \eta_i^{\sum c_k  } \|\nab^{\otimes(1+c_1)}v \|_{L^\infty} \dots \|\nab ^{\otimes(1+c_{n-p})} v\|_{L^\infty}\\
\times\Big(\|\nab\ka^{(p),f}\|_{L_\ga^2(B(x_i,\eta_i))} + \Big(\frac{C(\ell\|\nab^{\otimes 2}v\|_{L^\infty})^{2p}}{N^2\eta_i^{\s}}\Big)^{1/2}\Big).
\end{multline}
Abbreviating 
\begin{align}\label{eq:Apdef}
A_{v,p} \coloneqq  \sum_{\substack{c_1,\dots,  c_{n-p}\ge 0\\
\sum c_k \le 2  n} } \lambda^{\sum c_k  } \|\nab^{\otimes(1+c_1)}v \|_{L^\infty} \dots \|\nab ^{\otimes(1+c_{n-p})} v\|_{L^\infty},
\end{align}
with the convention that $A_{v,0}\coloneqq 1$, it follows from applying the bound \eqref{eq:nabxknfi} to the right-hand side of \eqref{eq:T2knfiMVT} that
\begin{align}\label{propk3}
\Big|\int_{\R^{\d+\k}} \kappa^{(n), f_i}    d(\delta_{x_i}-\tl\rho_{x_i}^{(\eta_i)})\Big| \le  \frac{C}{\eta_i^{\frac{\s}{2}}}\sum_{p=0}^{n} A_{v,p} \Big(\|\nab\ka^{(p),f}\|_{L_\ga^2(B(x_i,\eta_i))} + \Big(\frac{(\ell\|\nab^{\otimes 2}v\|_{L^\infty})^{2p}}{N^2\eta_i^{\s}}\Big)^{1/2}\Big).
\end{align}

Applying \eqref{propk3} and \eqref{propk2} to the first and second terms on the right-hand side of \eqref{eq:T2kndcomp}, respectively, and then summing over $i\in I_{\Omega'}$, we obtain that 
\begin{equation}
|\Te_2| \le C\sum_{i\in I_{\Omega'}} \frac{\|\nab\tv\|_{L^\infty}^n}{N^2 \eta_i^\s} 
+ \frac{C}N   \sum_{p=0}^n \sum_{i\in I_{\Omega'}} \frac{ A_{v,p}}{\eta_i^{\s/2}}   
\Big(\|\nab\ka^{(p),f}\|_{L_\ga^2(B(x_i,\eta_i))}+ \Big(\frac{ (\ell\|\nab^{\otimes 2}v\|_{L^\infty})^{2p}}{N^2 \eta_i^\s}\Big)^{1/2} \Big) .
\end{equation}
Using Cauchy-Schwarz and the disjointness of the balls $B(x_i, \eta_i)$,
\begin{align}
&\sum_{i\in I_{\Omega'}} \frac{ A_{v,p}}{\eta_i^{\s/2}} \Big(\|\nab\ka^{(p),f}\|_{L_\ga^2(B(x_i,\eta_i))}+ \Big(\frac{(\ell\|\nab^{\otimes 2}v\|_{L^\infty})^{2p}}{N^2 \eta_i^\s}\Big)^{1/2} \Big)\nn\\
&\le A_{v,p}\Big(\sum_{i\in I_{\Omega'}}\eta_i^{-\s}\Big)^{1/2}\Big(\sum_{i\in I_{\Omega'}} \|\nab\ka^{(p),f}\|_{L_\ga^2(B(x_i,\eta_i))}^2\Big)^{1/2} + \sum_{i\in I_{\Omega'}} \frac{A_{v,p}(\ell\|\nab^{\otimes 2}v\|_{L^\infty})^{p}}{N\eta_i^{\s}} \nn\\
&\le A_{v,p}\Big(\sum_{i\in I_{\Omega'}}\eta_i^{-\s}\Big)^{1/2}\|\nab\ka^{(p),f}\|_{L_\ga^2(\Omega\times\R^\k)}  + \sum_{i\in I_{\Omega'}} \frac{A_{v,p}(\ell\|\nab^{\otimes 2}v\|_{L^\infty})^{p}}{N\eta_i^{\s}} \nn\\
&\le CA_{v,p}\Big(\sum_{i\in I_{\Omega'}}\eta_i^{-\s}\Big)^{1/2} (\ell\|\nab^{\otimes 2}v\|_{L^\infty})^{p} \|\nab\bar{h}_{N,\vec\eta}\|_{L_\ga^2(\Omega\times\R^\k)} + \sum_{i\in I_{\Omega'}} \frac{A_{v,p}(\ell\|\nab^{\otimes 2}v\|_{L^\infty})^{p}}{N\eta_i^{\s}},
\end{align}
where the final inequality follows from applying the estimate \eqref{eq:comm2} of \cref{thm:comm2} to $\|\nab\ka^{(p),f}\|_{L_\ga^2(\Omega\times\R^\k)}$ and the fact that $h^f= \bar h_{N, \vec{\eta}}$. Using \eqref{eq:13} to bound $\sum_{i\in I_{\Omega'}}\eta_i^{-\s}$ and \eqref{HH} to bound  $\|\nab\bar{h}_{N,\vec\eta}\|_{L_\ga^2(\Omega\times\R^\k)}$, it follows that
\begin{multline}
\sum_{i\in I_{\Omega'}} \frac{\|\nab\tv\|_{L^\infty}^n}{N^2 \eta_i^\s}  + \frac1N\sum_{p=0}^n\sum_{i\in I_{\Omega'}} \frac{ A_{v,p}}{\eta_i^{\s/2}} \Big(\|\nab\ka^{(p),f}\|_{L_\ga^2(B(x_i,\eta_i))}+ \Big(\frac{ (\ell\|\nab^{\otimes 2}v\|_{L^\infty})^{2p}}{N^2 \eta_i^\s}\Big)^{1/2} \Big)\\
 \le C\sum_{p=0}^n A_{v,p}(\ell\|\nab^{\otimes 2}v\|_{L^\infty})^{p}  \Xi.
\end{multline} 
Ultimately,
\begin{equation}\label{eq:HOFIT2fin}
|\Te_2| \le C \sum_{p=0}^n A_{v,p}(\ell\|\nab^{\otimes 2}v\|_{L^\infty})^{p}  \Xi.
\end{equation}

{\bf Step 3: the third term}.

We reduce $\Te_3$ to $\Te_2+\Te_4$, where 
\begin{align}
\Te_4  &\coloneqq \frac{1}{N^2} \sum_{i,j\in I_{\Omega'} : i\neq j} \int_{(\R^\d)^2}\nabla^{\otimes n}\g(x-y) : (v(x)-v(y))^{\otimes n} (\rho_{x_j}^{(\eta_j)} -\delta_{x_j})(x)  (\rho_{x_j}^{(\eta_i)} -\delta_{x_i})(y) \nn\\
 &=  \frac{1}{N^2}\sum_{i, j \in I_{\Omega'} :  i\neq j} \int_{\R^{\d+\k}} \kappa^{(n), \rho_{x_j}^{(\eta_j)} - \delta_{x_j}} d(\tl\rho_{x_i}^{(\eta_i)}- \delta_{x_i}). \label{eq:HOFIT4def}
\end{align}
 We break the sum into $i,j$ such that $|x_i-x_j|\le \lambda$ and $i,j$ such that $|x_i-x_j|> \lambda$.

From using \eqref{eq:Kvbnd} and $\eta_i \le \frac14\min_{j\ne i}|x_i-x_j|$,
\begin{align}\label{eq:kannear}
&\frac{1}{N^2}\sum_{i, j \in I_{\Omega'} : i\neq j, |x_i-x_j|\le \la} \Big|\int_{\R^{\d+\k}} \kappa^{(n), \rho_{x_j}^{(\eta_j)} - \delta_{x_j}} d(\tl\rho_{x_i}^{(\eta_i)}- \delta_{x_i})\Big| \nn\\
&\le \frac{C}{N^2} \|\nab v\|_{L^\infty}^n\sum_{i, j \in I_{\Omega'}:  i\ne j,  |x_i-x_j|\le \lambda}\frac{1}{|x_i-x_j|^{\s}} \nn\\
&\le C \|\nab v\|_{L^\infty}^n \Xi,
\end{align}
where the final line follows from \cref{cor:coro3}.

Recycling notation, let us abbreviate $f_j \coloneqq  \rho_{x_j}^{(\eta_j)} - \delta_{x_j}$. We write
\begin{align}
\int_{\R^{\d+\k}} \kappa^{(n), f_j} d(\tl\rho_{x_i}^{(\eta_i)}- \delta_{x_i}) = \int_{\R^{\d+\k}}\Big( \kappa^{(n),f_j}- \kappa^{(n), f_j}(x_i)\Big) d\tl\rho_{x_i}^{(\eta_i)}.
\end{align}
{Taylor expanding $\ka^{(n),f_j}$ around $x_i$ to order $m$ and using that $\rho_0^{(\eta_i)}$ has vanishing moments up to order $m$, we find that}
\begin{align}\label{eq:kanfjfarpre}
\Big|\int_{\R^{\d+\k}} \kappa^{(n), f_j} d(\tl\rho_{x_i}^{(\eta_i)}- \delta_{x_i})\Big| \le C\eta_i^{m+1}\|\nab_x^{\otimes(m+1)}\ka^{(n),f_j}\|_{L^\infty(B(x_i,\eta_i))}.
\end{align}
Unpacking the definition of $\ka^{(n),f_j}$ and using the Leibniz rule, we see that
\begin{multline}
|\nab_x^{\otimes(m+1)}\ka^{(n),f_j}| \le C\sum_{p=0}^{m+1} \sum_{r=1}^{m+1-p}\sum_{\substack{1\le c_1,\ldots,c_r \\ c_1+\cdots+c_r = m+1-p}}\\
\Big|\int_{\R^{\d+\k}}\nab_x^{\otimes p}\nab^{\otimes n}\g(\cdot-y): \nab_x^{\otimes c_1}\tv\otimes \cdots \otimes \nab_x^{\otimes c_r}\tv \otimes (\tv-\tv(y))^{\otimes n-r}df_j(y)\Big|,
\end{multline}
where sums and products are vacuous if $p=m+1$. It follows from the mean-value theorem that in $B(x_i,\eta_i)$,
\begin{multline}
\Big|\int_{\R^{\d+\k}}\nab_x^{\otimes p}\nab^{\otimes n}\g(\cdot-y): \nab_x^{\otimes c_1}\tv\otimes \cdots \otimes \nab_x^{\otimes c_r}\tv \otimes (\tv-\tv(y))^{\otimes n-r}df_j(y)\Big| \\
\le C\|\nab^{\otimes c_1}v\|_{L^\infty}\cdots\|\nab^{\otimes c_r}v\|_{L^\infty}\|\nab v\|_{L^\infty}^{n-r}\int_{\R^{\d}}|\cdot-y|^{-\s-(p+r)}d|f_j|(y).
\end{multline}
Implicitly, we have used that $\tv$ is the trivial extension in $\R^{\d}\times\{|z|\le \ell\}^\k$ to replace the norms of $\tv$ by norms of $v$. Hence, making the change of index $q\coloneqq p+r$, we find in $B(x_i,\eta_i)$
\begin{multline}
|\nab_x^{\otimes(m+1)}\ka^{(n),f_j}|  \le C\sum_{q=1}^{m+1}\sum_{r=0}^{{q}} \sum_{\substack{1\le c_1,\ldots,c_r \\ c_1+\cdots+c_r = m+1-(q-r)}}\|\nab^{\otimes c_1}v\|_{L^\infty}\cdots\|\nab^{\otimes c_r}v\|_{L^\infty}\|\nab v\|_{L^\infty}^{n-r}\\
\times \int_{\R^{\d}}|\cdot-y|^{-\s-q}d|f_j|(y),
\end{multline}
with the convention that ranges are vacuous when they are nonsensical. After unpacking the definition of $f_j$, the preceding bound implies that
\begin{multline}
\|\nab_x^{\otimes(m+1)}\ka^{(n),f_j}\|_{L^\infty(B(x_i,\eta_i))} \le C\sum_{q=1}^{m+1}\sum_{r=0}^{{q}} \sum_{\substack{1\le c_1,\ldots,c_r \\ c_1+\cdots+c_r = m+1-(q-r)}}  \\ 
\|\nab^{\otimes c_1}v\|_{L^\infty}\cdots\|\nab^{\otimes c_r}v\|_{L^\infty}\|\nab v\|_{L^\infty}^{n-r}|x_i-x_j|^{-\s-q}.
\end{multline}
Hence, recalling our starting point \eqref{eq:kanfjfarpre}, we obtain
\begin{multline}
\frac{1}{N^2}\sum_{i,j\in I_{\Omega'} :i\ne j, |x_i-x_j|>\la} \Big|\int_{\R^{\d+\k}} \kappa^{(n), \rho_{x_j}^{(\eta_j)} - \delta_{x_j}} d(\tl\rho_{x_i}^{(\eta_i)}- \delta_{x_i})\Big| \le  \sum_{q=1}^{m+1}\sum_{r=0}^{{q}} \sum_{\substack{1\le c_1,\ldots,c_r \\ c_1+\cdots+c_r = m+1-(q-r)}}  \\ 
\|\nab^{\otimes c_1}v\|_{L^\infty}\cdots\|\nab^{\otimes c_r}v\|_{L^\infty}\|\nab v\|_{L^\infty}^{n-r} \frac{C}{N^2}\sum_{i,j\in I_{\Omega'} :i\ne j, |x_i-x_j|>\la}\frac{\eta_i^{m+1}}{|x_i-x_j|^{\s+q}}.
\end{multline}
Applying \cref{prop:multiscale} with $a=q$ (recall that, by choice of $\Omega',\Omega$, the condition $i,j \in I_{\Omega'}$ implies that $|x_i-x_j|\le \ell$ and $\dist(x_i,\p\Omega)\ge 4\ell$), it follows that
\begin{multline}\label{eq:kanfar}
\frac{1}{N^2}\sum_{i,j\in I_{\Omega'} :i\ne j, |x_i-x_j|>\la} \Big|\int_{\R^{\d+\k}} \kappa^{(n), \rho_{x_j}^{(\eta_j)} - \delta_{x_j}} d(\tl\rho_{x_i}^{(\eta_i)}- \delta_{x_i})\Big| \le \sum_{q=1}^{m+1}\sum_{r=0}^{{q}} \sum_{\substack{1\le c_1,\ldots,c_r \\ c_1+\cdots+c_r = m+1-(q-r)}}  \\ 
\|\nab^{\otimes c_1}v\|_{L^\infty}\cdots\|\nab^{\otimes c_r}v\|_{L^\infty}\|\nab v\|_{L^\infty}^{n-r}\Bigg( \lambda ^{m+1-q} { \int_{\Omega \times [-\ell,\ell]^\k} 
\zg |\nab h_{N,\vec\rs}|^2 }\Bigg)\\
 + C \frac{ \#I_\Omega \|\mu\|_{L^\infty(\hat{\Omega})}\la^{\d+m+1-\s-q}}{N}  + C\frac{\#I_\Omega \|\mu\|_{L^\infty(\hat{\Omega})}\la^{m+1}}{N}\begin{cases} \ell^{\d-\s-q}, & \s+q\neq \d\\  \log (\ell/\lambda),  & \s+q=\d \end{cases}\Bigg).
\end{multline}
We want
\begin{align}
\la^{m+1}\begin{cases} \ell^{\d-\s-q}, &  \s+q\neq \d\\  \log (\ell/\lambda),  & \s+q=\d \end{cases} \ \le\  C{\la^{\d-\s}}.
\end{align}
If $\s>\d-1$, it suffices to take $m=0$; if $\d-2<\s\le \d-1$, then it suffices to take $m=1$; and if $\s=\d-2$, then it suffices to take $m=2$. 

Applying the estimates \eqref{eq:kannear}, \eqref{eq:kanfar} to \eqref{eq:HOFIT4def}, we conclude that
\begin{align}\label{eq:HOFIT4fin}
|\Te_4| \le  C\Big(\|\nab v\|_{L^\infty}^n + B_{v,m}\Big) \Xi \le CB_{v,m}\Xi,
\end{align}
where we have abbreviated
\begin{align}\label{eq:defBvm}
B_{v,m} \coloneqq \sum_{q=1}^{m+1}\sum_{r=0}^{{q}} \sum_{\substack{1\le c_1,\ldots,c_r \\ c_1+\cdots+c_r = m+1-(q-r)}}\la^{m+1-q}\|\nab^{\otimes c_1}v\|_{L^\infty}\cdots\|\nab^{\otimes c_r}v\|_{L^\infty}\|\nab v\|_{L^\infty}^{n-r}.
\end{align}
Recalling the definition \eqref{eq:Apdef} of $A_{v,p}$, one readily checks that $B_{v,m} \le  A_{v,0}$, for the choices of $m$ above. Hence, combining with \eqref{eq:HOFIT2fin}, we obtain
\begin{align}\label{eq:HOFIT3fin}
|\Te_3| \le C B_{v,m}\Xi+  C \sum_{p=0}^n A_{v,p}(\ell\|\nab^{\otimes 2}v\|_{L^\infty})^p \Xi \le C \sum_{p=0}^n A_{v,p}(\ell\|\nab^{\otimes 2}v\|_{L^\infty})^p\Xi.
\end{align}
Combining the estimates \eqref{eq:HOFIT1fin}, \eqref{eq:HOFIT2fin}, \eqref{eq:HOFIT3fin}, \eqref{eq:HOFIT4fin} for $\Te_1$, $\Te_2$, $\Te_3$, $\Te_4$, respectively, the proof of \eqref{mainn} is then complete. This completes the proof of \cref{thm:FI}.

\subsection{Improved estimates for macroscopic $\ell$}\label{ssec:FImac}
When $\ell \asymp 1$, we can improve the higher-order estimate \eqref{mainn} of \cref{thm:FI} in terms of the regularity dependence on $v$.

\begin{prop}\label{prop:impHOFI}
Under the same assumptions as above, for any $m\ge 0$, we have
\begin{multline}\label{eq:impHOFI}
\Big|\int_{(\R^\d)^2\setminus\triangle}\nabla^{\otimes n}\g(x-y) : (v(x)-v(y))^{\otimes n} d\Big(\frac{1}{N}\sum_{i=1}^N\delta_{x_i}-\mu\Big)^{\otimes 2}(x,y)\Big| \\ 
\le C(\ell\|\nab^{\otimes 2}v\|_{L^\infty})^n\Xi + C\sum_{q=1}^{m+1}\sum_{r=0}^{{ q}}D_{v,q,r}\Bigg(\frac{\# I_{\Omega'}\la^{m+1}\ell^{-\s-q}}{N} +\la^{m+1-q}\Xi \\
 + \frac{\la^{m+1}\#I_\Omega \|\mu\|_{L^\infty(\hat {\Omega})} }{N} \begin{cases}  \ell^{\d-\s-q}, &  \s+q\neq \d \\
\log (\ell/\lambda),  & \s+q=\d \end{cases} \\
+\frac{\la^{m+1}\#I_{\Omega'}}{N}\begin{cases}\|\mu\|_{L^\infty}\la^{\d-\s-q}, & \s+q>\d \\ \|\mu\|_{L^\infty}( \ell^{\d-\s-q} - \la^{\d-\s-q}) +  \|\mu\|_{L^1}\ell^{-\s-q}, & {\s+q<\d} \\ \|\mu\|_{L^\infty}\log(\ell/\la) + \|\mu\|_{L^1}\ell^{-\s-q}, & {\s+q=\d} \end{cases}\Bigg),
\end{multline}
where the constant $C>0$ depends only on $\d,\s,n,m$ and
\begin{align}\label{eq:Dvqrdef}
D_{v,q,r} \coloneqq \sum_{\substack{1\le c_1,\ldots,c_r \\ c_1+\cdots+c_r = m+1-(q-r)}}\|\nab^{\otimes c_1}v\|_{L^\infty}\cdots\|\nab^{\otimes c_r}v\|_{L^\infty}\|\nab v\|_{L^\infty}^{n-r}.
\end{align}
\end{prop}

\begin{remark}
When $\ell \asymp 1$, the additive error in the estimate \eqref{mainn} is $\asymp \la^{\d-\s}$. We can achieve the same error in \eqref{eq:impHOFI} by taking $m$ large enough. Namely, if $\s>\d-1$, then it suffices to take $m=0$, which means our estimate \eqref{eq:impHOFI} only requires up to two derivatives of $v$. If $\d-2<\s\le \d-1$, then it suffices to take $m=1$, meaning the estimate still only requires up to two derivatives of $v$. And if $\s=\d-2$, then it suffices to take $m=2$, meaning the estimate only requires up to three derivatives of $v$. This is a much better regularity dependence than in \eqref{mainn}, particularly when $n$ is large.

\end{remark}

\begin{proof}[Proof of \cref{prop:impHOFI}]
The strategy is the same as in the previous subsection. The only step that requires modification is the estimate for $\Te_2$, in particular avoiding the use of \cref{th44}.

Using the decomposition \eqref{eq:ffisplit} and the bounds \eqref{eq:T2kndcomp}, \eqref{propk2'}, we see that
\begin{align}\label{eq:impHOT2pre}
|\Te_2| \le \frac{1}{N}\sum_{i\in I_{\Omega'}}\Big|\int_{\R^{\d+\k}} \kappa^{(n), f_i}    d(\delta_{x_i}-\tl\rho_{x_i}^{(\eta_i)})\Big| +\sum_{i\in I_{\Omega'}} \frac{C \|\nab\tv\|_{L^\infty}^n}{N^2 \eta_i^{\s}}.
\end{align}
For the first term on the right-hand side, we use the same Taylor expansion argument as in the previous subsection for $\Te_4$ to obtain
\begin{multline}\label{eq:impHOT20}
 \frac{1}{N}\sum_{i\in I_{\Omega'}}\Big|\int_{\R^{\d+\k}} \kappa^{(n), f_i}    d(\delta_{x_i}-\tl\rho_{x_i}^{(\eta_i)})\Big| \le  C\sum_{q=1}^{m+1}\sum_{r=0}^{{ q}} \sum_{\substack{1\le c_1,\ldots,c_r \\ c_1+\cdots+c_r = m+1-(q-r)}}\\
\|\nab^{\otimes c_1}v\|_{L^\infty}\cdots\|\nab^{\otimes c_r}v\|_{L^\infty}\|\nab v\|_{L^\infty}^{n-r}  \frac{1}{N}\sum_{i\in I_{\Omega'}}\eta_i^{m+1}\int_{\R^{\d+\k}}|\cdot-y|^{-\s-q}d|f_i|(y).
\end{multline}
Unpacking the definition \eqref{eq:ffisplit} of $f_i$ and using that $\eta_i\le \frac14\min_{j\ne i}|x_i-x_j|$, we need to estimate
\begin{align}\label{eq:impHOT21}
\frac1N\sum_{i\in I_{\Omega'}}\frac{\eta_i^{m+1}}{N}\sum_{1\le j\le N : j\ne i} |x_i-x_j|^{-\s-q}
\end{align}
and
\begin{align}\label{eq:impHOT22}
\frac1N\sum_{i\in I_{\Omega'}}\eta_i^{m+1}\int_{|x_i-y| \ge \la} |x_i-y|^{-\s-q}d\mu(y).
\end{align}

For \eqref{eq:impHOT21}, we note that $j\notin I_\Omega$ implies $|x_i-x_j| \ge \ell$. So,
\begin{multline}
\frac1N\sum_{i\in I_{\Omega'}}\frac{\eta_i^{m+1}}N\sum_{1\le j\le N : j\ne i} |x_i-x_j|^{-\s-q} \le \sum_{\substack{i\in I_{\Omega'}, j \in I_{\Omega} \\ |x_i-x_j|\le \la}}\frac{\eta_i^{m+1}}{N^2}|x_i-x_j|^{-\s-q} \\
+ \sum_{\substack{i\in I_{\Omega'}, j \in I_{\Omega} \\ \la \le |x_i-x_j| \le \ell}}\frac{\eta_i^{m+1}}{N^2}|x_i-x_j|^{-\s-q} + \sum_{\substack{i\in I_{\Omega'}, j\\ \ell < |x_i-x_j| }}\frac{\eta_i^{m+1}}{N^2} |x_i-x_j|^{-\s-q}.
\end{multline}
For the first right-hand side term, we use \cref{cor:coro3} (note that $\eta_i^{m+1}|x_i-x_j|^{-q}\le \la^{m+1-q}$ by definition of $\eta_i$); for the second term, we use \cref{prop:multiscale}; and for the third term, we crudely use $|x_i-x_j|\ge \ell$. All together,
\begin{multline}\label{eq:impHOT21fin}
\frac{1}{N^2}\sum_{i\in I_{\Omega'}}\sum_{1\le j\le N : j\ne i}\frac{\eta_{i}^{m+1}}{ |x_i-x_j|^{\s+q}} \le C\frac{\# I_{\Omega'}\la^{m+1}\ell^{-\s-q} }{N} \\
+C\la^{m+1-q}{\int_{\Omega\times [-\ell,\ell]^\k} \zg|\nab h_{N,\vec\rs}|^2 }+ C\frac{\#I_\Omega\|\mu\|_{L^\infty(\hat \Omega)}\la^{\d+m+1-\s-q}}{N}\\
+ C\frac{\#I_\Omega \|\mu\|_{L^\infty(\hat {\Omega})}\la^{m+1}}{N}\begin{cases}    \ell^{\d-\s-q}, & \s+q\neq \d\\
\log (\ell/\lambda),  &  \s+q=\d. \end{cases}
\end{multline}

For \eqref{eq:impHOT22}, if $\s+q > \d$, we may directly bound
\begin{align}
\int_{|x_i-y| \ge \la} |x_i-y|^{-\s-q}d\mu(y) \le C\|\mu\|_{L^\infty}\la^{\d-\s-q}.
\end{align}
Since $q\ge 1$, we see that this is always the case if $\s>\d+1$. 
Otherwise, we divide the integration into regions $\la \le |x_i-y|\le \ell$ and $|x_i-y|>\ell$ to obtain
\begin{multline}
\int_{|x_i-y| \ge \la} |x_i-y|^{-\s-q}d\mu(y) \le C\|\mu\|_{L^\infty}\Big( \ell^{\d-\s-q} - \la^{\d-\s-q} \\
+ \log(\ell/\la)\indic_{\s+q=\d}\Big) + C\ell^{-\s-q}\|\mu\|_{L^1}.
\end{multline}
Combining cases, we conclude that
\begin{multline}\label{eq:impHOT22fin}
\eqref{eq:impHOT22} \le C\frac{\la^{m+1}\#I_{\Omega'}\|\mu\|_{L^\infty}}{N}\Big(\la^{\d-\s-q}\indic_{\s+q>\d}+( \ell^{\d-\s-q} - \la^{\d-\s-q})\indic_{\s+q<\d} \\
+ \log(\ell/\la)\indic_{\s+q=\d}\Big) +  C\frac{\la^{m+1}\#I_{\Omega'}\ell^{-\s-q}\|\mu\|_{L^1}}{N}\indic_{\s+q\le \d}.
\end{multline}

Hence, combining \eqref{eq:impHOT21fin}, \eqref{eq:impHOT22fin} with \eqref{eq:impHOT20}, \eqref{eq:impHOT2pre}, we obtain that
\begin{multline}
\frac1N\sum_{i\in I_{\Omega'}}\Big|\int_{\R^{\d+\k}} \kappa^{(n), f_i}    d(\delta_{x_i}-\tl\rho_{x_i}^{(\eta_i)})\Big| \le C\sum_{q=1}^{m+1}\sum_{r=0}^{{ q}}D_{v,q,r}\Bigg(\frac{\# I_{\Omega'}\la^{m+1}\ell^{-\s-q}}{N}  \\
+\la^{m+1-q}\Xi + \frac{\la^{m+1}\#I_\Omega \|\mu\|_{L^\infty(\hat {\Omega})} }{N} \begin{cases}  \ell^{\d-\s-q}, &  \s+q\neq \d \\
\log (\ell/\lambda),  & \s+q=\d \end{cases} \\
+\frac{\la^{m+1}\#I_{\Omega'}}{N}\begin{cases}\|\mu\|_{L^\infty}\la^{\d-\s-q}, & \s+q>\d \\ \|\mu\|_{L^\infty}( \ell^{\d-\s-q} - \la^{\d-\s-q}) +  \|\mu\|_{L^1}\ell^{-\s-q}, & {\s+q<\d} \\ \|\mu\|_{L^\infty}\log(\ell/\la) + \|\mu\|_{L^1}\ell^{-\s-q}, & {\s+q=\d} \end{cases}\Bigg),
\end{multline}
where $D_{v,q,r}$ is as defined in \eqref{eq:Dvqrdef}. Applying this bound to the right-hand side of \eqref{eq:impHOT2pre} and using \eqref{eq:13} to handle the first term, we conclude that
\begin{multline}\label{eq:impHOFIT2fin}
|\Te_2| \le C\|\nab v\|_{L^\infty}^n\Xi + C\sum_{q=1}^{m+1}\sum_{r=0}^{{q}}D_{v,q,r}\Bigg(\frac{\# I_{\Omega'}\la^{m+1}\ell^{-\s-q}}{N}  \\
+\la^{m+1-q}\Xi + \frac{\la^{m+1}\#I_\Omega \|\mu\|_{L^\infty(\hat {\Omega})} }{N} \begin{cases}  \ell^{\d-\s-q}, &  \s+q\neq \d \\
\log (\ell/\lambda),  & \s+q=\d \end{cases} \\
+\frac{\la^{m+1}\#I_{\Omega'}}{N}\begin{cases}\|\mu\|_{L^\infty}\la^{\d-\s-q}, & \s+q>\d \\ \|\mu\|_{L^\infty}( \ell^{\d-\s-q} - \la^{\d-\s-q}) +  \|\mu\|_{L^1}\ell^{-\s-q}, & {\s+q<\d} \\ \|\mu\|_{L^\infty}\log(\ell/\la) + \|\mu\|_{L^1}\ell^{-\s-q}, & {\s+q=\d} \end{cases}\Bigg).
\end{multline}

Combining this new bound for $|\Te_2|$ with the bounds for $|\Te_1|$ and $|\Te_3|$ from the previous subsection, the proof is complete.
\end{proof}

{
\subsection{Unlocalized estimates}\label{ssec:FIunloc}
We close out this section by proving the estimate \eqref{eq:HOFIunloc}, which is the analogue of the estimate \eqref{mainn} of \cref{thm:FI} when $\ell=\infty$ (i.e. the set $\Om=\R^\d$). This then completes the proof of \cref{thm:mainunloc}.

\begin{proof}
We follow the same $\Te_1,\Te_2,\Te_3$ decomposition as before with small modifications that we now sketch.

For $\Te_1$, instead of the estimate \eqref{eq:comm2dual} of \cref{thm:comm2}, we use the unlocalized estimate \eqref{eq:comm2dual''} to obtain
\begin{align}
|\Te_1| \le C\|\nab v\|_{L^\infty}^n \Xi,
\end{align}
where $\Xi$ is as in \eqref{defXi} but with $\Om=\hat\Om=\R^\d$.

For $\Te_2$, we the bound \eqref{eq:impHOFIT2fin} from \cref{ssec:FImac} carries over unchanged ($\ell$ is now just an arbitrary parameter such that $\ell\ge\la$, with no relation to $\supp v$, which we relabel as $R$). In it, we take $\Om'=\hat\Om=\R^\d$.

For $\Te_3$, we use the same reduction to $\Te_2+\Te_4$, where the estimate for $\Te_4$ is unchanged from \cref{ssec:FImain}, and we again set $\Om'=\R^\d$.

Combining these estimates and taking $R=1$, which is allowed by our assumption that $\la<1$, we arrive at the stated inequality. 
\end{proof}

\begin{remark}
If we further assume that $\la\le\Big(\frac{\|\mu\|_{L^1}}{\|\mu\|_{L^\infty}}\Big)^{1/\d}$, then we may take $R = \Big(\frac{\|\mu\|_{L^1}}{\|\mu\|_{L^\infty}}\Big)^{1/\d}$ at the conclusion of the preceding proof, which yields a final estimate with terms that are better balanced in terms of their $\mu$ dependence. 
\end{remark}

}
\section{Application: optimal mean-field convergence rate}\label{sec:appMF}
Using our main technical result \cref{thm:mainunloc}, we now show convergence of the empirical measure for the mean-field particle dynamics \eqref{eq:MFode} to a (necessarily unique) solution of the limiting PDE \eqref{eq:MFlim} in the modulated energy metric with the optimal rate $N^{\frac{\s}{\d}-1}$. This proves \cref{thm:mainMF}.

\begin{proof}[Proof of \cref{thm:mainMF}]
We recall (e.g. see \cite[Lemma 2.1]{Serfaty2020} or \cite[Lemma 3.6]{RS2024ss}) that $\Fr_N(\ux_N^t,\mu^t)$ satisfies the differential inequality
\begin{align}
\frac{d}{dt}\Fr_N(\ux_N^t,\mu^t) \leq \int_{(\R^\d)^2\setminus\triangle}\nabla\g(x-y)\cdot\pa*{u^t(x)-u^t(y)}d\Big(\frac{1}{N}\sum_{i=1}^N\delta_{x_i^t}-\mu^t\Big)^{\otimes2}(x,y) .
\end{align}
where $u^t\coloneqq -\M\nabla\g\ast\mu^t+\mathsf{V}$. Applying the first-order estimate \cref{main2} of \cref{thm:mainunloc} pointwise in $t$ to the preceding right-hand side, then integrating with respect to time both sides of the resulting inequality and applying the fundamental theorem of calculus, it follows that
\begin{multline}
\Fr_N(\ux_N^t,\mu^t) + \frac{\log (N\|\mu^t\|_{L^\infty}) }{2 N\d} \indic_{\s=0} + C\|\mu^t\|_{L^\infty}^{\frac{\s}{\d}} N^{\frac{\s}{\d}-1} \leq \Fr_N(\ux_N^0,\mu^0) + \frac{\log (N\|\mu^t\|_{L^\infty}) }{2 N\d} \indic_{\s=0} \\ 
+C\|\mu^t\|_{L^\infty}^{\frac{\s}{\d}}N^{\frac{\s}{\d}-1}+ C\int_0^t \|\nabla u^\tau \|_{L^\infty}\Big(\Fr_N(\ux_N^\tau,\mu^\tau) +  \frac{\log (N\|\mu^\tau\|_{L^\infty})}{2N\d}\indic_{\s=0} + \frac{C(N\|\mu^\tau\|_{L^\infty})^{\frac{\s}{\d}}}{N}\Big)d\tau
\end{multline}
for some constant $C>0$ depending only on $|\M|,\s,\d$. 
Assuming that the constant $C>0$ above is sufficiently large depending on $\d,\s,$ the left-hand side defines a nonnegative quantity  in view of \eqref{eq:pr2}. An application of the Gr\"onwall-Bellman lemma then completes the proof.
\end{proof}

\appendix
\section{An alternative proof in the Coulomb case}\label{appA}
We show in this appendix how the estimates \eqref{eq:comm2}, \eqref{eq:comm2dual} of \cref{thm:comm2} may be obtained, under stronger demands on the regularity of $v$, via an iterated stress-energy tensor structure when  $\s=\d-2, \d-1$, corresponding to the Coulomb case in $\R^\d$ and $\R^{\d+1}$, respectively.  Remark that $\ga=0$ in these two cases. 

\begin{prop}\label{prop:commalt}
Let $\s=\d+\k-2$. Let $v:\R^{\d+\k}\rightarrow\R^{\d+\k}$ be a smooth vector field. Given a Schwartz function $f\in\Sc_0(\R^{\d+\k})$, let $h^f=\g*f$ and $\kappa^{(n),f}$ be as in \eqref{eq:kandef}, where $\g$ is viewed as a function on $\R^{\d+\k}$. There is a constant $C>0$ depending only on $\d+\k,\s,n$, such that
\begin{equation}\label{eq:commalt}
\|\nab\ka^{(n),f}\|_{L^2(\R^{\d+\k})} \leq C \sum_{\substack{0\leq c_1,\ldots,c_{n}\leq n \\ c_1+\cdots+c_{n}=n}}\|\nabla^{\otimes c_1}v\|_{L^\infty}\cdots\|\nabla^{\otimes c_{n}}v\|_{L^\infty}   \|\nab h^f\|_{L^2(\supp\nab v)}.
\end{equation}
Consequently, for any $f,w\in\Sc_0(\R^{\d+\k})$, it holds that
\begin{multline}\label{eq:commaltdual}
\left|\int_{(\R^{\d+\k})^2}\nabla^{\otimes n}\G(x-y)\cdot(v(x)-v(y))^{\otimes n} df(y)dw(x)\right|\\
\leq C  \sum_{\substack{0\leq c_1,\ldots,c_{n}\leq n \\ c_1+\cdots+c_{n}=n}}\|\nabla^{\otimes c_1}v\|_{L^\infty}\cdots\|\nabla^{\otimes c_{n}}v\|_{L^\infty}\|\nab h^f\|_{L^2(\supp \nab v)} \|\nab h^w\|_{L^2(\supp \nab v)} .
\end{multline}
\end{prop}

\begin{remark}
When $\s=\d-1$ and $\k=1$, the estimates \eqref{eq:commalt}, \eqref{eq:commaltdual} are still valid when $f = f_0\delta_{\R^\d\times\{0\}^\k}$ and $w = w_0\delta_{\R^\d\times\{0\}^\k}$, for $f_0,w_0\in\Sc(\R^\d)$. Indeed, we may approximate such measures by Schwartz functions in $\R^{\d+\k}$ and, as remarked in \cref{sec:FO}, $\|\nab h^f\|_{L^2(\R^{\d+\k})}, \|\nab h^w\|_{L^2(\R^{\d+\k})}$ are finite if $f_0,w_0 \in \dot{H}^{-1/2}(\R^{\d})$. 
\end{remark}

We now turn to the proof of \cref{prop:commalt}, which we break into a series of lemmas. 

By Cauchy-Schwarz, it is  an immediate consequence of \cref{prop:comm} that
\begin{multline}\label{eq:foduce}
\Big|\int_{(\R^{\d+\k})^2}(v(x)-v(y))\cdot\nabla\G(x-y)df(x)dw(y)\Big|\\
\leq C \|\nabla v\|_{L^\infty} \|\nab h^f\|_{L^2(\supp\nab v)} \|\nab h^w\|_{L^2(\supp \nab v)},
\end{multline}
where the constant $C>0$ depends only on $\d,\k,\s$. We now seek to inductively prove that $\ka^{(m),f}$ satisfies the announced $L^2$ estimate for all $m\in\N$.
To accomplish this goal, we use a variation-of-energy argument inspired by the prior work \cite{Serfaty2023}, which was limited to estimates corresponding to $m\leq 2$.

It is a straightforward calculus exercise to check that for any $n\in\N$,
\begin{align}
\int_{(\R^{\d+\k})^2}\nabla^{\otimes n}\G(x-y): (v(x)-v(y))^{\otimes n} df^{\otimes 2}(x,y) &= \frac{d^n}{dt^n}\Big|_{t=0}\int_{(\R^{\d})^2}\G(x-y)df_t^{\otimes 2}(x,y), \label{eq:dtnrhs}
\end{align}
where $f_t\coloneqq (\I+tv)\#f$ and
\begin{equation}\label{eq:ka1tfdef}
\ka_t^{(1),f} \coloneqq \int_{\R^{\d+\k}}\nabla\G(\cdot-y-tv(y))\cdot (v-v(y))d{ f}(y).
\end{equation}
We want to show that each time we differentiate with respect to $t$, we can massage the right-hand side into a finite combination of integrals involving only lower-order commutators, the estimates for which we will have already established as part of the induction. 
This part of the proof is largely algebraic. To make the computations manageable, we introduce the following notation: for integer $m\ge 0$ and a test function $f$, define the transported commutator
\begin{equation}\label{eq:kamtfdef}
\ka_{t}^{(m),f} \coloneqq \int_{\R^{\d+\k}}\nabla^{\otimes m}\G(\cdot-y-tv(y)) : (v-v(y))^{\otimes m}df(y).
\end{equation}
We adopt the convention that $\ka_t^{(0),f} \coloneqq h_t^f$. Note that when there is no ambiguity (i.e.~ $f$ is fixed), we will omit the dependence on $f$ in the superscript. Similarly, we set $\ka^{(m),f} \coloneqq \ka_0^{(m),f}$. With this notation, the following lemma gives some useful identities, in particular that $\{\ka_t^{(m)}\}_m$ satisfies a hierarchy of transport equations.

\begin{lemma}\label{lem:pthka}
For integer $m\in\N$, it holds that
\begin{align}
\p_t \ka_{t}^{(m)} + v\cdot\nabla\ka_t^{(m)} &= \ka_t^{(m+1)}  \label{eq:ptkam}
\end{align}
and for any coordinate indices $a,b$ and integer $0\leq l\leq m$, it holds that
\begin{multline}\label{eq:ptka}
\p_t\Big(\p_{a}\ka_t^{(l),f}\p_{ b}\ka_t^{(m-l),w}\Big) = \p_{ a}\ka_t^{(l+1),f}\p_{ b}\ka_t^{(m-l),w} +\p_{ a}\ka_t^{(l),f}\p_{ b}\ka_t^{(m+1-l),w} - v\cdot\nabla\Big(\p_{ a}\ka_t^{(l),f}\p_{ b}\ka_t^{(m-l),w}\Big) \\
- \p_{ a}v\cdot\nabla\ka_t^{(l),f}\p_{ b}\ka_t^{(m-l),w} - \p_{ b}v\cdot\p_{ a}\ka_t^{(l),f}\nabla\ka_t^{(m-l),f}.
\end{multline}
\end{lemma}

\begin{remark}
Since, in applications, $v^{\d+\k} = 0$ if $\k=1$, \eqref{eq:ptkam}, \eqref{eq:ptka} only depend on the first $\d$ partial derivatives of $\ka$. 
\end{remark}

\begin{proof}[Proof of \cref{lem:pthka}]
The proof is a direct computation. For the identity \eqref{eq:ptkam}, we observe that
\begin{equation}
\p_t \ka_{t}^{(m)} =-\int_{\R^{\d+\k}}\nabla^{\otimes m+1}\G(\cdot-y-tv(y)): v(y) \otimes (v-v(y))^{\otimes m}df(y),
\end{equation}
and then write $v(y)=(v(y)-v)+ v$.  For the identity \eqref{eq:ptka}, we use the identity \eqref{eq:ptkam} and the product rule:
\begin{align}
\p_t\Big(\p_{{a}}\ka_t^{(l)}\p_{{b}}\ka_t^{(m-l)}\Big) &= \p_{{a}}\Big(\ka_t^{(l+1)}-v\cdot\nabla\ka_t^{(l)}\Big)\p_{{b}}\ka_t^{(m-l)} + \p_{{a}}\ka_t^{(l)}\p_{{b}}\Big(\ka_t^{(m+1-l)}-v\cdot\nabla\ka_t^{(m-l)}\Big) \nn\\
&=-\p_{{a}}\ka_t^{(l+1)}\p_{{b}}\ka_t^{(m-l)} - \p_{{a}}v\cdot\nabla\ka_t^{(l)}\p_{{b}}\ka_t^{(m-l)} -v\cdot\nabla\p_{{a}}\ka_t^{(l)}\p_{{b}}\ka_t^{(m-l)} \nn\\
&\ph- \p_{{a}}\ka_t^{(l)}\p_{{b}}\ka_t^{(m+1-l)}-\p_{{a}}\ka_t^{(l)}\p_{ b} v \cdot\nab\ka_t^{(m-l)} - \p_{{a}}\ka_t^{(l)}v\cdot\nabla\p_{{b}}\ka_t^{(m-l)} \nn\\
&=\p_{{a}}\ka_t^{(l+1)}\p_{{b}}\ka_t^{(m-l)} +\p_{{a}}\ka_t^{(l)}\p_{{b}}\ka_t^{(m+1-l)} - v\cdot\nabla\Big(\p_{{a}}\ka_t^{(l)}\p_{{b}}\ka_t^{(m-l)}\Big) \nn\\
&\ph- \p_{{a}}v\cdot\nabla\ka_t^{(l)}\p_{{b}}\ka_t^{(m-l)} - \p_{{b}}v\cdot\p_{{a}}\ka_t^{(l)}\nabla\ka_t^{(m-l)}.
\end{align}
\end{proof}

The next lemma uses the identities of \cref{lem:pthka} to give the desired representation formula of the right-hand side of \eqref{eq:dtnrhs} in terms of expressions involving lower-order commutators.

\begin{lemma}\label{lem:dtnrep}
For any $n\in\N_0$, test function $\psi$, and $i,j\in [\d+\k]$, there exist a family of functions
\begin{equation}
\{\phi_{ablm n} : 1\leq a,b\leq \d, \ 0\leq m\leq n, \ 0\leq l\leq m\},
\end{equation}
with the property
\begin{align}
\|\nabla^{\otimes k}\phi_{abl mn}\|_{L^\infty} \leq C\sum_{\substack{0\leq c_0,\ldots,c_m\leq n-m+k \\c_0+\cdots+c_{m}=n-m+k}} \|\nabla^{\otimes c_0}  \psi \|_{L^\infty} \|\nabla^{\otimes c_1}v\|_{L^\infty}\cdots\|\nabla^{\otimes c_{m}}v\|_{L^\infty}, \label{eq:phidk}
\end{align}
for some constant $C$ depending only $\d,k,l,m,n$, such that
\begin{equation}\label{eq:dtmrep}
\frac{d^n}{dt^n} \int_{\R^{\d+\k}}\psi\comm{\nab h_t^f}{\nab h_t^w}_{ij} = \sum_{m=0}^n\sum_{a,b=1}^{\d+\k} \sum_{l=0}^m \int_{\R^{\d+\k}}\phi_{ablm n}\p_a\ka_t^{(l),f}\p_b\ka_t^{(m-l),w}.
\end{equation}
\end{lemma}
\begin{proof}
We prove the lemma by induction on $n$. For the base case $n=0$, we have by definition \eqref{eq:stdef} of the stress-energy tensor that 
\begin{equation}
\comm{\nab h_t^f}{\nab h_t^w}_{ij} = \pa*{\p_ih_t^f\p_jh_t^w +\p_jh_t^f\p_ih_t^w-\nabla h_t^f\cdot\nabla h_t^w\delta_{ij}},
\end{equation}
so the desired representation holds with constant functions
\begin{equation}
\phi_{ab000} \coloneqq \begin{cases}2\delta_{ia}\delta_{jb}, & {i\neq j} \\0 ,& {i=j \ \text{and} \ a\neq b}\\ -1, & {i= j \ \text{and} \ a=b \ \text{and} \ a\neq i} \\ 1, & {i=j=a=b}. \end{cases}
\end{equation}

For the induction step, suppose that the representation \eqref{eq:dtmrep} holds for all test functions $\psi$ with $0\leq n'\leq n$, for some $n\in\N$. Then
\begin{align}
&\frac{d^{n+1}}{dt^{n+1}}\int_{\R^{\d+\k}}\psi\comm{\nab h_t^f}{\nab h_t^w}_{ij} \nn\\
&= \frac{d}{dt}\Bigg(\sum_{a,b=1}^{\d+\k} \sum_{m=0}^n\sum_{l=0}^m \int_{\R^{\d+\k}}\phi_{ablmn}\p_a\ka_t^{(l),f}\p_b\ka_t^{(m-l),w}\Bigg)\nn\\
&= \sum_{a,b=1}^{\d+\k} \sum_{m=0}^n\sum_{l=0}^m \int_{\R^{\d+\k}}\phi_{abl mn}\Big(\p_a\ka_t^{(l+1),f}\p_b\ka_t^{(m-l),w}+\p_a\ka_t^{(l),f}\p_b\ka_t^{(m+1-l),w} \nn\\
&\quad - v\cdot\nabla\Big(\p_a\ka_t^{(l),f}\p_b\ka_t^{(m-l),w}\Big)- \p_a v\cdot\nabla\ka_t^{(l),f}\p_b\ka_t^{(m-l),w} - \p_b v\cdot\p_a\ka_t^{(l),f}\nabla\ka_t^{(m-l),w}\Big),\label{eq:repinduc}
\end{align}
where we use \cref{lem:pthka} to obtain the second equality. The terms without $v$ have the desired form.  For the remaining terms, integrating by parts, we see that
\begin{equation}
-\int_{\R^{\d+\k}}\phi_{abl mn}v\cdot\nabla\Big(\p_a\ka_t^{(l),f}\p_b\ka_t^{(m-l),w}\Big) =\int_{\R^{\d+\k}}\div(v\phi_{ablmn}) \p_a\ka_t^{(l),f}\p_b\ka_t^{(m-l),w}.
\end{equation}
Next, observe that
\begin{multline}
-\int_{\R^{\d+\k}}\phi_{ablmn}\Big((\p_a v\cdot\nabla\ka_t^{(l),f})\p_b\ka_t^{(m-l),w} + (\p_b v\cdot\nabla\ka_t^{(m-l),w})\p_a\ka_t^{(l),f}\Big) \\
= -\sum_{c=1}^{\d+\k} \int_{\R^{\d+\k}}\Big((\phi_{ablmn}\p_a v^c)\p_c \ka_t^{(l),f}\p_b\ka_t^{(m-l),w} +(\phi_{abl mn}\p_b v^c)\p_c\ka_t^{(m-l),w}\p_a\ka_t^{(l),f}\Big).
\end{multline}
Therefore,
\begin{multline}
-\sum_{a,b=1}^{\d+\k}\int_{\R^{\d+\k}}\phi_{abl mn}\Big(v\cdot\nabla\Big(\p_a\ka_t^{(l),f}\p_b\ka_t^{(m-l),w}\Big) + (\p_a v\cdot\nabla\ka_t^{(l),f}) \p_b\ka_t^{(m-l),w} + (\p_b v\cdot\nabla \ka_t^{(m-l),w})\p_a\ka_t^{(l),f} \Big)\\
=\sum_{a,b,c=1}^{\d+\k} \Bigg(\frac{1}{\d+\k}\int_{\R^{\d+\k}}\div(v\phi_{ablmn}) \p_a\ka_t^{(l),f}\p_b\ka_t^{(m-l),w} \\
- \int_{\R^{\d+\k}}\Big((\phi_{ablmn}\p_a v^c)\p_c\ka_t^{(l),f}\p_b\ka_t^{(m-l),w} -(\phi_{ablmn}\p_b v^c)\p_c \ka_t^{(m-l),w}\p_a\ka_t^{(l),f}\Big)\Bigg).
\end{multline}
Making the change of index $(a,b,c)\mapsto (c,b,a)$ and $(a,b,c)\mapsto (a,c,b)$ in the second and third terms, respectively, the right-hand side equals
\begin{equation}
\sum_{a,b=1}^{\d+\k} \int_{\R^{\d+\k}}\sum_{c=1}^{\d+\k}\Bigg(\frac{1}{\d+\k}\div(v\phi_{abl mn})  - \phi_{cblmn}\p_c v^{a} - \phi_{aclmn}\p_c v^b\Bigg)\p_a\ka_t^{(l),f}\p_b\ka_t^{(m-l),w}.
\end{equation}
Finally, we make a change of index to write
\begin{multline}
\sum_{a,b=1}^{\d+\k}\sum_{m=0}^n\sum_{l=0}^m \int_{\R^{\d+\k}}\phi_{abl m n}\Big(\p_a\ka_t^{(l+1),f}\p_b\ka_t^{(m-l),w}+ \p_a\ka_t^{(l),f}\p_b\ka_t^{(m+1-l),w}\Big) \\
=\sum_{a,b=1}^{\d+\k}\Bigg(\sum_{m=1}^{n+1}\sum_{l=1}^{m} \int_{\R^{\d+\k}}\phi_{ab(l-1)(m-1)n}\p_a\ka_t^{(l),f}\p_b\ka_t^{(m-l),w}\\
+\sum_{m=1}^{n+1}\sum_{l=0}^m\int_{\R^{\d+\k}}\phi_{ab\ga(m-1)n}\p_a\ka_t^{(l),f}\p_b\ka_t^{(m-l),w}\Bigg).
\end{multline}
Therefore, defining
\begin{multline}
\phi_{abl m(n+1)} \coloneqq\sum_{c=1}^{\d+\k}\Bigg(\frac{1}{\d+\k}\div(v\phi_{abl mn})  - \phi_{cblmn}\p_c v^{a} - \phi_{acl mn}\p_c v^b \\
+\phi_{ab(l-1)(m-1)n}+\phi_{ab\ga(m-1)n}\Bigg),
\end{multline}
with the understanding that the last two terms are zero if $l-1=-1$ or $m-1=-1$, we see that the right-hand side of \eqref{eq:repinduc} equals
\begin{equation}
\sum_{a,b=1}^{\d+\k}\sum_{m=0}^{n+1}\sum_{l=0}^m \int_{\R^{\d+\k}}\phi_{abl m(n+1)}\p_a\ka_t^{(l),f}\p_b\ka_t^{(m-l),w},
\end{equation}
as desired.

To complete the induction step, it remains to check that the functions $\phi_{abl m(n+1)}$ satisfy the derivative bounds \eqref{eq:phidk} with $n$ replaced by $n+1$. It is straightforward from the Leibniz rule that for any integer $p\geq 0$,
\begin{align}
&\|\nabla^{\otimes p}\phi_{abl m(n+1)}\|_{L^\infty} \nn\\
&\leq C\sum_{i=0}^p\sum_{a',b'=1}^{\d+\k}\Big(\|\nabla^{\otimes i}v\|_{L^\infty}\|\nabla^{\otimes p+i-l}\phi_{a'b'l mn}\|_{L^\infty} + \|\nabla^{\otimes i+1}v\|_{L^\infty}\|\nabla^{\otimes p-l}\phi_{a'b'l mn}\|_{L^\infty} \Big) \nn\\
&\quad+ C\Big(\|\nabla^{\otimes p}\phi_{ab(l-1)(m-1)n}\|_{L^\infty} + \|\nabla^{\otimes p}\phi_{abl(m-1)n}\|_{L^\infty}\Big) \nn\\
&\leq C\sum_{i=0}^p \Bigg(\|\nabla^{\otimes i}v\|_{L^\infty}\sum_{c_0+\cdots+c_{n-m}=n-m+p+1-i}\|\nabla^{\otimes c_0}\psi\|_{L^\infty}\|\nabla^{\otimes c_1}v\|\cdots\|\nabla^{\otimes c_{n-m}}v\|_{L^\infty} \nn\\
&\quad+ \|\nabla^{\otimes i+1}v\|_{L^\infty}\sum_{c_0+\cdots+c_{n-m}=n-m+p-i}\|\nabla^{\otimes c_0}\psi\|_{L^\infty}\|\nabla^{\otimes c_1}v\|\cdots\|\nabla^{\otimes c_{n-m}}v\|_{L^\infty}\Bigg) \nn\\
&\quad+\sum_{c_0+\cdots+c_{n-(m-1)}=n-(m-1)+p} \|\nabla^{\otimes c_0}\psi\|_{L^\infty}\|\nabla^{\otimes c_1}v\|_{L^\infty}\cdots \|\nabla^{\otimes c_{n-(m-1)}}v\|_{L^\infty}\nn\\
&\leq C\sum_{c_0+\cdots+c_{(n+1)-m}=(n+1)-m+p} \|\nabla^{\otimes c_0}\psi\|_{L^\infty}\|\nabla^{\otimes c_1}v\|\cdots\|\nabla^{\otimes c_{(n+1)-m}}v\|_{L^\infty},
\end{align}
where between lines the constant $C$ is taken possibly larger. This completes the proof of the lemma.
\end{proof}

\begin{remark}\label{rem:phirep}
In fact, in \cref{lem:dtnrep}, we can say more about the functions $\phi_{abl mn}$ than just the derivative estimates of \eqref{eq:phidk}. An examination of the proof of the lemma reveals that  given a test function $\psi$ in the left-hand side of \eqref{eq:dtmrep}, each $\phi_{abl mn}$ is a finite linear combination of expressions of the form
\begin{equation}
\Big(\p_{1}^{c_{01}}\cdots\p_{\d+\k}^{c_{0(\d+\k)}}\psi\Big)\Big(\p_{1}^{c_{11}}\cdots\p_{\d+\k}^{c_{1(\d+\k)}}v\Big)\cdots\Big(\p_1^{c_{(n-m)1}}\cdots\p_{\d+\k}^{c_{(n-m)(\d+\k)}}v\Big),
\end{equation}
where for each $0\leq n-m$, $c_{i1}+\cdots+c_{i(\d+\k)}=c_i$, and $c_0+\cdots+c_{n-m} = n-m$. In particular, $\supp\phi_{ablmn} \subset \supp\psi$.
\end{remark}

As a consequence of \cref{lem:dtnrep}, we have an $L^2$ bound for the first $m$ commutators implies an $L^2$ bound for the $(m+1)$-th commutator. In other words, the induction step has been proved.

\begin{lemma}\label{lem:kainduc}
Let $f,w\in\Sc_0(\R^{\d+\k})$. Then 
\begin{align}\label{firstclaim}
\int_{\R^{\d+\k}} \ka^{(n+1), f} d{w} &=\int_{(\R^{\d+\k})^2}\nabla^{\otimes n+1}\G(x-y):\pa*{v(x)-v(y)}^{\otimes n+1}df(x)dw(y)\nn \\ &=
\sum_{p=0}^n {n\choose p}\sum_{m=0}^{n-p}\sum_{a,b=1}^{\d+\k}\sum_{l=0}^m \int_{\R^{\d+\k}}\phi_{abl m(n-p)}\p_a\ka^{(l),f}\p_b\ka^{(m-l),w},
\end{align}
where the functions $\phi_{ablm(n-p)} $ satisfy \eqref{eq:phidk}.
Moreover, suppose that for every $1\leq m\leq n$, it holds that
\begin{equation}\label{eq:gradkam}
\|\nab\ka^{(m),f} \|_{L^2(\R^{\d+\k})}  \leq C\|\nab h^f\|_{L^2(\supp\nab v)} \sum_{\substack{0\leq c_1,\ldots,c_m\leq m\\ c_1+\cdots+c_m=m}} \|\nabla^{\otimes c_1}v\|_{L^\infty}\cdots\|\nabla^{\otimes c_m}v\|_{L^\infty},
\end{equation}
where $C>0$ depends only on $\d+\k,\s,n$. Then it holds that
\begin{equation}
\|\nabla\ka^{(n+1),f}\|_{L^2(\R^{\d+\k})} 
\leq C\|\nab h^f\|_{L^2(\supp\nab v)} \sum_{\substack{0\leq c_1,\ldots,c_{n+1}\leq n+1 \\ c_1+\cdots+c_{n+1}=n+1}}\|\nabla^{\otimes c_1}v\|_{L^\infty}\cdots\|\nabla^{\otimes c_{n+1}}v\|_{L^\infty},
\end{equation}
where $C$ depends again only on $\d+\k,\s,n$.
\end{lemma}

\begin{remark}
By iterating \eqref{firstclaim}, one can in principle deduce an expression for $\kappa^{(n),f}$ in terms only of combinations of first derivatives of $h^f$. The coefficients $\phi_{ablmn}$ effectively represent the coefficients  of the ``next-order stress-energy tensor" behind $\ka^{(n),f}$ which generalizes \eqref{eq:commst} to higher order. This representation should be compared with the equation \eqref{eqLk} for $L\ka^{(n),f}$.
\end{remark}

\begin{proof}[Proof of \cref{lem:kainduc}]
By the definition of $\ka^{(n+1),f}$ and Fubini-Tonelli, we have that
\begin{align}\label{eq:stIDinduc}
\int_{(\R^{\d+\k})^2}\nabla^{\otimes(n+1)}\G(x-y)  : ( v(x)-v(y))^{\otimes(n+1)}df(x)dw(y) = \int_{\R^{\d+\k}}\ka^{(n+1),f}d{w}.
\end{align}
 We now write this expression in a different way, which is amenable to application of \cref{lem:dtnrep}.

First, we write $\frac{d}{dt^{n+1}}=\frac{d}{dt^n}\frac{d}{dt}$ and observe
\begin{multline}
\frac{d}{dt} \int_{(\R^{\d+\k})^2} \G(x-y)df_t(x)dw_t(y) = \int_{\R^{\d+\k}} v(x)\cdot\nabla h^{w_t}(x+tv(x))df(x) \\
+ \int_{\R^\d} v(x)\cdot\nabla h^{f_t}(x+tv(x))dw(x),
\end{multline}
so that \eqref{eq:stIDinduc} equals
\begin{equation}
\frac{d^n}{dt^n}\Big|_{t=0}\Bigg(\int_{\R^{\d+\k}} v(x)\cdot h_t^{ w}(x+tv(x))df(x) + \int_{\R^{\d}} v(x)\cdot h_t^{ f}(x+tv(x))dw(x)\Bigg).
\end{equation}
By approximation, we may impose the qualitative assumption that $v$ is $C^\infty$. Now for any $x\in\R^{\d+\k}$, $\nabla(\I+t v)(x) = \I+t\nabla v(x)$. Since $\|\nabla v\|_{L^\infty}<\infty$, we see from the inverse function theorem that for $t$ sufficiently small depending only on $\|\nabla v\|_{L^\infty}$, the map $\I+t\nabla v(x):\R^{\d+\k}\rightarrow\R^{\d+\k}$ is invertible. Furthermore, there exists some $\ka>0$ such that
\begin{equation}
\inf_{\R^{\d+\k}}\det(\I+t\nabla v)\geq \ka \quad \text{and} \quad (\I+t\nabla v)^{-1} \in C^\infty(\R^{\d+\k}).
\end{equation}
Thus,
\begin{align}
\int_{\R^{\d+\k}} v(x)\cdot \nab h_t^{w}(x+tv(x))df(x) = \int_{\R^{\d+\k}}\Big(v\circ (\I+tv)^{-1}\Big)\cdot\nabla h_t^w d{f}_t, \\
\int_{\R^{\d+\k}} v(x)\cdot \nab h_t^{f}(x+tv(x))dw(x) = \int_{\R^{\d+\k}}\Big(v\circ (\I+tv)^{-1}\Big)\cdot\nabla h_t^fd{w}_t.
\end{align}
Writing $d{f}_t = -\frac{1}{\cd}\div(\nabla h_t^f)$ in the first line and $dw_t = -\frac{1}{\cd}\div(\nabla h_t^w)$ in the second line, we obtain from the identity \eqref{eq:commst} of \cref{lem:Lka1} that
\begin{multline}
\int_{\R^{\d+\k}} v(x)\cdot \nab h_t^{w}(x+tv(x))df(x) +\int_{\R^{\d+\k}} v(x)\cdot \nab h_t^{f}(x+tv(x))dw(x) \\
= \frac{1}{\cd} \int_{\R^{\d+\k}}\nabla\Big(v\circ (\I+tv)^{-1}\Big):  \comm{\nab h_t^w}{\nab h_t^f}. 
\end{multline}
Recalling our starting point, we have shown that
\begin{align}
&\int_{(\R^{\d+\k})^2}\nabla^{\otimes(n+1)}\G(x-y):\pa*{v(x)-v(y)}^{\otimes(n+1)}df(x)dw(y) \nn\\
&= \frac{1}{\cd}\frac{d^n}{dt^n}\Big|_{t=0}\int_{\R^{\d+\k}}\nabla\Big(v\circ (\I+t v)^{-1}\Big) : \comm{\nab h_t^w}{\nab h_t^f}. \label{eq:dtn+1preLeib}
\end{align}

Applying the Leibniz rule (and dropping the constant factor), the right-hand side of \eqref{eq:dtn+1preLeib} equals
\begin{equation}\label{eq:dtpsum}
\sum_{p=0}^n {n\choose p}\int_{\R^{\d+\k}} \frac{d^p}{dt^p}\Big|_{t=0}\nabla\Big(v\circ (\I+t v)^{-1}\Big) : \frac{d^{n-p}}{dt^{n-p}}\Big|_{t=0}\comm{\nab h_t^w}{\nab h_t^f}.
\end{equation}
For every $0\leq p\leq n$, we use the representation \eqref{eq:dtmrep} with $t=0$, $n$ replaced by $n-p$, and $\psi$ replaced by
\begin{equation}
\psi_{p}\coloneqq \frac{d^p}{dt^p}\Big|_{t=0}\nabla\Big(v\circ (\I+tv)^{-1}\Big)
\end{equation}
to find a family of functions $\phi_{ablm(n-p)}$ satisfying the estimates \eqref{eq:phidk} (with $n$ replaced by $n-p$) and such that
\begin{multline}\label{eq:dtprep}
\int_{\R^{\d+\k}} \frac{d^p}{dt^p}\Big|_{t=0}\nabla\Big(v\circ (\I+tv)^{-1}\Big): \frac{d^{n-p}}{dt^{n-p}}\Big|_{t=0}\comm{\nab h_t^w}{\nab h_t^f}\\
=\sum_{m=0}^{n-p}\sum_{a,b=1}^{\d+\k} \sum_{l=0}^m \int_{\R^{\d+\k}}\phi_{abl m(n-p)}\p_a\ka^{(l),f}\p_b\ka^{(m-l),w}.
\end{multline}
Using \cref{rem:phirep}, the integral over $\R^{\d+\k}$ in the right-hand side may be taken over $\supp\nabla v$.  This establishes \eqref{firstclaim}.

Suppose now that the estimate \eqref{eq:gradkam} holds for all $1\leq m\leq n$, for some $n\in\N$. We will show that \eqref{eq:gradkam} also holds for $m=n+1$. By Cauchy-Schwarz, we have
\begin{multline}
\Big|\int_{\R^{\d+\k}}\phi_{abl m(n-p)}\p_a\ka^{(l),f}\p_b\ka^{(m-l),w}\Big| \\
\leq \|\phi_{abl m(n-p)}\|_{L^\infty} \|\p_a\ka^{(l),f}\|_{L^2(\supp\nab v)} \|\p_b\ka^{(m-l),w}\|_{L^2(\supp\nab v)}.
\end{multline}
To control the $L^\infty$ norm factor, we argue as follows. Using Fa\`{a} di Bruno's formula and the Leibniz rule, one can show that for any $q\in\N$,
\begin{equation}\label{eq:dqdtp}
\left\|\nabla^{\otimes q} \frac{d^p}{dt^p}\Big|_{t=0}\nabla\pa*{v\circ (\I+tv)^{-1}}\right\|_{L^\infty} \leq C\sum_{\substack{0\leq c_0',\ldots,c_p'\leq q+p+1 \\ c_0'+\cdots+c_p'=q+p+1}} \|\nabla^{\otimes c_0'}v\|_{L^\infty}\cdots \|\nabla^{\otimes c_p'}v\|_{L^\infty}.
\end{equation}
It follows now from the property \eqref{eq:phidk}  for the functions $\phi_{abl m(n-p)}$ that for any $0\leq m\leq n-p$,
\begin{align}\label{eq:phinpest}
\|\phi_{ablm(n-p)}\|_{L^\infty} &\leq C\sum_{c_0+\cdots+c_{m}=n-p-m+k} \|\nabla^{\otimes c_0}{ \psi_p} \|_{L^\infty} \|\nabla^{\otimes c_1}v\|_{L^\infty}\cdots\|\nabla^{\otimes c_{m}}v\|_{L^\infty} \nn\\
&\leq C\sum_{c_0+\cdots+c_{m}=n-p-m}\sum_{c_0'+\cdots+c_p'=c_0+p} \Big(\|\nabla^{\otimes c_1}v\|_{L^\infty}\cdots\|\nabla^{\otimes c_{m}}v\|_{L^\infty} \nn\\
&\qquad \times\|\nabla^{\otimes c_0'+1}v\|_{L^\infty}\cdots \|\nabla^{\otimes c_p'}v\|_{L^\infty} \Big),
\end{align}
where to obtain the second inequality, we use \eqref{eq:dqdtp} with $q=c_0$. Combining \eqref{eq:dtpsum}, \eqref{eq:dtprep}, \eqref{eq:dqdtp} with \eqref{eq:phinpest}, we arrive at
\begin{multline}\label{eq:rhsinducsub}
\int_{(\R^{\d+\k})^2}\nabla^{\otimes(n+1)}\G(x-y):\pa*{v(x)-v(y)}^{\otimes(n+1)}df(x)dw(y) \\
\leq C\sum_{p=0}^n \sum_{m=0}^{n-p}\sum_{a,b=1}^{\d+\k} \sum_{l=0}^m \sum_{c_0+\cdots+c_{m}=n-p-m}\sum_{c_0'+\cdots+c_p'=c_0+p} \Big(\|\nabla^{\otimes c_1}v\|_{L^\infty}\cdots\|\nabla^{\otimes c_{m}}v\|_{L^\infty} \\
\times\|\nabla^{\otimes c_0'+1}v\|_{L^\infty}\cdots \|\nabla^{\otimes c_p'}v\|_{L^\infty} \Big) \|\p_a\ka^{(l),f}\|_{L^2(\supp\nab v)} \|\p_b\ka^{(m-l),w}\|_{L^2(\supp\nab v)}.
\end{multline}
We note that $\max(l,m-l) \leq n$ above. Therefore, we may apply the induction hypothesis to obtain that
\begin{align}
\|\nabla\ka^{(l),f}\|_{L^2} &\leq C\|\nabla h^f\|_{L^2(\supp\nabla v)}\sum_{d_1+\cdots+d_l=l} \|\nabla^{\otimes d_1}v\|_{L^\infty}\cdots\|\nabla^{\otimes d_l}v\|_{L^\infty},\\
\|\nabla\ka^{(m-l),w}\|_{L^2} &\leq C\|\nabla h^w\|_{L^2(\supp\nabla v)}\sum_{e_1+\cdots+e_{m-l}=m-l} \|\nabla^{\otimes e_1}v\|_{L^\infty}\cdots\|\nabla^{\otimes e_{m-l}}v\|_{L^\infty}.
\end{align}
Recalling our starting identity \eqref{eq:stIDinduc}, we see after substituting these bounds into the right-hand side of \eqref{eq:rhsinducsub}, making a change of index, and taking $C$ possibly larger that
\begin{multline}\label{eq:kan1}
\Big|\int_{\R^{\d+\k}}\ka^{(n+1),f}dw\Big| \leq C \|\nabla h^f\|_{L^2(\supp\nabla v)}\|\nabla h^w\|_{L^2(\supp\nabla v)}\\
\times\sum_{c_1+\cdots+c_{n+1} = n+1} \|\nabla^{c_1}v\|_{L^\infty}\cdots\|\nabla^{\otimes c_{n+1}}v\|_{L^\infty}.
\end{multline}

To deduce a bound for $\|\nabla\ka^{(n+1),f}\|_{L^2}$ from the estimate \eqref{eq:kan1}, we use a duality argument. Since $L\g = \cd\delta_0$ in $\R^{\d+\k}$, we have that\footnote{ {As mentioned in the introduction}, the property $\frac{1}{\cd}h^{Lw} = w$ is no longer true if $\s \ne \d+\k-2$ or $w$ is not supported in $\R^{\d}\times \{0\}^\k$, which is the case for $\ka^{(n+1),f}$. This is the reason why this proof is restricted to the special cases $\s=\d+\k-2$.} 
\begin{align}
\frac{1}{\cd} h^{L\ka^{(n+1),f}} = -\frac{1}{\cd}\G\ast\div(\nab\ka^{(n+1),f}) = \ka^{(n+1),f}.
\end{align}
Taking $w=\frac{1}{\cd}L\ka^{(n+1),f}$, we find from integration by parts that
\begin{align}
&\frac{1}{\cd}\int_{\R^{\d+\k}}|\nabla\ka^{(n+1),f}|^2 \nn\\
&= \int_{\R^{\d+\k}}\ka^{(n+1),f}w\nn\\
&\leq C\|\nabla h^f\|_{L^2(\supp\nabla v)}\|\nabla h^{w}\|_{L^2(\supp\nabla v)}\sum_{c_1+\cdots+c_{n+1} = n+1} \|\nabla^{c_1}v\|_{L^\infty}\cdots\|\nabla^{\otimes c_{n+1}}v\|_{L^\infty} \nn\\
&=C\|\nabla h^f\|_{L^2(\supp\nabla v)} \|\nabla\ka^{(n+1),f}\|_{L^2(\supp\nabla v)}\sum_{c_1+\cdots+c_{n+1} = n+1} \|\nabla^{c_1}v\|_{L^\infty}\cdots\|\nabla^{\otimes c_{n+1}}v\|_{L^\infty},
\end{align}
to obtain the final line. Dividing both sides by $\|\nabla\ka^{(n+1),f}\|_{L^2(\supp\nabla v)}$, we see that the proof of the lemma is complete.

\end{proof}

All of the ingredients for the proof of \cref{prop:commalt} are present. We now combine them to conclude this section.
\begin{proof}[Proof of \cref{prop:commalt}]
By \cref{lem:kainduc}, the desired estimate \eqref{eq:commalt} holds for any $n\in\N$, if it holds for $n=1$. By estimate \eqref{eq:comm} of \cref{prop:comm}, \eqref{eq:commalt} holds for $n=1$, therefore it holds for all $n\in\N$. The dualized estimate \eqref{eq:commaltdual} holds by virtue of \eqref{eq:kan1}. 
\end{proof}

\bibliographystyle{alpha}\bibliography{../../MASTER}
\end{document}